\documentclass[11pt]{amsart}
\usepackage{url}
\usepackage{hyperref,amssymb}
\usepackage{appendix}
\hypersetup{colorlinks=true,linkcolor=blue,
urlcolor=cyan,citecolor=magenta} 
\usepackage[hyperpageref]{backref}
\usepackage{frcursive,calrsfs,mathrsfs}
\usepackage{xcolor} 
\newcommand{\Q}{{\mathbb{Q}}}
\newcommand{\Z}{{\mathbb{Z}}}

\newcommand{\N}{{\mathbb{N}}}
\newcommand{\ds}{\displaystyle}
\newcommand{\sst}{\scriptstyle}
\newcommand{\ov}{\overline}
\newcommand{\wt}{\widetilde}

\newcommand{\ft}{\footnotesize}
\newcommand{\ns}{\normalsize}
\newcommand{\BN}{{\bf N}}

\newcommand{\BJ}{{\bf J}}

\newcommand{\CA}{{\mathcal A}}
\newcommand{\CC}{{\mathcal C}}
\newcommand{\CD}{{\mathcal D}}

\newcommand{\CH}{{\mathcal H}}

\newcommand{\CM}{{\mathcal M}}

\newcommand{\CP}{{\mathcal P}}
\newcommand{\CR}{{\mathcal R}}
\newcommand{\CT}{{\mathcal T}}
\newcommand{\CU}{{\mathcal U}}

\newcommand{\CW}{{\mathcal W}}

\newcommand{\CZ}{{\mathcal Z}}
\newcommand{\order}{\raise0.8pt \hbox{${\scriptstyle \#}$}}
\newcommand{\lien}{\mathrel{\mkern-4mu}}
\newcommand{\too}{\relbar\lien\rightarrow}
\newcommand{\tooo}{\relbar\lien\relbar\lien\too}
\newcommand{\toooo}{\relbar\lien\relbar\lien\tooo}
\newcommand{\tooooo}{\relbar\lien\relbar\lien\toooo}
\newcommand{\plus}{\ds\mathop{\raise 0.5pt \hbox{$\bigoplus$}}\limits}
\newcommand{\prd}{\ds\mathop{\hbox{$\prod$}}\limits}
\newcommand{\sm}{\ds\mathop{\raise 1.0pt \hbox{$\sum$}}\limits}
\newcommand{\ffrac}[2]{\hbox{\ft $\displaystyle\frac{#1}{#2}$}}
\newcommand{\limproj}{\mathop{\lim_{\longleftarrow}}}
\newcommand{\mb}{{\ {\sst \bullet}\ }}
\newcommand{\Gal}{{\rm Gal}}
\newcommand{\rk}{{\rm rank}}
\newcommand{\tor}{{\rm tor}}
\newcommand{\Ker}{{\rm Ker}}
\newcommand{\pr}{{\rm pr}}

\newcommand{\ram}{{\rm ram}}

\newcommand{\nr}{{\rm nr}}
\newcommand{\gen}{{\rm gen}}
\newcommand{\lc}{{\rm lc}}

\newcommand{\bp}{{\rm bp}}
\newcommand{\cyc}{{\rm c}}
\newcommand{\acyc}{{\rm ac}}
\newcommand{\hk}{{\rm h}_k}
\newcommand{\hp}{{\rm h}_{\mathfrak p}}
\newcommand{\he}{{\rm h}_e}

\newcommand{\Ccl}{c\hskip-0.1pt{\ell}}
\newcommand{\Bcl}{\hbox{$\rm c\hskip-0.1pt{{l}}$}}
\newcommand{\M}{{\rm M}}

\newcommand{\spp}{{\bf s}}
\newcommand{\re}{{\bf r}}
\newcommand{\0}{{\hbox{\tiny$0$}}}
\newtheorem{theorem}{Theorem}[section]
\newtheorem{lemma}[theorem]{Lemma}
\newtheorem{corollary}[theorem]{Corollary}
\newtheorem{proposition}[theorem]{Proposition}

\newtheorem{definitions}[theorem]{Definitions}

\newtheorem{remark}[theorem]{Remark}
\newtheorem{remarks}[theorem]{Remarks}
\newtheorem{conclusion}[theorem]{Conclusion}
\newtheorem{conjecture}[theorem]{Conjecture}
\numberwithin{equation}{section}
\usepackage[english]{babel}
\title[Chevalley--Herbrand formulas and $\Z_p$-extensions]
{Chevalley--Herbrand formulas and $\Z_p$-extensions \\
of a $p$-principal imaginary quadratic field}

\author{Georges Gras}

\address{Villa la Gardette, 4 chemin de Ch\^ateau Gagni\`ere, 
F-38520 Le Bourg d'Oisans}
\email{g.mn.gras@wanadoo.fr}

\urladdr{\url{http://orcid.org/0000-0002-1318-4414}}



\begin{document}

\date{August 19, 2026}

\begin{abstract}
Let $k$ be an imaginary quadratic field and let $p \ne 2$ be a prime number, split in $k$ into 
${\mathfrak p}{\ov {\mathfrak p}}$. We assume that the $p$-class group of $k$ is trivial. Let 
$\delta \geq 0$ be the ${\mathfrak p}$-valuation of the ${\ov {\mathfrak p}}$-Fermat quotient 
of the fundamental ${\mathfrak p}$-unit $x$ of $k$. Let $K/k$ be any bi-ramified $\Z_p$-extension 
and let $p^e$ be the degree of the inertia field of ${\mathfrak p}$, $\ov {\mathfrak p}$ being totally 
ramified. We prove that if $e \geq \delta$, then $\lambda(K/k) = 1$, $\mu(K/k) = 0$
(Theorem \ref{fundamental}); if $e<\delta$ a characterization is obtained from Iwasawa 
invariants of the $S^{\mathfrak p}$-class groups (Theorem \ref{mainbis}). 
This approach only uses generalizations of
Chevalley--Herbrand formulas and the non-nullity of a $p$-adic regulator $\CR^{\mathfrak p}_\delta$ 
in incomplete $p$-ramification. It provides effective and computable results that complement some 
aspects of Iwasawa theory. Conjecture \ref{conj} states that only the cyclotomic $\Z_p$-extension 
is ``exceptional''; justifications are given \S\,\ref{appliconj}. A {\sc pari/gp} program 
computes $\delta$ and $\CR^{\mathfrak p}_\delta$, for $p=3$, $e=1$.

\bigskip\noindent
{\sc R\'esum\'e.} 
Soit $k$ un corps quadratique imaginaire et soit $p \ne 2$ un nombre premier, d\'ecompos\'e dans 
$k$ en ${\mathfrak p}{\ov {\mathfrak p}}$. On suppose que le $p$-groupe des classes de $k$ est 
trivial. Soit $\delta \geq 0$ la ${\mathfrak p}$-valuation du ${\ov {\mathfrak p}}$-quotient de Fermat 
de la ${\mathfrak p}$-unit\'e fondamentale $x$ de $k$. Soit $K/k$ une $\Z_p$-extension 
bi-ramifi\'ee et soit $p^e$ le degr\'e du corps d'inertie de ${\mathfrak p}$, $\ov {\mathfrak p}$ 
\'etant totalement ramifi\'e. Nous prouvons que si $e \geq \delta$, alors $\lambda(K/k) = 1$, 
$\mu(K/k) = 0$ (Th\'eor\`eme \ref{fundamental}); si $e < \delta$, une caract\'erisation est 
obtenue \`a partir des invariants d'Iwasawa des $S^{\mathfrak p}$-class groups 
(Th\'eor\`eme \ref{mainbis}). Cette approche n'utilise que des g\'en\'eralisations 
des formules de Chevalley--Herbrand et la non nullit\'e d'un r\'egulateur $p$-adique 
$\CR^{\mathfrak p}_\delta$ en $p$-ramification incompl\`ete. Elle donne des r\'esultats 
effectifs et calculables compl\'etant certains aspects de la th\'eorie d'Iwasawa.
La Conjecture \ref{conj} affirme que seule la $\Z_p$-extension cyclotomique est ``exceptionnelle;
des justifications sont donn\'ees au \S\,\ref{appliconj}. Un programme 
{\sc pari/gp} calcule $\delta$ et $\CR^{\mathfrak p}_\delta$ pour $p=3$, $e = 1$.
\end{abstract}

\maketitle

\tableofcontents

\section{Introduction -- Main results}

Let $k$ be an imaginary quadratic field and let $p$ be an odd prime number split 
in $k$; we denote by $\tau$ the complex conjugation and set $(p) = {\mathfrak p}
{\ov {\mathfrak p}}$, where ${\ov {\mathfrak p}} := {\mathfrak p}^\tau$.

\smallskip
In \cite[Theorem 1]{Oza2001}, Ozaki states:

\smallskip
{\it Let $k$ be an imaginary quadratic field and $p \geq 2$ a prime number. 
Assume that the prime $p$ splits in $k$ and the class number of $k$ is prime 
to $p$. Then $\lambda(K/k) = 1$ and $\mu(K / k) = 0$ for all but finitely
many $\Z_p$-extensions $K$ over $k$.}

\smallskip
We intend to analyze, in a practical and effective point of view, this result.
Let $\wt k$ be the compositum of the $\Z_p$-extensions of $k$.
Let $K/k$ be a $\Z_p$-extension
in which ${\mathfrak p}$, ${\ov {\mathfrak p}}$ ramify from some layers, 
which excludes the conjugate $\Z_p$-extensions $L$, $\ov L$, defined 
as the inertia fields, in $\wt k/k$, of ${\mathfrak p}$, ${\ov{\mathfrak p}}$;
indeed, $L/k$ is unramified at ${\mathfrak p}$, and ${\ov{\mathfrak p}}$
is totally ramified in $L/k$ since the $p$-class group $\CH_k$ 
of $k$ is trivial; hence, $L \cap \ov L = k$ and $L \ov L = \wt k$.

\smallskip
Using the original Chevalley--Herbrand formula \cite[pp. 402--406]
{Che1933} in $L_n/k$ gives $\lambda(L/k) = \mu(L/k) =\nu(L/k) = 0$ 
and similarly for $\ov L$. But $L, \ov L$ play an important role.

\smallskip
Denote by $K_n$ the layer of $K$ of degree $p^n$ over $k$.
Since $\CH_k = 1$, the $\Z_p$-extension $K/k$
is totally ramified at one of the two primes ${\mathfrak p}$, 
${\ov {\mathfrak p}}$. If $K/k$ is Galois, it is totally ramified at $p$.
If $K/k$ is non-Galois, let's say, for example, that ${\ov {\mathfrak p}}$ 
is totally ramified and that ${\mathfrak p}$ ramifies from $K_e$, 
$e \geq 0$; since for the complex conjugate $\ov K$, ${\mathfrak p}$ is 
totally ramified and $\ov{\mathfrak p}$ is ramified from $\ov K_e$, one 
may choose the conjugate $K$ such that $\ov e = 0$ and $e \geq 0$.

\smallskip
We have obtained the following result, without assuming that the pro-$p$-class 
group $\CH_{\,\wt k}$ of $\wt k$ is a pseudo-null $\Lambda$-module, 
$\Lambda := \Z_p[[\Gal(\wt k/k)]]$ (Minardi \cite{Min1986}), nor using any 
reasoning in Iwasawa's theory. Indeed, a systematic use of $p$-adic class 
field theory is sufficient, especielly under the form of classical genus theory, 
via suitable Chevalley--Herbrand formulas that allows more explicit results; 
that said, all the approaches in the literature depend on such class field theory 
techniques, often in a more or less explicit form.

\medskip
\noindent 
{\bf Main Results} (Theorems \ref{fundamental}, \ref{pseudo-null}, \ref{mainbis}). 
{\it Let $k$ be a $p$-principal\,\footnote{An ideal ${\mathfrak a}$
of $F$ is $p$-principal if there exists $h \not \equiv 0 \pmod p$ such that
${\mathfrak a}^h = (\alpha)$, $\alpha \in F^\times$; the field $F$ is $p$-principal 
if its class number is prime to $p$. By abuse we will sometimes write ${\mathfrak a} 
=: (\alpha)$, instead of ${\mathfrak a}^h = (\alpha)$.}
imaginary quadratic field and let $p \geq 3$ be a prime 
number split in $k$. Let $K/k$ be a $\Z_p$-extension, $K \ne L, \ov L$. Let $p^e := 
[K \cap L : k]$, $K \cap \ov L = k$. Let $p^\delta$ be the degree of the decomposition 
field of ${\mathfrak p}$ in $L/k$ (we refer to \cite[Appendix A]{Gra2026b} for the fact that 
$\delta$ is also the $p$-valuation of the order of the logarithmic class group $\wt \CH_k$ 
and the $\ov {\mathfrak p}$-valuation of the ${\mathfrak p}$-Fermat quotient of the 
fundamental ${\mathfrak p}$-unit $x$ of $k$ (i.e. $\delta = v_{\ov {\mathfrak p}}(x^{p-1} - 1) - 1$), 
cf. \S\,\ref{logx}). Put $S_n^{\mathfrak p} := \{{\mathfrak p}_n\! \mid\! {\mathfrak p} \, 
\hbox{in $K_n/k$}\}$. Then:

\smallskip
(i) If $e \geq \delta$, $\lambda(K/k) = 1$, $\mu(K/k) = 0$, $\nu(K/k) = t_{\delta'} - e$,
where $p^{t_{\delta'}}$ is the order of a $p$-torsion group $\CT^{\mathfrak p}_{\delta'}$,
only depending on $K_{\delta'} = L_{\delta'}$, where $\delta' = \min(e, \delta)$ (see 
Theorem \ref{Ptorsion}).

\smallskip
(ii) Let's say that Iwasawa's invariants of $K/k$ are minimal if $\lambda(K/k) = 1$, 
$\mu(K/k) = 0$. When $\lambda(K/k) > 1$ or $\mu(K/k) > 0$, we will say, following 
\cite{Oza2001}, that $K/k$ is exceptional.
Then, if $e < \delta$, Iwasawa' invariants are minimal if and only if the
Iwasawa invariants corresponding to the $S_n^{\mathfrak p}$-class groups 
in $K/k$ are trivial ($\lambda^{\mathfrak p}(K/k) = \mu^{\mathfrak p}(K/k) = 0$).}

\medskip
This limit, given by $\delta$, defines a precise and computable arithmetic 
invariant that elucidates certain results of Greenberg \cite{Gree1973}, then 
generalized by Kleine \cite{Kle2014, Kle2021}, among others.
We deduce that exceptional $\Z_p$-extensions $K/k$ are such that 
$e < \delta$; the $\Z_p$-extensions with $e < \delta$ are infinite in 
number (except if $\delta = 0$, a particular case of $e \geq \delta$
yet obtained in \cite[Theorems 10.7, 10.8]{Gra2026b}). This result 
shows that the Jaulent logarithmic class group $\wt \CH_k$
(\cite{Jau1994, Jau2026}, \cite[Section A, Theorem A.2 by 
Jaulent]{Gra2026b}) is an important obstruction in Iwasawa's 
theory since, in our context, $\order \wt \CH_k = p^\delta$. 

\smallskip
We examine, Section \ref{e<delta}, the case $e < \delta$, to understand 
the specific properties which allow, for instance, large $\lambda(k^\cyc/k)$'s
for the cyclotomic $\Z_p$-extensions $k^\cyc$ of $k$ (for which $e = 0$). 

\smallskip
Note that the minimality of Iwasawa's invariants is equivalent to the 
finiteness of $H_K^\nr/\wt k$, where $H_K^\nr$ is the pro-$p$-Hilbert class 
field of $K$; this will be obtained by proving that the group of co-invariants 
$\big(\CH_{\,\wt k} \big)_{\Gal(K/k)}^{}$ is finite, since $H_K^\nr$ is the genus 
field of $\wt k/K$ (see Diagram \ref{diagram1}).

\section{Generalities on the \texorpdfstring{$\Z_p$-extensions of $k$}{Lg}}

Recall, the parametrization of the set of the $\Z_p$-extensions $K/k$ 
carried out in \cite[Section B, Theorem B.1]{Gra2026b}; we restrict
ourselves to the case $p$ split in $k$ and $\CH_k = 1$:

\subsection{Parametrization of the \texorpdfstring{$\Z_p$-extensions 
$K/k$, $K \ne L, \ov L$}{Lg}}

The group $\langle \tau \rangle$ operates on $\Gamma := \Gal(\wt k/k)$ 
and we define $\Gamma^+ := \Gamma^{\frac{1+\tau}{2}}$ 
(resp. $\Gamma^- := \Gamma^{\frac{1-\tau}{2}}$). Thus, the cyclotomic 
$\Z_p$-extension $k^\cyc := k \,\Q^\cyc$ is fixed by $\Gamma^-$ and
the anti-cyclotomic $\Z_p$-extension $k^{\acyc}$ is fixed by $\Gamma^+$; 
$k^\cyc$ and $k^\acyc$ are totally ramified at $p$ since $\CH_k = 1$. 

\smallskip
Let $\wt\sigma_-$ and $\wt\sigma_+$ be topological generators of 
$\Gamma^- = \Gal(\wt k/k^\cyc)$ and $\Gamma^+ = \Gal(\wt k/k^\acyc)$, 
respectively. Any subgroup $G$ of $\Gamma$, $G \ne 1, \Gamma$, is free 
of $\Z_p$-rank $1$. We define any $\Z_p$-extension $K/k$ by means of a
generator $\wt\sigma_K^{}$ of $G_K := \Gal(\wt k/K)$ so that $\Gamma/G_K \simeq \Z_p$. 
Put $\wt\sigma_K^{} = \wt\sigma_-^{p^\alpha u} \! \cdot \wt\sigma_+^{p^\beta v}$, 
$u, v \in \Z_p^\times$ with, necessarily, $\alpha \cdot \beta = 0$. With suitable 
choices of the generators $\wt\sigma_-, \wt\sigma_+$, one may suppose that $L$ is
defined by $\wt\sigma_L^{} = \wt\sigma_- \!\cdot \wt\sigma_+$ (whence $\ov L$ defined by
$\wt\sigma_{\ov L}^{} = \wt\sigma_-^{-1}\! \cdot \wt\sigma_+$), which fixes $\wt\sigma_-$ and 
$\wt\sigma_+$; to fix $\sigma_K^{}$, we may suppose $v=1$; we have 
$K \cap k^\cyc = k^\cyc_\beta$ and $K \cap k^\acyc = k^\acyc_\alpha$:

\begin{proposition}\label{3cases}
Let $K$ be fixed by $G_K = \langle \wt\sigma_-^{p^\alpha u} \! 
\cdot \wt\sigma_+^{p^\beta v} \rangle$, with $\alpha \cdot \beta = 0$, 
$u, v \in \Z_p^\times$. If $\CH_k = 1$, this describes the 
$\Z_p$-extensions $K \ne L, \ov L$, and implies the following 
values of $e$, $\ov e$:
\begin{equation*}
\left \{\begin{aligned} 
&\hbox{$(i) \ \ $ if $K \cap k^\cyc \ne k$ 
($\alpha = 0$, $\beta \ne 0$, $u, v \in \Z_p^\times$)}, 
&& \hbox{$e = \ov e = 0$}, \\
&\hbox{$(ii)\ $  if $K \cap k^\acyc \ne k$ ($\alpha \ne 0$, $\beta = 0$, $u, v \in \Z_p^\times$)},
&&\hbox{$e = \ov e = 0$}, \\
&\hbox{$(iii)$ 
 if $K \cap k^\cyc = K \cap k^\acyc = k$ ($\alpha = \beta = 0$, $u, v \ne \pm 1$)}, 
&&\hbox{$e = v_p(u - 1)$, $\ov e = v_p(u + 1)$}.
\end{aligned}\right .
\end{equation*}
\end{proposition}

In case (iii), we have chosen $\ov e = 0$, then $u \not \equiv -1 \pmod p$.
The case $e = \ov e = 0$ is possible with case (iii), except for $p=3$.
It is easy to see that $K \cap L = K_e$ and $K \cap \ov L  = K_{\ov e}= k$.

\begin{remarks} 
(i) In general, we do not assume any relation between $\delta$ and $e$ so that 
many results use $\delta' = \min(e, \delta)$. Indeed, in the tower $K/k$, 
only the layers $K_e$ and $K_{\delta'}$ make sense; when $\delta > e$, 
$K_\delta$ does exist, but has no interest since inertia fields of ${\mathfrak p}$ 
coincide with $K_e = L_e$, and decomposition fields with $K_{\delta'} = L_{\delta'}$
(see \S\,\ref{ramdec}). The fact that many properties of $K$ only depend 
on $L_e$, $L_{\delta'}$ is remarkable since this concerns infinitely 
many $K$'s. Diagram \ref{Tpsplit2} gives the two fundamental cases.

\smallskip
(ii) We intend to examine the existence of non-Galois
exceptional $\Z_p$-extensions (if any) and if they can be described 
by means of a class field theory approach. As it is well-known, the 
cyclotomic $\Z_p$-extension $k^\cyc$ may have large invariant
$\lambda$ and is then exceptional, but the condition $e \geq \delta$ 
is not satisfied for $\delta \ne 0$ since $e = 0$ (see \S\,\ref{cyclo}
for numerical illustration). The anti-cyclotomic $\Z_p$-extension is 
also totally ramified, but its Iwasawa's invariants are, a priori, unknown.
\end{remarks}

\subsection{Logarithms -- Computation of \texorpdfstring
{$\delta$}{Lg}} \label{logx}
Let $\log_p := (\log_{\mathfrak p}, \log_{\ov{\mathfrak p}}$),
be the $p$-adic $\log$-function with values in the $p$-completions 
$k_{\mathfrak p}$, $k_{\ov {\mathfrak p}}$ of $k$
(isomorphic to $\Q_p$), {\it with the convention 
$\log_{\mathfrak p}(p) = \log_{\ov{\mathfrak p}}(p) = 0$}.
Let $\hk$ (prime to $p$) be the class number of $k$, and let 
$x \in k^\times$ be the fundamental ${\mathfrak p}$-unit of $k$ 
(obtained, up to sign, from ${\mathfrak p}^{\hp} =: (x)$, $\hp \mid \hk$ 
being the order of the classes of ${\mathfrak p}$ and $\ov {\mathfrak p}$).
Since $x \ov x = p^{\hp}$, this defines in the 
$\Z_p[\langle \tau \rangle]$-algebra $k_{\mathfrak p} \oplus k_{\ov{\mathfrak p}}$:
\begin{equation*}
\left \{\begin{aligned}
& \log_{\mathfrak p}(\ov x) \ \hbox{(as usual since 
$\ov x$ is a unit in $k_{\mathfrak p}^\times$)}, \ \, 
 \log_{\ov{\mathfrak p}}(x) \ \hbox{(conjugate 
of $\log_{\mathfrak p}(\ov x)$)}, \\
& \log_{\mathfrak p}(x) = - \log_{\mathfrak p}(\ov x), \ \,
\log_{\ov {\mathfrak p}}(\ov x)  = - \log_{\ov{\mathfrak p}}(x)\, 
\hbox{(from $\log_p(p)=0$)}.
\end{aligned}\right.
\end{equation*}

This defines, without any ambiguity, in terms of $p$-adic valuations 
$v_p$ on $\Q_p^\times$, or $v_{\mathfrak p}$ on $k^\times$:
$$\delta := v_p \big (\hbox{$\frac{1}{p}$} \log_{\mathfrak p}(\ov x) \big)
= v_p \big (\hbox{$\frac{1}{p}$} \log_{\ov {\mathfrak p}}(x) \big) 
= v_{\mathfrak p}(\ov x^{\,p-1}-1) - 1. $$

\section{Abelian \texorpdfstring{$p$}{Lg}-ramification theory -- Norm 
residue symbols} \label{schema0}

\subsection{Diagram of the abelian \texorpdfstring{$p$}{Lg}-ramification}

Let's consider the following {\it general} diagram of abelian $p$-ramification 
theory over any number field $F$, under the Leopoldt conjecture:
\unitlength=1.0cm 
\begin{equation*}
\begin{aligned}
\vbox{\hbox{\hspace{-4.0cm} 
\begin{picture}(10.0,3.3)
\put(6.6,2.50){\line(1,0){1.45}}
\put(9.0,2.50){\line(1,0){1.8}}
\put(3.7,2.50){\line(1,0){1.8}}
\put(3.9,2.25){\tiny\hbox{${\rm unramified}$}}
\put(9.35,2.1){\ft$\CW_F^\bp$}
\put(4.2,0.50){\line(1,0){1.25}} 
\put(1.6,0.50){\line(1,0){1.3}}
\put(1.3,0.40){$F$}
\bezier{300}(1.38,0.34)(3.5,0.0)(5.6,0.34)
\put(3.2,-0.12){\ft$\CH_F$}
\bezier{400}(3.7,2.7)(7.25,3.3)(10.8,2.7)
\put(7.0,3.14){\ft$\CT_F$}
\put(3.5,0.9){\line(0,1){1.25}}
\put(5.7,0.9){\line(0,1){1.25}}
\bezier{350}(6.2,0.5)(8.5,0.6)(10.9,2.3)
\put(8.6,0.75){\ft$\CU_F\!/\! \ov E_F$}
\put(10.85,2.4){$H_F^\pr$}
\put(5.6,2.4){$\wt F H_F^\nr$}
\put(8.15,2.4){${H_F^\bp}$}
\put(7.0,2.15){\ft$\CR_F$}
\put(3.35,2.4){$\wt F$}
\put(5.5,0.4){$H_F^\nr$}
\put(2.95,0.4){$\wt F \!\cap\! H_F^\nr$}
\end{picture}}} 
\end{aligned}
\end{equation*}
\unitlength=1.0cm

Recall that $\wt F$ is the compositum of the $\Z_p$-extensions of $F$,
$H_F^\pr$ is the maximal abelian $p$-ramified pro-$p$ extension 
of $F$, $\CT_F = \Gal(H_F^\pr/\wt F)$ is finite, 
$H_F^\nr$ is the $p$-Hilbert class field of $F$, $H_F^\bp$ 
the Bertrandias--Payan field, then $\CR_F \simeq 
\Gal(H_F^\bp/\wt F H_F^\nr)$ is the normalized $p$-adic regulator 
and $\CW_F^\bp = \bigoplus_v \mu_p(F_v) \big / \mu_p(F) \simeq 
\Gal(H_F^\pr/H_F^\bp)$.

\smallskip
The group of principal local units $\CU_F$, whose image gives inertia groups in 
$H_F^\pr/H_F^\nr$, is of the form $\CU_F = \bigoplus_{v \mid p}\CU_F^v$, where 
$\CU_F^v := 1 + {\mathfrak p}_v$ (${\mathfrak p}_v =$ prime ideal associated to the 
place $v$). Then $\CU_F/\ov E_F \simeq \Gal(H_F^\pr/H_F^\nr)$, where $\ov E_F$ 
is the topological closure, in $\CU_F$, of the diagonal embedding $\iota_p(E_F)$
of the group of $p$-principal global units of $F$.

\smallskip
In our context, where $F$ is a $p$-principal
imaginary quadratic field $k$, $\CR_k$ is trivial since $E_k \otimes \Z_p = 1$, 
$\ov E_k = 1$, the subgroup $\CW_k^\bp$ is trivial in the split case of $p \geq 3$
in $k$, whence $\CT_k = 1$, $H_k^\pr = \wt k$, since $\CH_k = 1$. Then $\CU_k = 
\CU_k^{\mathfrak p} \oplus \CU_k^{\ov {\mathfrak p}}  \simeq (1+p\Z_p)
\oplus (1+p\Z_p)$.

\smallskip
In other words, the field $k$ is $p$-rational, giving the simplest arithmetic, as
shown by the following well-known property:

\begin{proposition}
Let $k$ be an imaginary quadratic field, and let $p \geq 3$ be a prime number. 
Assume that $-m \not\equiv -3 \pmod 9$ and that $\CH_k = 1$. Then, $k$ is $p$-rational 
and any $p$-ramified $p$-extension $M$ of $k$ is $p$-rational; then 
$\CT_M = 1$ and $\CH_M \subset \wt M$.
\end{proposition}

\begin{proof}
More generally, let $F$ be a $p$-rational field and let $M/F$ be a 
$p$-ramified $p$-extension. Let $G := \Gal(M/F)$. 
We have $\CT_M^G  \simeq \CT_F = 1$, since $M/F$ is $p$-ramified 
(see for instance \cite[Theorem IV.3.3]{Gra2005}; this has been often proven:
\cite{Gra1986} (from \cite{Gra1982, Gra1983}), \cite{Jau1986}, 
\cite{NQD1986}, \cite{Mov1988}), \cite[Theorem,\,($\beta$),\,p.\,344]{GJ1989}. 
In the same context, one gets that in any extension $M/M'/F$ the transfer 
map $\CT_{M'} \to \CT_M$ is injective.
\end{proof}

\subsection{Where does each \texorpdfstring{$p$-place split in 
$\wt k/k$ ?}{Lg}}\label{ramdec}

Our assumptions imply that the field $H_k^\lc$, maximal locally 
cyclotomic pro-$p$-extension of $k$ (i.e. such that $p$ totally splits
in $H_k^\lc/k^\cyc$) is contained in $H_k^\pr = \wt k$ and $[H_k^\lc : k^\cyc] 
= p^\delta$ since the logarithmic class group $\wt \CH_k$ is 
then of order $p^\delta$ (recall that for any imaginary quadratic field
in which $p$ splits, $\order \wt \CH_k = p^\delta \!\cdot \order \CH_k^{S_k}$,
where $\CH_k^{S_k}$ is the $S_k$-class group of $k$, $S_k := \{{\mathfrak p},
\ov {\mathfrak p}\}$).

\smallskip
The above means that, in $\wt k/k^\cyc/k$, $p$ is totally ramified in 
$k^\cyc/k$, that the two decomposition groups of ${\mathfrak p}$ and 
${\ov {\mathfrak p}}$ in the unramified pro-$p$-cyclic extension $\wt k/k^\cyc$, 
coincide and fix the decomposition field $H_k^\lc$ of degree 
$p^\delta$ over $k^\cyc$.

\smallskip
In $\wt k/k^\acyc/k$, $p$ is totally ramified in $k^\acyc/k$ and the two 
decomposition groups in the unramified pro-$p$-cyclic extension $\wt k/k^\acyc$ 
are of degree $p^\delta$ over $k^\acyc$ since no splitting is possible 
elsewhere. But some splitting does exist in any extension $\wt k/K/k$ as follows:

\smallskip
The inertia group of ${\mathfrak p}$ (resp. of $\ov {\mathfrak p}$) in $\wt k/k$  
being $\Gal(\wt k/L)$ (resp. $\Gal(\wt k/\ov L)$), the decomposition group of 
${\mathfrak p}$ in $L/k$ (resp. of $\ov {\mathfrak p}$ in $\ov L/k$) fixes the layer 
$L_\delta$ (resp. the layer $\ov L_\delta$). This phenomenon has repercussions 
on the splitting of $p$ in the non-Galois $\Z_p$-extensions $K/k$. 

\smallskip
More precisely, since $K$ is fixed such that
$K \cap \ov L = K_{\ov e} = k$, $\ov {\mathfrak p}$ does not split in $K$;
then $K \cap L  = K_e$ and the decomposition field of ${\mathfrak p}$ in $K/k$
is of degree $p^{\delta'}$, $\delta' := \min(e, \delta)$, over $k$. If $e < \delta$ the 
complementary part of the splitting of ${\mathfrak p}$ takes place over $K$ in $\wt k/K$. 

\smallskip
Since $K/k$, $\ov K/k$ are totally ramified at $p$ from layers $K_e$, $\ov K_e$, 
the pro-$p$-cyclic extensions $\wt k/K$ and $\wt k/\ov K$ are unramified.
The splitting of ${\mathfrak p}$ and $\ov {\mathfrak p}$ in $\wt k/k$ does imply 
immediately the inertia of ${\mathfrak p}$ in $K/k$, which is finite (residue degree 
of ${\mathfrak p}$ in $K/k$ equal to $p^{\max(0, e - \delta)}$). 

\smallskip
To get numerical illustrations, in the anti-cyclotomic $\Z_p$-extensions $k^\acyc$, 
we refer to the {\sc pari/gp} programs of \cite{Gra2026a}, valid without any 
assumptions on $\CH_k$, $e$, $\delta$.

\subsection{Functorial properties of the norm residue symbols}
Recall for convenience these classical properties in the case of abelian 
$p$-extensions and $p$-places (see for instance \cite[Theorem II.3.1.3]
{Gra2005} or any book on class field theory):

\begin{lemma}\label{functorial}
Let $M/F$ be a finite abelian extension of number fields, $M'/F$ a sub-extension, 
and let $w \mid w' \mid v$, be $p$-places of $M$, $M'$, $F$, then $\CU_M^w$, 
$\CU_{M'}^{w'}$, $\CU_F^v$ the unit groups of the completions $M_w$, 
$M'_{w'}$, $F_v$, respectively.
We identify $D_v(M/F) := \Gal(M_w/F_v)$ with the decomposition group of $v$ 
in $M/F$ and its inertia group with $I_v(M/F) \subseteq D_v(M/F)$. 

\smallskip
Then there exists a canonical homomorphism
$(\mb , M_w/F_v) :  F_v^\times  \to  D_v(M/F)$. 
We define the Hasse norm residue symbols $\big(\frac{x, M/F}{v} \big)$, 
$x \in F^\times$, as the image of $(x, M_{w}/F_v)$ in $D_v(M/F)$. 
These symbols have the following properties:

\smallskip
{\it (i) We have the exact sequence defining the group of local norms:
$$1 \too F^\times \cap \BN_{M_w/F_v}(M_w^\times) \tooo F^\times\, 
\mathop{\toooo}^{\hbox{\tiny$\big(\frac{\mb, M/ F}{v} \big)$}} D_v(M/F) \too 1. $$

(ii) The composition of $\big(\frac{\mb, M/ F}{v} \big)$ with
$\Gal(M/F) \too \Gal(M'/F)$ is $\big(\frac{\mb, M'/ F}{v} \big)$.

\smallskip
(iii) The image of $F^\times \cap \CU_F^v$ under $\big(\frac{\mb, M/ F}{v} \big)$ 
is the inertia group $I_v(M/F)$. 

\smallskip
(iv) For $x' \in M'^\times$, the image of $\prod_{w' \mid v}
\big(\frac{x', M/ M'}{w'} \big)$ in $\Gal(M/F)$ is $\big(\frac{\BN_{M'/F}(x'), M/F}{v}\big)$.

\smallskip
(v) For $x \in F^\times$, the image of $\big(\frac{x, M/ F}{v} \big)$ under the
transfer map (from $\Gal(M/F)$ to $\Gal(M/M')$) is 
$\prod_{w' \mid v}\big(\frac{x, M/ M'}{w'} \big)$.

\smallskip
(vi) For $x \in F^\times$, $\!\big(\frac{x, M/ F}{v} \big)^s\! :=
s \,\hbox{\tiny$\circ$} \,\big(\frac{x, M/ F}{v} \big)\, \hbox{\tiny$\circ$} \,s^{-1}\! =
\!\big(\frac{x^s, s M / s F}{s v})$, for any isomorphism $s$ of $M$.}
\end{lemma}

\section{Computation of some norm residue symbols}

We will particularize the $p$-extension $M/F$ as follows in our context of
$p$-principal imaginary quadratic field $k$.

\subsection{Some notations and definitions}

Let $K/k$ be a $\Z_p$-extension, distinct from $L$, $\ov L$, and let 
$K_n$, $n \geq 0$, be the subfield of $K$ of degree $p^n$ over $k$. 
Let $e \geq 0$ defining the inertia subfield 
$K_e = K \cap L$ for the prime ${\mathfrak p}$, that of $\ov {\mathfrak p}$ 
being $k$. Let $\delta' := \min(e, \delta)$. 

\smallskip
Let $\BN_{a/b}$ be the arithmetic norm in $K_a/K_b$ for any $0 \leq b \leq a$, 
and so on for any invariant depending on extensions $K_a/K_b$; whence
$g_{n/e} := \Gal(K_n/K_e)$, $g_{n/\delta'} := \Gal(K_n/K_{\delta'})$, and for 
short, we put $G_n := g_{n/0} = \Gal(K_n/k)$. Recall that $G_K := \Gal(\wt k/K)$.

\smallskip
Since we are interested by the growth of $\order \CH_n$ when $n \to \infty$, 
we will assume $n$ large enough.

\smallskip
For any set of places $T_n$ and the group of units $E_n$, of a layer $K_n$, 
we denote by $\iota_{T_n}^{}$ the diagonal embedding of $E_n$ in the group of principal 
local units $\CU_n^{T_n} := \bigoplus_{v \in T_n} \CU_n^v$, where $\CU_n^v := 
1 + {\mathfrak p}_v$; the topological closure of $\iota_{T_n}^{}(E_n)$ 
in $\CU_n^{T_n}$ is denoted $\ov E_n^{T_n}$. 

\smallskip
Let $\spp := p^{\delta'}$.
We put by abuse $S_{K_{\delta'}}^{\mathfrak p} =: S_{\delta'}^{\mathfrak p} :=
\{{\mathfrak p}_{\delta', 1}^{}, \ldots, {\mathfrak p}_{\delta', \spp}^{} \}$, 
where ${\mathfrak p}_{\delta', i}^{} \mid {\mathfrak p}$ in $K_{\delta'}$,
and $S_{\delta'}^{\ov{\mathfrak p}} := \{{\mathfrak p}_{\delta', 0}^{} \}$, 
${\mathfrak p}_{\delta',0}^{}$ being the unique prime ideal above 
$\ov {\mathfrak p}$ in $K_{\delta'}$; then we put $S_{\delta'} := 
S_{\delta'}^{\mathfrak p} \cup S_{\delta'}^{\ov{\mathfrak p}}$.

\smallskip
In what follows, we will simplify and write instead
$S_{\delta'} := \{{\mathfrak p}_1, \ldots, {\mathfrak p}_{\spp}, {\mathfrak p}_\0 \}$ 
for the layer $K_{\delta'}$ and similarly for 
$S_{\delta'}^{\mathfrak p}$ and $S_{\delta'}^{\ov{\mathfrak p}}$.
By extension in $K_{\delta'}$, $({\mathfrak p})_{\delta'} = \prod_{i=1}^\spp
{\mathfrak p}_i$ and $(\ov {\mathfrak p})_{\delta'} = {\mathfrak p}_\0^\spp$.

\smallskip
The inertia groups of these ideals ${\mathfrak p}_i$ in $K_n/K_{\delta'}$,
$i \in [1, \spp]$, are equal to $g_{n/e}$ of order $p^{n - e}$, 
and the decomposition groups are equal to $g_{n/{\delta'}}$
of order $p^{n - \delta'}$ (note that $\delta' \leq e$ by definition). 

\smallskip
For local units, the notations $\CU_{\delta'}^{\mathfrak p}$ and
$\CU_{\delta'}^{{\mathfrak p}_\0}$ mean, by abuse, $\CU_{\delta'}^{T_{\delta'}}$ for 
$T_{\delta'} = \{{\mathfrak p}_1, \ldots, {\mathfrak p}_{\spp}\}$ (prime ideals dividing 
${\mathfrak p}$ in $K_{\delta'}$) and $T_{\delta'} = \{{\mathfrak p}_\0\}$ (prime ideal 
dividing $\ov {\mathfrak p}$), respectively.

\subsection{On the use of Chevalley--Herbrand formulas}
The most classical application of Chevalley--Herbrand formula in $K_n/k$ (hence the
computation of $\order \CH_n^{G_n}$ equal to $p^{n-e}$), gives minor information on 
Iwasawa's invariants of $K \ne L, \ov L$. 

\smallskip
However, we have, in full generality for $0 \leq n \leq m$, the exact sequences:
\begin{equation}\label{gmn}
1 \to \BJ_{m/n} (\CH_n) \CH_m^\ram \too \CH_m^{g_{m/n}} 
\too E_n \cap \BN_{m/n}(K_m^\times)/\BN_{m/n}(E_m) \to 1,
\end{equation}

\noindent
where $ \BJ_{m/n}$ is the transfer map and $\CH_m^\ram$ the 
subgroup of $\CH_m$ generated by the invariant ideals of
$S_m := \langle {\mathfrak p}_{m,i} \rangle_{{\mathfrak p}_{m,i} \mid p}^{}$, 
$i \in [1,\ldots , \spp] \cup \{0\}$, where $\spp = p^{\delta'}$, whence suitable 
invariants products of these ideals. The main obstruction being the unit term 
on the right or, equivalently, the denominator in the Chevalley--Herbrand 
formula (for $m \geq n \geq \delta'$):
$$\order \CH_m^{g_{m/n}} = \frac{\order \CH_n \cdot [p^{m-n}]^\spp
\cdot p^{m-n}} {p^{m-n} \cdot (E_n : E_n \cap \BN_{m/n}(K_m^\times))} =
\frac{\order \CH_n \cdot [p^{m-n}]^{\spp}}{(E_n : E_n \cap \BN_{m/n}(K_m^\times))}. $$

So, a trick is to use the classical generalizations of
Chevalley--Herbrand formulas for the $T_n$-class groups $\CH_n^{T_n}$
with suitable subset $T_n \subseteq S_n$ of $p$-places of $K_n$, hopping that these
$T_n$-class groups are trivial in some circumstances, giving $\CH_n = \Ccl_n (T_n)$; since
$n \geq \delta'$, each element of $T_n$ is then fixed by the decomposition subgroup 
$g_{n/\delta'} \subseteq G_n$, so that $\CH_n \subseteq \CH_n^{g_{n/\delta'}}$, the resulting 
equality giving $\order \CH_n = \order \CH_n^{g_{n/\delta'}}$, when we are able to compute it. 

\smallskip
This is the key of the process; indeed, in this {\it relative} Chevalley--Herbrand formula, 
norm properties of the units intervene non-trivially, and this will introduce an interesting 
$p$-adic regulator in incomplete $p$-ramification theory (see Theorem \ref{Ptorsion}).

\subsection{The two fundamental cases}\label{twocases}

Consider a non-Galois $\Z_p$-extension $K \ne L, \ov L$.

\begin{remark}\label{eprime}
Let $e' \leq e$. Since a single place ramifies in $K_{e'}/k$, Chevalley--Herbrand 
formula in $K_{e'}/k$ yields $\CH_{e'} = 1$ since $\CH_k = 1$; 
then, in the exact sequence \eqref{gmn} for $K_e/K_{e'}$:
\begin{equation*}
1 \to \CH_e^\ram \too \CH_e^{g_{e/e'}} 
\too E_{e'} \cap \BN_{e/e'}(K_e^\times)/\BN_{e/e'}(E_e) \to 1,
\end{equation*}

\noindent
the three terms are trivial, which gives $E_{e'} \cap \BN_{e/e'}(K_e^\times)
= E_{e'} = \BN_{e/e'}(E_e)$ because of the product formula 
in the case of a single ramified place, implying that any unit is norm. 
\end{remark}

Because of the hypotheses $\CH_k = 1$ and $\delta \ne 0$, there is indeed 
a non trivial splitting of $p$ in $\wt k/k$ which intervenes in $K/k$ depending 
on the intersection $K \cap L$. For this, the two following diagrams may be 
useful for the sequel:
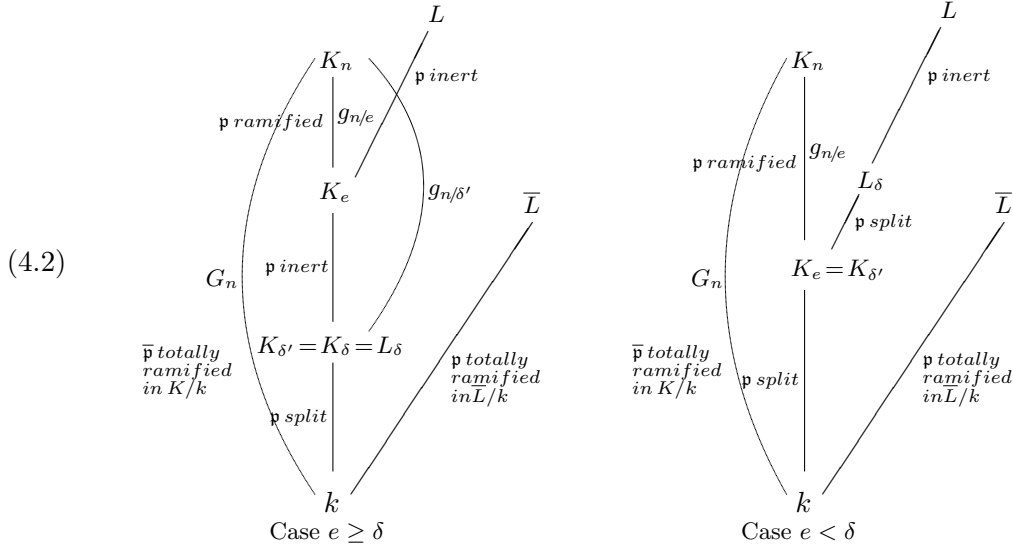
\unitlength=1.2cm 
\begin{equation}\label{Tpsplit2}
\begin{aligned}
\vbox{\hbox{\hspace{-1.0cm} 
\begin{picture}(10.0,5.8)
\put(2.4,0.1){$k$}
\put(0.4,1.8){\tiny\hbox{${\ov{\mathfrak p}}\,totally$}}
\put(0.4,1.6){\tiny\hbox{$ramified$}}
\put(0.4,1.4){\tiny\hbox{$in\,K/k$}}
\put(2.5,0.55){\line(0,1){1.2}}
\put(1.24,4.35){\tiny\hbox{${\mathfrak p}\,ramified$}}
\put(1.75,2.75){\tiny\hbox{${\mathfrak p}\,inert$}}
\put(1.8,1.1){\tiny\hbox{${\mathfrak p}\,split$}}
\put(2.5,2.2){\line(0,1){1.2}}
\put(2.35,3.55){\ft$K_e$}
\put(1.66,1.86){\ft$K_{\delta'}\!=\!K_{\delta}\!=\!L_\delta$}
\put(2.5,3.9){\line(0,1){1.0}}
\put(2.35,5.0){\ft$K_n$}
\put(2.75,3.8){\line(1,2){0.8}}
\put(2.7,0.3){\line(2,3){2.0}}
\put(3.8,1.72){\tiny\hbox{${\mathfrak p}\,totally$}}
\put(3.8,1.52){\tiny\hbox{$ramified$}}
\put(3.8,1.32){\tiny\hbox{$in \ov L/k$}}
\put(4.6,3.4){\ft$\ov L$}
\put(3.55,5.5){\ft$L$}
\put(3.4,4.85){\tiny\hbox{${\mathfrak p}\,inert$}}
\put(1.1,2.6){\ft$G_n$}
\put(2.55,4.44){\ft$g_{n\!/\!e}$}
\put(3.55,3.6){\ft$g_{n\!/\!{\delta'}}$}
\bezier{350}(2.3,0.3)(0.7,2.7)(2.3,5.1)
\bezier{250}(2.9,5.1)(4.1,3.8)(2.9,2.1)
\put(1.8,-0.2){\ft${\rm Case}\  e \geq \delta$}
\put(7.6,0.1){$k$}
\put(5.8,1.8){\tiny\hbox{${\ov{\mathfrak p}}\,totally$}}
\put(5.8,1.6){\tiny\hbox{$ramified$}}
\put(5.8,1.4){\tiny\hbox{$in\,K/k$}}
\put(6.46,3.9){\tiny\hbox{${\mathfrak p}\,ramified$}}
\put(7.0,1.5){\tiny\hbox{${\mathfrak p}\,split$}}
\put(7.7,0.55){\line(0,1){2.0}}
\put(7.55,2.7){\ft$K_e\!=\!K_{\delta'}$}
\put(7.7,3.1){\line(0,1){1.8}}
\put(7.55,5.0){\ft$K_n$}
\put(8.0,3.0){\line(1,2){0.3}}
\put(7.9,0.3){\line(2,3){2.0}}
\put(9.0,1.72){\tiny\hbox{${\mathfrak p}\,totally$}}
\put(9.0,1.52){\tiny\hbox{$ramified$}}
\put(9.0,1.32){\tiny\hbox{$in \ov L/k$}}
\put(8.25,3.25){\tiny\hbox{${\mathfrak p}\,split$}}
\put(9.8,3.4){\ft$\ov L$}
\put(9.2,5.55){\ft$L$}
\put(8.28,3.68){\ft$L_\delta$}
\put(8.45,3.95){\line(1,2){0.75}}
\put(9.05,4.85){\tiny\hbox{${\mathfrak p}\,inert$}}
\bezier{350}(7.5,0.3)(6.15,2.7)(7.5,5.1)
\put(6.45,2.6){\ft$G_n$}
\put(7.75,4.05){\ft$g_{n\!/\!e}$}
\put(7.0,-0.2){\ft${\rm Case}\  e < \delta$}
\end{picture}}} 
\end{aligned}
\smallskip
\end{equation}
\unitlength=1.0cm

\subsection{Norm residue symbols of the  
\texorpdfstring{$S_k$-units $x, \ov x$}{Lg}}

The norm residue symbols of the fundamental $S_k$-units $x$ and 
$\ov x$ have been computed in \cite[Theorem 5.3, Remark 5.4]{Gra2026b};
we recall this result, crucial for some Chevalley--Herbrand formulas; it strongly 
depends on the number $\delta$ but does not need any hypothesis on $\CH_k$,
nor on the numbers $e, \ov e$:

\begin{proposition}\label{mainorder}
Let $k$ be an imaginary quadratic field and let $p \geq 3$ be a prime 
number split in $k$. Let $x \in k^\times$ be the generator of 
${\mathfrak p}^{\hp}$, where $\hp \mid \hk$ is the order 
of $\Bcl({\mathfrak p})$, and let $\delta := v_{\ov {\mathfrak p}}
(x^{\,p-1}-1) - 1$. Let $K$ be a $\Z_p$-extension of $k$ such that 
${\mathfrak p}$, ${\ov {\mathfrak p}}$ are ramified from some layers
$K_e$, $K_{\ov e}$. For $n \geq \max(e + \delta, \ov e + \delta)$,
the symbols $\big (\ffrac{x\,,\,{K_n/k}}{\mathfrak p} \big)$,
$\big (\ffrac{\ov x\,,\,{K_n/k}}{{\mathfrak p}} \big)$ are of orders
$p^{n - \ov e - \delta}$, $p^{n - e - \delta}$, respectively. 
Whence $\big (\ffrac{x\,,\,{K_n/k}}{\ov {\mathfrak p}} \big)
= \big (\ffrac{x\,,\,{K_n/k}}{\mathfrak p} \big)^{-1}$,
$\big (\ffrac{\ov x\,,\,{K_n/k}}{\ov{\mathfrak p}} \big) =
\big (\ffrac{\ov x\,,\,{K_n/k}}{\mathfrak p} \big)^{-1}$.
\end{proposition}

Note that the order of $\big (\ffrac{x\,,\,{K_n/k}}{\mathfrak p} \big)$ 
depends on $\ov e$ and that of $\big (\ffrac{\ov x\,,\,{K_n/k}}
{\mathfrak p} \big)$ depends on $e$; this is coherent since when 
$e > \ov e$, the ideal ${\mathfrak p}$ does not ramify in $K_e/k$ and 
$\ov x$ is a ${\mathfrak p}$-adic unit, so that, by restriction to $K_e$,
$\big (\ffrac{\ov x\,,\,{K_n/k}}{\mathfrak p} \big) = 1$, giving the ``smallest order''.

\begin{corollary}\label{coro}
In our context where $e \geq 0$, $\ov e = 0$, put
$\wt \Omega_{n/0} := \{(s,s') \in G_n \times G_n,\ s\cdot s' = 1 \}$
(i.e., $(s,s')$ satisfying the product formula).
Let $\omega_{n/0} : E_k^{S_k} = \langle x, \ov x \rangle
\too \wt \Omega_{n/0} \simeq G_n \simeq \Z/p^n \Z$ be defined
by $\omega_{n/0}(y) := \big (\ffrac{y \,,\,{K_n/k}}{{{\mathfrak p}}} \big)$
for all $y \in E_k^{S_k}$. Then $\omega_{n/0}(E_k^{S_k}) =
\langle \omega_{n/0} (x), \omega_{n/0}(\ov x) \rangle =
\langle \omega_{n/0} (x) \rangle$ since $\omega_{n/0}(\ov x)
= \omega_{n/0} (x)^{p^e}$, $G_n$ being $p$-cyclic.
\end{corollary}

\subsection{Computation of \texorpdfstring
{$\order \omega_{n/\delta'} (E_{\delta'})$}{Lg}}
Recall that $({\mathfrak p})_{\delta'} := \prod_{i = 1}^{\spp} {{\mathfrak p}_i}$
and $(\ov {\mathfrak p})_{\delta'} = ({\mathfrak p}_\0)^\spp$
in $K_{\delta'}$.

\begin{lemma}\label{delta}
For any $\varepsilon \in E_{\delta'}$ and ${\mathfrak p}_i \in 
S_{\delta'}^{\mathfrak p}$, let $\delta_{{\mathfrak p}_i}(\varepsilon) \geq 0$ 
be defined by the ideals in $K_{\delta'}$:
$$\big({\varepsilon^{p - 1} - 1}\big) =  {\mathfrak a} \cdot
\hbox{$\prod_{i = 0}^{\spp}$}\ 
{{\mathfrak p}_i}^{1+\delta_{{\mathfrak p}_i}(\varepsilon)} =
{\mathfrak a} \cdot ({\mathfrak p})_{\delta'} \cdot 
{\mathfrak p}_\0^{1+\delta_{{\mathfrak p}_\0}(\varepsilon)}
\cdot \hbox{$\prod_{i = 1}^{\spp}$}\ 
{{\mathfrak p}_i}^{\delta_{{\mathfrak p}_i}(\varepsilon)}, $$

\noindent
where ${\mathfrak a}$ is a prime-to-$p$ ideal.
Then, for $i \in [1, \spp]$, $\iota_{{\mathfrak p}_i}(\varepsilon) \in 
(\CU_{\delta'}^{{\mathfrak p}_i})^{p^{\delta_{{\mathfrak p}_i}(\varepsilon)}}$; 
so, the symbol $\big(\ffrac{\varepsilon\,,\,K_n/K_{\delta'}}
{{\mathfrak p}_i} \big) \in g_{n/e}$ is of order 
$p^{n - e - \delta_{{\mathfrak p}_i}(\varepsilon)}$ for $n$ large enough.
\end{lemma}

\begin{proof}
Since ${\mathfrak p}_i$ is of residue degree $1$ in $K_{\delta'}/\Q$, for 
$i \in [1, \spp]$ and since $p \ne 2$, the group $\CU_{\delta'}^{{\mathfrak p}_i} 
\simeq 1 + p\Z_p \simeq \Z_p$ is pro-cyclic 
and its image by the local reciprocity map is the inertia group $\Gal(K/K_e)$ 
of ${\mathfrak p}_i$. We get $\iota_{{\mathfrak p}_i}(\varepsilon) \in  
(\CU_{\delta'}^{{\mathfrak p}_i})^{p^{\delta_{{\mathfrak p}_i}(\varepsilon)}}
\setminus (\CU_{\delta'}^{{\mathfrak p}_i})^{p^{1+\delta_{{\mathfrak p}_i}(\varepsilon)}}$, 
whence $\big(\ffrac{\varepsilon\,,\,K_n/K_{\delta'}}{{\mathfrak p}_i} \big)$,
generator of $\Gal(K_n/K_e)^{p^{\delta_{{\mathfrak p}_i}}}$, is of order 
$p^{n - e - \delta_{{\mathfrak p}_i}(\varepsilon)}$, proving the claim.
\end{proof}

Note that for $i = 0$, 
since $\ov {\mathfrak p}$ is ramified in $K_{\delta'}/k$, the 
characterization of $(\CU_{\delta'}^{{\mathfrak p}_\0})^p$ is more difficult 
as soon as $e \ne 0$ (see \cite{FV2002} giving in particular that 
$1 + {\mathfrak p}_\0^{2\cdot \spp} \subseteq (1 + {\mathfrak p}_\0)^p$;
which explain that we try to privilege the ``${\mathfrak p}$-part'' of the 
$p$-adic properties. For instance, Hasse norm residue symbols in 
$K/K_{\delta'}$ define maps with values in $\Omega_{n/\delta'} := 
\bigoplus_{i=0}^{\spp} g_{n/\delta'}$; but taking into account the 
product formula we can consider:
$$\wt \Omega_{n/\delta'}^{\,\mathfrak p} \simeq 
\hbox{$\bigoplus_{i=1}^{\spp}$}\, g_{n/\delta'}
\simeq \big(\Z/p^{n-\delta'}\Z\big)^{\spp}, $$ 

\noindent
with the canonical map $\omega_{n/\delta'} : \big\{x \in K_{\delta'}^\times,\ 
(x) = \BN_{n/\delta'}({\mathfrak a}_n) \, \hbox{for some ideal of $K_n$} \big\} 
\to \wt \Omega_{n/\delta'}^{\,\mathfrak p}$, justified by the fact that any element, 
norm of an ideal of $K_n$, is local norm at all the non-ramified places. 
Similarly, any element of $K_n^\times$, prime to a non-ramified place is 
a local norm at this place. So, the Chevalley--Herbrand formula in this 
context writes, when $\CH_k = 1$:
$$\order \CH_n^{g_{n/\delta'}} = \frac{[p^{n-e}]^{\spp} \cdot p^{n-\delta'}}
{[K_n : K_{\delta'}] \cdot \order \omega_{n/\delta'} (E_{\delta'}) } =
\frac{[p^{n-e}]^{\spp}}{\order \omega_{n/\delta'} (E_{\delta'})}
= \frac{\order \wt \Omega_{n/e}^{\,\mathfrak p}}
{\order \omega_{n/e} (E_{\delta'})} , $$

\noindent
where the computation of the denominator is the non-trivial part of such 
relative formula.
When we consider units, norm residue symbols are in 
the inertia group $g_{n/e} \subseteq g_{n/\delta'}$ and we may put, 
by abuse, $\wt \Omega_{n/e}^{\,\mathfrak p} = \bigoplus_{i=1}^{\spp} g_{n/e}
\simeq (\Z/p^{n-e}\Z)^\spp$, which explain the above Chevalley--Herbrand 
formula with $\omega_{n/\delta'} (E_{\delta'}) = \omega_{n/e} (E_{\delta'})
\subset \wt \Omega_{n/e}^{\,\mathfrak p}$. If we consider $S_{\delta'}$-units,
$\omega_{n/\delta'}$ with values in $\wt \Omega_{n/{\delta'}}^{\,\mathfrak p}  = 
\bigoplus_{i=1}^{\spp} g_{n/{\delta'}}$ is needed.

\smallskip 
Let $\{\varepsilon_1, \ldots, \varepsilon_\re \}$ be a fundamental system
of units of $K_{\delta'}$; then $\omega_{n/\delta'}(E_{\delta'})$ is generated 
by the $\omega_{n/\delta'}(\varepsilon_j)$, $j \in [1, \re]$; we note that
the $\Z$-rank of $E_{\delta'}$ is only $\re = {\spp}- 1$; so, in the above 
formula, necessarily $\order \CH_n \geq p^{n - e}$, as expected in the 
imaginary case. 

\smallskip
We can state, in full generality for a $p$-principal imaginary quadratic field $k$: 

\begin{theorem}\label{Ptorsion}
Let $k$ be an imaginary quadratic field, and let $p \geq 3$ split 
in $k$. Assume that $\CH_k = 1$. Let $K/k$ be a $\Z_p$-extension,
$K \ne L, \ov L$, and let $e \geq 0$ be defined by $K \cap L =: K_e$. Let 
$\delta = v_{\ov{\mathfrak p}}(x^{\,p-1}-1) - 1$, where $x$ is the 
fundamental $S_k$-unit given by ${\mathfrak p}$; put $\delta' := 
\min(e, \delta)$, $\spp = p^{\delta'}$ and $\re = \spp - 1$.
Let $\CT^{\mathfrak p}_{\delta'} := \tor_{\Z_p} 
(\Gal(H^{{\mathfrak p}\,\pr}_{\delta'})/K_{\delta'})$, where
$H^{{\mathfrak p}\,\pr}_{\delta'}$ is the maximal abelian 
${\mathfrak p}$-ramified pro-$p$ extension of $K_{\delta'}$.\,\footnote{\,The 
case $\delta' = 0$ occurs in the second case of Schema \eqref{Tpsplit2}, 
otherwise $\delta = 0$ which has been ruled out; then the second 
case implies $e = 0$ (total ramification of $p$ in $K/k$), then 
$\order \CH_n^{G_n} = p^n$, $\spp = 1$, $\re = 0$, $\CT^{\mathfrak p}_{\delta'} = 1$.}
Then, the order of the finite group $\CT^{\mathfrak p}_{\delta'}$ 
is given by the regulator $\CR^{\mathfrak p}_{\delta'} \in \Z_p \setminus \{0\}$ 
having the following properties:

\smallskip
(i) $\CR^{\mathfrak p}_{\delta'} = \CU_{\delta'}^{\mathfrak p}{}^*/
\ov E_{\delta'}^{\,\mathfrak p}$, where $\CU_{\delta'}^{\mathfrak p}{}^*$
is the subgroup of $\CU_{\delta'}^{\mathfrak p} =  
\hbox{$\bigoplus_{i=1}^{\spp}$}\CU_{\delta'}^{{\mathfrak p}_i}$ 
of elements of norm~$1$ in $K_{\delta'}/k$, 
$\ov E_{\delta'}^{\,\mathfrak p}$ is the topological closure of
the diagonal embedding $\iota_{\mathfrak p}(E_{\delta'})$ of $E_{\delta'}$ in 
$\CU_{\delta'}^{\mathfrak p}{}^*$;

\smallskip
(ii) $\order \omega_{n/{\delta'}}(E_{\delta'}) = 
\frac{[p^{n-e}]^{\re}}{\CR^{\mathfrak p}_{\delta'}}$;

\smallskip
(iii) $\order \CH_n^{g_{n/\delta'}} = p^{n - e} \cdot \CR^{\mathfrak p}_{\delta'}$. 
\end{theorem}

\begin{proof}
From Lemma \ref{functorial}\,(iv), there exist some specific ``product 
formulas'' for symbols of units $\varepsilon \in E_{\delta'}$, due to $E_k = 1$; 
they are as follows:
\begin{equation*}
\left \{\begin{aligned}
\Big(\ffrac{\varepsilon,K_n/K_{\delta'}} {{\mathfrak p}_\0^{}} \Big) 
& = \Big(\ffrac{\BN_{\delta'/0}(\varepsilon),K_n/k}{\ov {\mathfrak p}} \Big) = 1, \\
\prd_{k = 1}^{\spp} \Big(\ffrac{\varepsilon,K_n/K_{\delta'}}{{\mathfrak p}_k} \Big)
& = \Big(\ffrac{\BN_{\delta'/0}(\varepsilon),K_n/k}{{\mathfrak p}} \Big) = 1.
\end{aligned} \right .
\end{equation*}

Which shows that $\omega_{n/e} : E_{\delta'} \too 
\wt \Omega_{n/e}^{\,\mathfrak p} \simeq (\Z/p^{n-e} \Z)^{\spp}$ will not 
be surjective since the $\Z$-rank of $E_{\delta'}$ is $\re = \spp - 1$.
The kernel of $\omega_{n/e}$ is the subgroup of $E_{\delta'}$
of units $\varepsilon$ such that, for instance,
$\big(\frac{\varepsilon\,,\,K_n/K_{\delta'}}{{\mathfrak p}_i} \big) = 1$
for all $i \in [1, \re]$; indeed, using the two product formulas, it follows that,
for such a unit, 
$\big(\frac{\varepsilon\,,\,K_n/K_{\delta'}}{{\mathfrak p}_{\spp}} \big) = 
\big(\frac{\varepsilon\,,\,K_n/K_{\delta'}}{{\mathfrak p}_0} \big) = 1$
automatically.

The order of these symbols 
$\big(\frac{\varepsilon\,,\,K_n/K_{\delta'}}{{\mathfrak p}_i} \big) \in g_{n/e}$ is $1$; 
this means, from Lemma \ref{delta}, that, locally, $\iota_{{\mathfrak p}_i}
(\varepsilon) \in (\CU_{\delta'}^{{\mathfrak p}_i})^{p^{n-e}}$, for all 
$i \in [1, \spp]$, which may be written
$\iota_{\mathfrak p}(\varepsilon) \in (\CU_{\delta'}^{\mathfrak p})^{p^{n-e}}$, 
where $\iota_{\mathfrak p}$ is the diagonal embedding $E_{\delta'} \too 
\CU_{\delta'}^{\mathfrak p} := \hbox{$\bigoplus_{i=1}^{\spp}$} 
\CU_{\delta'}^{{\mathfrak p}_i}$. 
Since the reciprocal is then obvious,
$$\Ker(\omega_{n/{\delta'}}) = \{\varepsilon \in E_{\delta'},\  \iota_{\mathfrak p}
(\varepsilon) \in (\CU_{\delta'}^{\mathfrak p})^{p^{n-e}}\}. $$

Since $\BN_{\delta'/0}(\varepsilon) = 1$ and $\tor_{\Z_p} (\CU_k) = 1$, 
if $\iota_{\mathfrak p}(\varepsilon) = u^{p^{n-e}}$ this implies 
$\BN_{\delta'/0}(u) = 1$, hence $u \in \CU_{\delta'}^{\mathfrak p}{}^*$. 
Whence the exact sequence: 
$$1 \too \Ker(\omega_{n/{\delta'}})/E_{\delta'}^{p^{n-e}} \tooo E_{\delta'}/E_{\delta'}^{p^{n-e}} 
\mathop {\tooo}^{\iota_{\mathfrak p}} \CU_{\delta'}^{\mathfrak p}{}^* /
(\CU_{\delta'}^{\mathfrak p}{}^*)^{p^{n-e}}, $$

\noindent
where the image is $\iota_{\mathfrak p}(E_{\delta'})
(\CU_{\delta'}^{\mathfrak p}{}^*)^{p^{n-e}} /
(\CU_{\delta'}^{\mathfrak p}{}^*)^{p^{n-e}} =
\ov E_{\delta'}^{\,\mathfrak p} (\CU_{\delta'}^{\mathfrak p}{}^*)^{p^{n-e}}/
(\CU_{\delta'}^{\mathfrak p}{}^*)^{p^{n-e}}$ since for all 
$\ov \varepsilon \in \ov E_{\delta'}^{\,\mathfrak p}$ and $N \geq n - e$,
there exists $\varepsilon_N \in E_{\delta'}$ such that $\ov \varepsilon \in
\iota_{\mathfrak p}(\varepsilon_N) (\CU_{\delta'}^{\mathfrak p}{}^*)^{p^N}$.

\smallskip
It is known, from Ax--Baker--Brumer, that the $\Z_p$-rank of $\ov E_{\delta'}^{\,\mathfrak p}$ 
is $\re$ (analogue, for imaginary quadratic fields, of the proof of Leopoldt's
conjecture for real abelian extensions of $\Q$); its rank is equal to that of 
$\CU_{\delta'}^{\mathfrak p}{}^*$. So, the index 
$[\CU_{\delta'}^{\mathfrak p}{}^* : \ov E_{\delta'}^{\,\mathfrak p}]$
is finite and for $n$ large enough, $(\CU_{\delta'}^{\mathfrak p}{}^*)^{p^{n-e}} 
\subseteq \ov E_{\delta'}^{\,\mathfrak p}$
and  $\iota_{\mathfrak p}(E_{\delta'}/E_{\delta'}^{p^{n-e}}) = 
\ov E_{\delta'}^{\,\mathfrak p}/(\CU_{\delta'}^{\mathfrak p}{}^*)^{p^{n-e}}$, 
giving the final exact sequence:
$$1 \too \Ker(\omega_{n/{\delta'}})/E_{\delta'}^{p^{n-e}} \tooo E_{\delta'}/E_{\delta'}^{p^{n-e}} 
\mathop {\tooo}^{\iota_{\mathfrak p}} \ov E_{\delta'}^{\,\mathfrak p}/
(\CU_{\delta'}^{\mathfrak p}{}^*)^{p^{n-e}} \too 1; $$

\noindent
thus, $\order {\rm Im}(\iota_{\mathfrak p}) = 
\ffrac{[\CU_{\delta'}^{\mathfrak p}{}^* : (\CU_{\delta'}^{\mathfrak p}{}^*)^{p^{n-e}}]}
{[\CU_{\delta'}^{\mathfrak p}{}^* :  \ov E_{\delta'}^{\,\mathfrak p}]}
= \ffrac{[p^{n - e}]^\re}{\CR^{\mathfrak p}_{\delta'}}$, then the  
Chevalley--Herbrand formula (iii).
\end{proof}

From class field theory in the tricky case of incomplete $p$-ramification 
(see for instance \cite[III.1\,(c)]{Gra2005} for more details, where all possible
sets of $p$-places are considered), we have (since $K_{\delta'}$ is $p$-principal)
$\Gal(H^{{\mathfrak p}\,\pr}_{\delta'}/
K_{\delta'}) \simeq \CU_{\delta'}^{\mathfrak p}/\ov E_{\delta'}^{\,\mathfrak p}$; we 
deduce that $\Gal(H^{{\mathfrak p}\,\pr}_{\delta'}/K_{\delta'})$ is of $\Z_p$-rank $1$,
whence $\Gal(H^{{\mathfrak p}\,\pr}_{\delta'}/K_{\delta'}) \simeq \Z_p \oplus 
\CT^{\mathfrak p}_{\delta'}$, where the torsion group $\CT^{\mathfrak p}_{\delta'}$
may be seen as the analog of $\CT_F = \tor_{\Z_p}^{} \Gal(H_F^\pr/F)$ for general number 
fields $F$, but in incomplete $p$-ramification; since 
$\tor_{\Z_p}^{}(\CU_{\delta'}^{\mathfrak p}/\ov E_{\delta'}^{\,\mathfrak p}) = 
\CU_{\delta'}^{\mathfrak p}{}^*/\ov E_{\delta'}^{\,\mathfrak p}$, its order is 
given by the regulator $\CR^{\mathfrak p}_{\delta'}$.

\smallskip
In \S\,\ref{torsion} we give a {\sc pari/gp} program computing 
$\CT^{\mathfrak p}_{\delta'}$ and prove that $\CT^{\mathfrak p}_{\delta'} 
\ne 1$ for $e \ne 0$.

\begin{remarks}\label{remas}
(i) Note that, in $\CU_{\delta'}^{{\mathfrak p}_\0}{}^*$, the fact that 
$\varepsilon$ is local norm at ${\mathfrak p}_\0$ does not imply 
$\varepsilon \in (\CU_{\delta'}^{{\mathfrak p}_\0}{}^*)^{p^{n-e}}$;
otherwise we would have used the following result \cite[Theorems 
2.3, 3.5\,(ii), Definition 5.1]{Gra2018}, which, perhaps, suggests improvements
in incomplete $p$-ramification:

\smallskip
Let $F$ be a number field satisfying the Leopoldt conjecture at the 
prime $p$. Let $E_F$ be the group of $p$-principal global units of 
$F$ and let $\CU_F$ be the group of principal local units at $p$.
The Kummer--Leopoldt constant $\kappa^{}_F$ is the smallest 
integer $c$ such that the following condition (restricted to
our particular context) is fulfilled:

\smallskip
\centerline{\it For all $n \geq 0$, any unit $\varepsilon \in E_F$, such that 
$\varepsilon \in \CU_F^{p^{n+c}}$, is necessarily in $E_F^{p^n}$.}

\medskip
Then, the exponent of the normalized $p$-adic regulator $\CR_F$
is equal to $p^{\kappa^{}_F}$. If $F$ is $p$-rational, then $\kappa^{}_F = 0$.
Unfortunately, here, the condition on the place ${\mathfrak p}_\0$ is not fulfilled.

\smallskip
(ii) Let $\wt {K_{\delta'}}$ be the compositum of the 
$\Z_p$-extensions of $K_{\delta'}$, equal to $H_{\delta'}^\pr$, the 
maximal abelian $p$-ramified pro-$p$ extension of $K_{\delta'}$, and
let $H^{{\mathfrak p}\,\pr}_{\delta'} \subseteq H_{\delta'}^\pr$ 
be the maximal abelian ${\mathfrak p}$-ramified pro-$p$ extension of 
$K_{\delta'}$. 
Since the usual regulator $\CR_{\delta'}$ is trivial, 
$\tor_{\Z_p}(\CU_{\delta'}^*/\ov E_{\delta'})= 1$.
Thus $\ov E_{\delta'}$ is a direct factor of $\CU_{\delta'}^*$.
The $\Z_p$-modules $\ov E_{\delta'}$ and 
$\ov E_{\delta'}^{\,\mathfrak p}$ are free of same $\Z_p$-rank $\re$.

\smallskip
(iii) Since the Fermat quotients defining the $\delta_i(\varepsilon_j)$'s, for 
$i \in [1, \spp]$, are often trivial, the orders of the $\omega_{n/\delta'}
(\varepsilon_j)$ are often $p^{n - e}$. Many conjectural aspects of the
$p$-adic regulators of algebraic numbers have been discussed in
\cite{Gra2016, Gra2019c}, being aware that they are currently inaccessible.
See also \cite{Mai2002, Mai2005},  \cite{Nel2013} for similar subjects. 

\smallskip
(iv) Similar property of units has been used in \cite{OT1995} for the 
totally real base fields and the cyclotomic $\Z$-extension, then in 
\cite{Gra2017b, Gra2019a}, still about Greenberg totally real conjectures.
This type of reasoning is necessary encountered in Iwasawa's theory of imaginary 
base fields, as for instance in \cite[Lemma 4]{Oza2001}, \cite[Theorem 1.3]{Yan2026}. 

\smallskip
For a general setting, larger conjectures, and proofs following the
techniques of Ax--Baker--Brumer and others, see \cite{Jau1985}.
\end{remarks}

Note that the $\Z_p$-rank
of $\ov E_{\delta'}^{\ov {\mathfrak p}}$ is also $\re$, which is largely 
confirmed by the {\sc pari/gp} program given in \S\,\ref{torsion}.

\section{General Chevalley--Herbrand formulas}

\subsection{Invariant classes in \texorpdfstring{$K_n/k$}{Lg}}
The Chevalley--Herbrand formula in $K_n/k$, when $\CH_k = 1$,
gives, for $n$ large enough (without any hypotheses on $\delta$, $e$): 
\begin{equation}\label{CH1}
\order \CH_n^{G_n} = \frac{\order \CH_k \times [p^{n-e}
\cdot p^n]}{p^n \times 1} = p^{n-e}. 
\end{equation}

Let $\CP$, $\CP_\0$ be the products of prime ideals of 
$K_n$ above ${\mathfrak p}$, ${\ov {\mathfrak p}}$, respectively; then
we have $\CP = {\mathfrak P}_{1} \cdots {\mathfrak P}_\spp$,
$\CP_\0 = {\mathfrak P}_{0}$, where ${\mathfrak P}_{j}$ denotes the prime 
ideals of $K_n$ above ${\mathfrak p}$ for $j \in [1, \spp]$ and 
${\mathfrak P}_{0}$ above $\ov{\mathfrak p}$. Then $\CH_n^{G_n} = \CH_n^\ram = 
\Ccl_n (\langle\CP,\CP_\0 \rangle)$; whence: 
$$\BN_{n/0}(\CH_n^{G_n}) = \BN_{n/0} (\langle \CP, \CP_\0 \rangle)
= \langle {\mathfrak p}^{p^e}, \ov {\mathfrak p} \rangle  = 
\langle x^{p^e}, \ov x \rangle, $$ 

\noindent
up to prime-to-$p$ powers.
The first element $\CH_n^1:= \CH_n^{G_n}$ of the filtration of $\CH_n$ and the second one
$\CH_n^2$, defined by $\CH_n^2/\CH_n^1 = (\CH_n/\CH_n^1)^{G_n}$,
give an order reduced to the norm factor $\ffrac{\order \wt \Omega_{n/0}^{\,\mathfrak p}}
{\order \omega_{n/0}(\langle x^{p^e}, \ov x \rangle)}$, where $\wt \Omega_{n/0}^{\,\mathfrak p}
\simeq \Z/p^{n-e} \Z$; 
then $\omega_{n/0}(x) = \big(\frac{x, K_n/k}{{\mathfrak p}} \big)$, 
$\omega_{n/0}(\ov x) = \big(\frac{\ov x, K_n/k}{{\mathfrak p}} \big)$, are of 
orders $p^{n-\ov e-\delta} = p^{n - \delta}$, 
$p^{n-e-\delta}$, respectively (Corollary \ref{coro} to Proposition \ref{mainorder}); 
so, $\omega_{n/0}(x)^{p^e}$ is of order $p^{n-e-\delta}$, thus
$\order(\CH_n^2/\CH_n^1) = \ffrac{p^{n-e}}{p^{n-e-\delta}} = p^\delta$,
giving $\order \CH_n^2 = p^{n - e + \delta}$, for $n$ large enough,
hence $\order \CH_n \geq p^{n - e + \delta}$. One finds again the
case $\delta = 0$ for which $\CH_n^2 = \CH_n^1$, thus 
$\order \CH_n = p^{n-e}$ for all $n \geq e$ (\cite[Theorem 7.1]
{Gra2026b} and \cite[Th\'eor\`eme~11]{Jau2026}).

\smallskip
We have $\order \langle \Ccl_n(\CP), \Ccl_n(\CP_\0)\rangle  = p^{n-e}$
without information on the structure of this group (see \S\,\ref{acyclo}
giving this structure for the first layer of the anti-cyclotomic 
$\Z_3$-extensions).

\subsection{Invariant \texorpdfstring{$S_n$-classes in $K_n/k$}{Lg}}\label{Sn}
Let $S_n := \{ {\mathfrak p}_{1}, \ldots, {\mathfrak p}_{\spp}, {\mathfrak p}_{0}\}$,
$S_n^{\mathfrak p} := \{ {\mathfrak p}_{1}, \ldots, {\mathfrak p}_{\spp}\}$
and $S_n^{\ov {\mathfrak p}} := \{ {\mathfrak p}_\0 \}$, in $K_n$.
Let $E_k^{S_k} = \langle x, \ov x \rangle$, $E_k^{S_k^{\mathfrak p}} = 
\langle x \rangle$, $E_k^{S_k^{\ov {\mathfrak p}}} = \langle \ov x \rangle$,
be the corresponding groups of $S_k$-units of~$k$.
The local degrees $d_v$ of $K_n/k$ at ${\mathfrak p}$, $\ov {\mathfrak p}$, are
$p^{n - \delta'}$, $p^n$, where $\delta' := \min(e,\delta)$,
and the ramification indices $\rho_v$ are $p^{n-e}$, $p^n$. To compute $\order 
(\CH_n^{T_n})^{G_n}$, $T_n \in \{S_n, S_n^{\mathfrak p}, S_n^{\ov {\mathfrak p}} \}$, 
one must take into account the local degrees and/or the ramification indices, depending
on $T_n$ (\cite[Th\'eor\`eme III.1.9, p. 177]{Jau1986}, \cite[Corollary 3.9]{Gra2017a});
for $n$ large enough:
\begin{equation}\label{CH3}
\left \{ \begin{aligned}
(i)\ \  \order (\CH_n^{S_n})^{G_n} & = 
\frac{\order \CH_k^{S_k} \cdot \prod_{v \in S_k} d_v} 
{[K_n : k] \cdot (E_k^{S_k} : E_k^{S_k} \cap \BN_{n/0} (K_n^\times))} 
 = \frac{p^{n - \delta'} \cdot p^n}{p^n \cdot 
\order \omega_{n/0} (\langle x, \ov x \rangle)} \\
& \hspace{-1.3cm} = \frac{p^{n - \delta'}}
{\order \langle \omega_{n/0} (x), \omega_{n/0}(\ov x) \rangle} 
 = \frac{p^{n - \delta'}}{\order \langle \omega_{n/0} (x) \rangle} 
= p^{\delta - \delta'} =
\left \{ \begin{aligned}
& 1 & (e \geq \delta) \\
& p^{\delta - e} & (e < \delta) 
\end{aligned} \right. \\
& \hspace{-2.5cm} \hbox{Then formulas with $S_k^{\mathfrak p} = \langle x \rangle$ 
and $S_k^{\ov {\mathfrak p}} = \langle \ov x \rangle$:} \\
(ii)\ \,  \order (\CH_n^{S_n^{\mathfrak p}})^{G_n} & =
\frac{p^{n - \delta'} \cdot p^n}{p^n \cdot \order \langle \omega_{n/0} (x) \rangle}
= p^{\delta - \delta'} =
\left \{ \begin{aligned}
& 1 & (e \geq \delta) \\
& p^{\delta - e} & (e < \delta) 
\end{aligned} \right. \\
(iii)\  \order (\CH_n^{S_n^{\ov{\mathfrak p}}})^{G_n} & = 
\frac{ p^n \cdot p^{n - e}}{p^n \cdot \order \langle \omega_{n/0} (\ov x) \rangle}
= p^\delta.
\end{aligned} \right.
\end{equation}

We observe the particular case $e \geq \delta$ giving $\CH_n^{S_n} 
= \CH_n^{S_n^{\mathfrak p}} = 1$ in (i) and (ii), hence the trick described in 
the Introduction and that we consider in the next section. 

\smallskip
Formula (ii) has some importance since reasonings with the set of 
primes above ${\mathfrak p}$ is much simpler. In the case $e \geq \delta$, 
Formulas (i) and (ii) lead to $\Ccl_n ({\mathfrak p}_\0) \in \Ccl_n(S_n^{\mathfrak p})$;
so, this implies $\CH_n^{G_n} = \langle \Ccl_n(\CP)\rangle$ (of order $p^{n-e}$)
and $\Ccl_n(\CP_\0) = \Ccl_n(\CP)^a$, $a \in \Z_p^\times$.

\section{The case \texorpdfstring{$e \geq \delta$}{Lg}}

In this section, we refer especially to Theorem \ref{Ptorsion}\,(iii) and 
to Formulas \eqref{CH3} in which  $\delta' = \min(e, \delta) = \delta$,
for $e \geq \delta$.

\subsection{Main result}
The case $e \geq \delta$ yields the following infinite family of $\Z_p$-extensions 
with minimal Iwasawa's invariants (i.e., non-exceptional $\Z_p$-extensions):

\begin{theorem}\label{fundamental}
Let $k$ be an imaginary quadratic field, and let $p \geq 3$, split in $k$, such
that $\CH_k = 1$.  Let $\delta = v_{\ov {\mathfrak p}}(x^{\,p-1}-1) - 1$, where $x$ 
is the fundamental $S_k$-unit given by ${\mathfrak p}$. 
Let $K/k$ be a $\Z_p$-extension distinct from $L, \ov L$; put $K \cap L =: K_e$,
and assume $e \geq \delta$. Let $S_n^{\mathfrak p}$, for $n \geq \delta$,
be the set of the $\spp := p^\delta$ prime ideals of $K_n$ above ${\mathfrak p}$.

\smallskip
Then $\CH_n = \Ccl_n (\langle S_n^{\mathfrak p} \rangle) = 
\Ccl_n (\langle {\mathfrak p}_1, \ldots, {\mathfrak p}_\spp \rangle)$, 
$\lambda(K/k) = 1$, $\mu(K/k) = 0$, $\nu(K/k) = t_\delta - e$, where $p^{t_\delta} = 
[\CU_\delta^{\mathfrak p}{}^* : \ov E_\delta^{\,\mathfrak p}]$ is the order of
$\CT^{\mathfrak p}_\delta := \tor_{\Z_p} 
(\Gal(H^{{\mathfrak p}\,\pr}_\delta)/K_\delta)$. 
\end{theorem}

\begin{proof}
From Formula \eqref{CH3}\,(ii) for $e \geq \delta$, we get 
$\CH_n = \Ccl_n (\langle S_n^{\mathfrak p} \rangle)$; thus $\CH_n \subseteq 
\CH_n^{g_{n/\delta}}$ since each prime ideals ${\mathfrak p}_i$ of $S_n^{\mathfrak p}$ 
is invariant under $g_{n/\delta}$ (being above those of $S_\delta^{\mathfrak p}$ which
are ramified in $K_n/K_e$ and inert in $K_e/K_\delta$); whence $\CH_n = 
\CH_n^{g_{n/\delta}}$. From the Chevalley--Herbrand formula of 
Theorem \eqref{Ptorsion}\,(iii), we obtain:
$$\order \CH_n = p^{n - e}\cdot \CR^{\mathfrak p}_\delta = 
p^{n - e}\cdot \order \CT^{\mathfrak p}_\delta, $$

\noindent
hence the result.
\end{proof}

If $\CH_k = 1$ and $\delta = 0$, then $\spp = 1$ and $\CT^{\mathfrak p}_\delta = 1$; 
one finds again Gold--Sands criterion \cite{San1993}, then \cite[Theorem\,10.7]{Gra2026b} 
and \cite[Th\'eor\`emes 9, 11]{Jau2026}.

\subsection{Interpretation of Chevalley--Herbrand formulas 
and regulators}

We will introduce some genus fields to interpret the previous results and
to deduce some properties. 

\subsubsection{Definitions -- Notations}
We consider 
the diagram below, where $K$ is any $\Z_p$-extension distinct from $L, \ov L$, 
and where $H_K^\nr = \bigcup_n K H_n^\nr$ is the $p$-Hilbert class field of $K$, 
with $\Gal(H_K^\nr/K) \simeq \CH_K := \ds \limproj_n \CH_n$ (the $p$-class group 
of $K$). 

The $\Z_p$-extension $K/k$ is characterized by a generator $\wt\sigma_K^{}$ 
of $G_K = \Gal(\wt k/K)$ (cf. Proosition \ref{3cases}). 

\smallskip
In view of Theorem \ref{C}, we assume $e \geq \delta$.

\smallskip 
Let $H_{n/0}^\gen$ be the genus field of $K_n/k$ (maximal unramified 
$p$-extension of $K_n$, abelian over~$k$). Since $K_n/k$ is cyclic, $H_{n/0}^\gen$ 
is the subfield of the Hilbert class field $H_n^\nr$ of $K_n$ fixed by the image
of $\CH_n^{1 - \sigma_{n/0}}$, where $\sigma_{n/0}$ is a generator of $G_n = 
\Gal(K_n/k)$. Since $H_{n/0}^\gen/k$ is abelian and $p$-ramified, $H_{n/0}^\gen 
\subset H_k^\pr = \wt k$. Then, $\wt k = \bigcup_n K H_{n/0}^\gen$ since 
$[K H_{n/0}^\gen : K] = \order \CH_n^{G_n} = p^{n - e}$ yields 
$\rk_{\Z_p}(\Gal(\bigcup_n K H_{n/0}^\gen/K)) = 1$. 

\smallskip
Let $H_{\wt k}^\nr = \bigcup_{F \subset \wt k,\, F/k\, {\rm finite}} 
\wt k H_F^\nr$ be the $p$-Hilbert class field of $\wt k$, where 
$\Gal(H_{\wt k}^\nr/\wt k) \simeq \CH_{\,\wt k} = \ds \limproj _F\CH_{F}$
($p$-class group of $\wt k$).

Similarly, let $H_{n/\delta}^\gen$ be the genus field of $K_n/K_\delta$; it
is of degree $\order \CH_n^{g_{n/\delta}}\!=\!
p^{n-e} \cdot \CR_\delta^{\mathfrak p}$ over $K_n$ (Theorem \ref{Ptorsion}\,(iii)),
hence of constant degree $\CR_\delta^{\mathfrak p}$ over $H_{n/0}^\gen$, and
is fixed by the image of $\CH_n^{1 - \sigma_{n/\delta}}$, where $\sigma_{n/\delta}$
is a generator of $g_{n/\delta}$.

\smallskip
The field $H_\delta^{{\mathfrak p}\,\pr}$ is the maximal
abelian ${\mathfrak p}$-ramified pro-$p$-extension of $K_\delta$, and 
$\wt{H_\delta^{\mathfrak p}}/K_\delta$ the unique ${\mathfrak p}$-ramified
$\Z_p$-extension of $K_\delta$, hence the compositum $K_\delta \ov L$:

\subsubsection{First diagram}\label{diagram1}

Some parts of the diagram (like $H_n^\nr = H_{n/\delta}^\gen$ and $H_K^\nr 
= \bigcup_n K H_{n/\delta}^\gen$) will be justified in the proof of the following 
Theorem \ref{C}, still in the $p$-principal $p$-split case for $k$, with $e \geq \delta$:

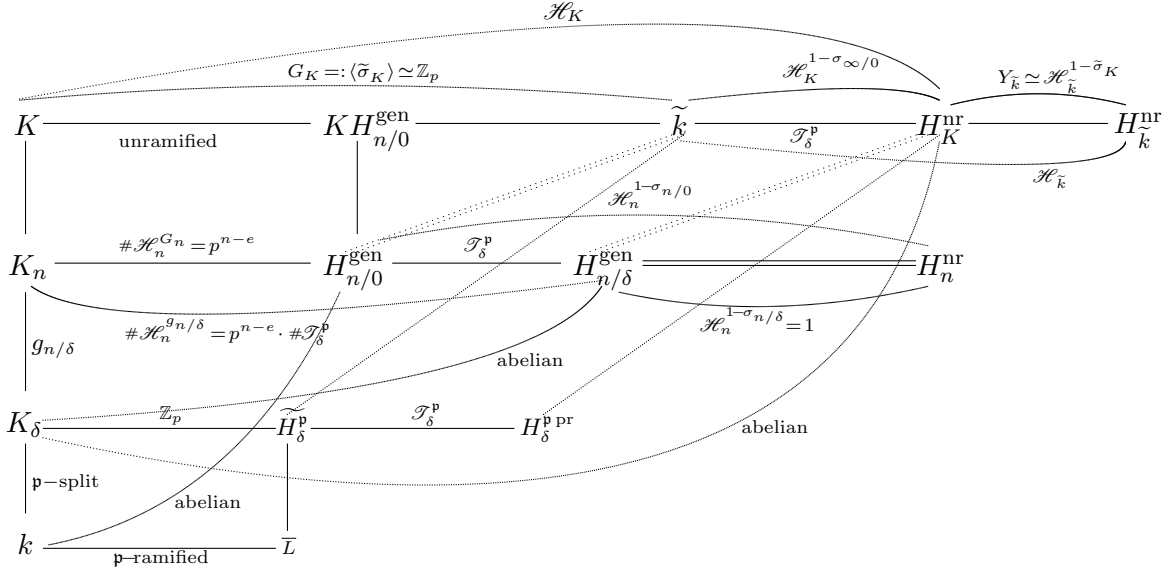
\unitlength=1.54cm 
\begin{equation*}
\begin{aligned}
\vbox{\hbox{\hspace{-10.2cm} 
\begin{picture}(10.0,4.7)
\put(6.6,2.50){\line(1,0){1.45}}
\put(6.0,2.4){$H_{n/0}^\gen$}
\bezier{60}(6.2,2.6)(7.6,3.1)(9.0,3.6)
\bezier{60}(6.3,2.6)(7.7,3.1)(9.1,3.6)
\bezier{60}(8.3,2.6)(9.7,3.1)(11.1,3.6)
\bezier{60}(8.4,2.6)(9.8,3.1)(11.2,3.6)
\put(8.75,2.48){\line(1,0){2.35}}
\put(8.75,2.52){\line(1,0){2.35}}
\put(11.1,2.4){$H_n^\nr$}
\put(8.15,2.4){$H_{n/\delta}^\gen$}
\put(3.3,2.4){$K_n$}
\put(3.7,2.50){\line(1,0){2.2}}
\put(7.2,2.6){\tiny$\CT_\delta^{\mathfrak p}$} 
\put(4.25,2.6){\tiny$\order \CH_n^{G_n}\!=\!p^{n-e}$}
\put(3.45,2.75){\line(0,1){0.8}}
\put(6.3,2.75){\line(0,1){0.8}}
\put(3.35,3.6){$K$}
\bezier{500}(3.4,3.91)(10.1,5.15)(11.3,4.0)
\ft\put(7.9,4.6){$\CH_{K}$}\ns
\put(3.6,3.7){\line(1,0){2.3}}
\put(4.25,3.5){\tiny${\rm unramified}$}
\put(6.0,3.6){$K H_{n/0}^\gen$}
\put(6.8,3.7){\line(1,0){2.1}}
\put(9.0,3.6){$\wt k$}
\put(9.2,3.7){\line(1,0){1.9}}
\put(11.1,3.6){$H_K^\nr$}
\put(11.55,3.7){\line(1,0){1.2}}
\put(12.8,3.62){$H_{\wt k}^\nr$}
\bezier{300}(9.15,3.9)(10.9,4.1)(11.3,3.9)
\bezier{300}(9.08,3.55)(12.8,3.15)(12.9,3.54)
\put(12.1,3.2){\tiny$\CH_{\,\wt k}$}
\put(9.95,4.1){\tiny$\CH_K^{1 - \sigma_{\infty/0}}$}
\put(10.0,3.54){\tiny$\CT_\delta^{\mathfrak p}$}
\bezier{400}(3.4,3.9)(6.15,4.15)(9.0,3.9)
\put(5.7,4.10){\tiny$G_K\!=:\!\langle \wt\sigma_K^{} \rangle\! \simeq\! \Z_p$}
\bezier{300}(11.4,3.86)(12.15,4.05)(12.9,3.86)
\put(11.8,4.05){\tiny$Y_{\wt k} \!\simeq\! \CH_{\,\wt k}^{1-\wt\sigma_K^{}}$}
\ft\put(8.45,3.0){\tiny$\CH_n^{1\!-\!\sigma_{n/0}}$}\ns
\bezier{300}(6.5,2.7)(8.8,3.15)(11.2,2.65)
\ft\put(9.24,1.92){\tiny$\CH_n^{1\!-\!\sigma_{n/\delta}}\! \!=\! 1$}\ns 
\bezier{300}(8.55,2.25)(9.8,1.98)(11.2,2.3)
\put(3.45,1.4){\line(0,1){0.85}}
\put(3.28,1.05){$K_\delta$}
\put(3.45,0.35){\line(0,1){0.6}}
\put(3.38,0.0){$k$}
\put(3.6,1.08){\line(1,0){2.0}}
\put(5.9,1.08){\line(1,0){1.75}}
\ft\put(5.6,1.02){$\wt{H_\delta^{\mathfrak p}}$}\ns
\ft\put(7.7,1.02){${H_\delta^{{\mathfrak p}\,\pr}}$}\ns
\put(6.75,1.16){\tiny$\CT_\delta^{\mathfrak p}$}
\put(4.6,1.16){\tiny$\Z_p$}
\bezier{200}(7.9,1.2)(9.6,2.4)(11.3,3.6)
\bezier{200}(5.7,1.2)(7.4,2.4)(9.1,3.6)
\put(3.6,0.05){\line(1,0){2.0}}
\put(5.65,0.0){\tiny$\ov L$}
\put(5.7,0.20){\line(0,1){0.75}}
\put(4.2,-0.08){\hbox{\tiny${\mathfrak p}{\rm\!-\!ramified}$}}
\bezier{500}(3.6,1.0)(10.5,-0.5)(11.3,3.54)
\bezier{400}(3.6,1.15)(7.8,1.4)(8.4,2.3)
\put(9.6,1.04){\tiny${\rm abelian}$}
\bezier{400}(3.5,2.3)(4.2,1.8)(8.4,2.35)
\ft\put(4.3,1.85){\tiny$\order \CH_n^{g_{n/\delta}}\!=\! 
p^{n-e}\! \cdot \order\! \CT_\delta^{\mathfrak p}$} \ns 
\bezier{300}(3.65,0.08)(5.3,0.4)(6.15,2.25)
\put(7.5,1.62){\tiny${\rm abelian}$}
\put(4.73,0.4){\tiny${\rm abelian}$}
\ft\put(3.5,1.75){$g_{n/\delta}$}\ns
\put(3.5,0.58){\tiny${\mathfrak p}{\rm -split}$}
\end{picture}}} 
\end{aligned}
\end{equation*}
\unitlength=1.0cm

\smallskip
\begin{theorem}\label{C}
Let $K/k$ be a $\Z_p$-extension, $K \ne L$, $\ov L$, and put $K \cap L =: K_e$,
$e \geq \delta$. Let $H_K^\nr$ (resp. $H_{\wt k}^\nr$) be the $p$-Hilbert 
class field of $K$ (resp. $\wt k$), let $\CH_K$ (resp. $\CH_{\,\wt k}$) be the $p$-class group 
of $K$ (resp. $\wt k$). We have the exact sequences, where 
$Y_{\wt k} = \Gal(H_{\wt k}^\nr/H_K^\nr)$:
\begin{equation*}
\left \{\begin{aligned}
1 \to Y_{\wt k} & \too \Gal(H_{\wt k}^\nr/K) \too \Gal(H_K^\nr/K)
\simeq G_K \oplus \CT_\delta^{\mathfrak p} \to 1, \\
1 \to Y_{\wt k}  &\too \Gal(H_{\wt k}^\nr/K_\delta) \too \Gal(H_K^\nr/K_\delta)
\simeq \Z_p \oplus G_K \oplus \CT_\delta^{\mathfrak p} \to 1, 
\end{aligned} \right .
\end{equation*}

\noindent
where $\CT_\delta^{\mathfrak p}$ (of order $\CR_\delta^{\mathfrak p}$) is 
isomorphic to $\Gal(H_K^\nr/\wt k)$. Then $H_K^\nr$ is the genus field of 
$\wt k/K$ whence $Y_{\wt k} \simeq \CH_{\,\wt k}^{1-\wt\sigma_K^{}}$; it is 
also the genus field of $\wt k/K_\delta$.
\end{theorem}

\begin{proof}
It is clear that $\wt k \subseteq H_K^\nr$. Since $k$ is $p$-rational (hence 
$H_k^\pr = \wt k$), the maximal sub-extension of $H_K^\nr$, abelian over $k$, 
is $\wt k$; then $\wt k$ is the genus field of $K/k$ and is fixed by the image of
$\CH_K^{1 - \sigma_{\infty/0}}$, where $\sigma_{\infty/0}$ is a 
topological generator of $\Gal(K/k)$.

\smallskip
For $e \geq \delta$, $\CH_n = \Ccl(S_\delta^{\mathfrak p}) = 
\CH_n^{g_{n/\delta}}$ (Theorem \ref{fundamental}); it follows
$\CH_n^{1-\sigma_{n/\delta}} = 1$, giving $H_{n/\delta}^\gen = H_n^\nr$,
whence $H_{n/0}^\gen \subseteq H_{n/\delta}^\gen$ since $H_{n/0}^\gen 
\subseteq H_n^\nr$ (this justifies the diagram).
Then $\CH_K \simeq \Gal(H_K^\nr/K) \simeq \Z_p \oplus \CT^{\mathfrak p}_\delta$. 
The $p$-Hilbert class field of $K$ is $H_K^\nr = \bigcup_n K H_{n/\delta}^{\gen}$, 
then $H_K^\nr$ is abelian maximal over $K$ and $K_\delta$ (i.e. $H_K^\nr = 
H_{\wt k/K}^\gen = H_{\wt k/K_\delta}^\gen$).
\end{proof}

\subsection{Pseudo-nullity of \texorpdfstring{$\CH_{\,\wt k}$}{Lg}}
Taking genus fields over some $K$'s is used for proving that $\CH_{\,\wt k}$ 
is a pseudo-null $\Lambda$-module (in particular, under the existence of 
a $\Z_p$-extension $K/k$ 
such that $\order (\CH_{\,\wt k}/\CH_{\,\wt k}^{1 - \wt\sigma_K^{}}) < \infty$;
cf. Minardi--Ozaki \cite[\S\,3, Lemma 1; \S\,4, Theorem B]{Oza2001}). 

\smallskip
Of course, if $\delta = 0$, $K_\delta = k$,
$H_K^\nr/k$ is abelian, $H_K^\nr = \wt k$, 
and in that case, $\CT_\delta^{\mathfrak p} = 1$, implying 
$\CH_{\,\wt k} = \CH_{\,\wt k}^{1-\wt\sigma_K^{}}$, whence $Y_{\wt k} = 1$ and
$\CH_{\,\wt k} = 1$. 

\smallskip
In our approach, all the $K$'s such that
$e \geq \delta$ are suitable for the following result proving that $\CH_{\,\wt k}$
is a pseudo-null $\Lambda$-module:

\begin{theorem}\label{pseudo-null}
(i) When $\CH_k = 1$, the $\Lambda$-module $\CH_{\,\wt k}$ is pseudo-null in the 
meaning that there exists a $\Z_p$-extension $K_\0$ of $k$, distinct from $L, \ov L$, 
such that $(\CH_{\,\wt k})_{G_\0}^{} := \CH_{\,\wt k}/\CH_{\,\wt k}^{1 - \wt\sigma_\0^{}}$ is 
finite of order $p^{t_\delta}$ (see Theorem \ref{Ptorsion}\,(i)), where $\wt\sigma_\0^{}$ is 
a generator of $G_\0 := \Gal(\wt k/K_\0)$. 

\smallskip
(ii) The $\Z_p[G_\0]$-module $\CH_{\,\wt k}$ is of finite type.
\end{theorem}

\begin{proof}
In some sense, the study of the previous diagram and Theorem \ref{C} give the result; 
but we may provide some observations with an independent proof:

\medskip
(i) Let $K_\0 \ne L, \ov L$ be a non-Galois $\Z_p$-extension such that 
$\lambda(K_\0/k) = 1$, $\mu(K_\0/k) = 0$; such an extension does exist from Theorem 
\ref{fundamental} as soon as we take $e \geq \delta$. In that case:
$$\order \CH_{\0,n} := \order \CH_{K_{\0,n}} = p^{n-e+t_\delta}\ \hbox{(Theorem 
\ref{fundamental})}. $$ 

Since $K_\0 \ov {K_\0} = \wt k$ and $K_\0 \cap \ov {K_\0} = k$ 
($e \geq \delta$ is non-zero and $K_\0 \cap \ov {K_\0}/k$ is Galois), 
let's consider the extensions:
$$F_n := K_{\0,n} \ov{K_{\0}}_{,n-e}, $$ 

\noindent
for which $\bigcup_n F_n = \wt k$. 
One verifies that $F_n/K_{\0,n}$ (of degree $p^{n-e}$) is unramified for $n$ large 
enough (indeed, inertia groups of ${\mathfrak p}$ and $\ov {\mathfrak p}$ can not 
intersect $\Gal(F_n/K_{\0,n})$); we note that in $F_n/\ov{K_{\0}}_{,n-e}$, the
inertia subgroup of ${\ov {\mathfrak p}}$ is of order $p^e$, hence of order $p$
in $F_{n+1}/F_n$ (comes from its total ramification in $K_{\0,n}/k$ and its 
non-ramification in $\ov{K_{\0}}_{,e}/k$). 
The inertia subgroup of ${\mathfrak p}$ in $F_{n+1}/F_n$, of order $p$, 
is disjoint from that of ${\ov {\mathfrak p}}$.
We deduce from this that $F_{n+1}/F_n$ of degree $p^2$ does not 
contain a non trivial unramified sub-extension.

\smallskip
The group $G_\0$ projects on $\Gal(F_n/K_{\0,n})$
by restriction. Since $\order \CH_{\0,n} = p^{n-e+t_\delta}$, the 
Chevalley--Herbrand formula in the unramified extension $F_n/K_{\0,n}$ is:
\begin{equation*} 
\order (\CH_{F_n}^{G_\0}) = \order \CH_{\0,n} \times \ffrac{1}{p^{n-e}}
= p^{n-e+t_\delta} \times \ffrac{1}{p^{n-e}} = p^{t_\delta}.
\end{equation*}

We deduce that $\order( \CH_{F_n}/\CH_{F_n}^{1 - \wt\sigma_\0^{}}) = p^{t_\delta}$ 
for all $n \gg 0$, where $p^{t_\delta} = \order \CT^{\mathfrak p}_\delta$ is an absolute 
constant (it is computed in $K_{\0,\delta}$ which is nothing else than $L_\delta$,
independent of the choice of $K_\0$). Thus, for $m \geq n \gg 0$, $\BN_{F_m/F_n}
(\CH_{F_m}) = \CH_{F_n}$, which defines the isomorphisms:
$$\CH_{F_m}/\CH_{F_m}^{1-\wt\sigma_\0^{}} \too \CH_{F_n}/\CH_{F_n}^{1-\wt\sigma_\0^{}}. $$

Whence $\order\big(\CH_{\,\wt k}/\CH_{\,\wt k}^{1 - \wt\sigma_\0^{}} \big) = p^{t_\delta}$.

\medskip
(ii) There exists a finite group $C_n \subseteq \CH_{F_n}$, 
such that $\CH_{F_n} = \CH_{F_n}^{1 - \wt\sigma_\0^{}} 
\cdot C_n$. Thus, for all $N$, $\CH_{F_n} = \CH_{F_n}^{(1 - \wt\sigma_\0^{})^{p^N}}\!\! 
\cdot \langle C_n \rangle_{\Z[G_\0]}$. For all $N'$, there exists $N$ such that, 
$(1 - \wt\sigma_{\0,n})^{p^N} \in p^{N'}\Z[\Gal(F_n/K_{\0,n})]$; taking for $p^{N'}$ 
the exponent of $\CH_{F_n}$, we get, for all $n$:
$$\CH_{F_n} = \langle C_n \rangle_{\Z[G_\0]}, $$

\noindent
of finite type as $\Z[G_\0]$-module; thus, $\CH_{\,\wt k} = \ds \limproj_n \CH_{F_n}$ is a 
$\Z_p[G_\0]$-module of finite type, with $r_\delta^{}$~generators, where $r_\delta^{}$ 
is the $p$-rank of $\CT_\delta^{\mathfrak p}$ (see \S\,\ref{torsion} for some numerical 
examples). 

\smallskip
One finds again that if $\delta = 0$, then $\CT_\delta^{\mathfrak p} = 1$, thus
$r_\delta^{} = 0$.

\smallskip
Note that we have $\Gal(H_{K_\0}^\nr/\wt k) \simeq \tor_{\Z_p}^{}
(\Gal(H_{\0, \delta}^{{\mathfrak p}\,\pr}/K_{\0, \delta})) = \CT_\delta^{\mathfrak p}$ and 
$\CC_k \simeq \CT_\delta^{\mathfrak p}$; ndeed, from $K_\0 \ov L = \wt k$ and 
$K_\0 H_{\0, \delta}^{{\mathfrak p}\,\pr} \subseteq H_{K_\0}^\nr$, the result follows.
\end{proof}

\section{A characterization in the case \texorpdfstring{$e < \delta$}{Lg}}\label{e<delta}

We will give some properties of the norm residue symbols of $S$-units and apply 
them to the $S_n^{\mathfrak p}$-class groups. 
We assume $e < \delta$ (whence $\delta' = e$, $\spp = p^e$). So ${\mathfrak p}$ 
totally splits in $K_e/k$ and totally ramifies in $K/K_e$; $\ov {\mathfrak p}$ 
totally ramifies (this refers to the second case of Schema \eqref{Tpsplit2}).

\smallskip
But equalities of the form $\CH_n = \Ccl_n (T_n)$, $T_n \in \{S_n, 
S_n^{\mathfrak p}, S_n^{\ov {\mathfrak p}}\}$,
do not hold since Chevalley--Herbrand Formulas \eqref{CH3} give
$\order (\CH_n^{S_n})^{G_n} = 
\order (\CH_n^{S_n^{\mathfrak p}})^{G_n} = p^{\delta - e} \ne 1$ and
$\order (\CH_n^{S_n^{\ov {\mathfrak p}}})^{G_n} = p^\delta \ne 1$.  

\smallskip
In any case, $\CH_n^{G_n} = \CH_n^\ram = 
\Ccl_n (\langle\CP,\CP_\0 \rangle)$ of oder $p^{n-e}$ (Formula \eqref{CH1}); 
so, the inclusion $\Ccl_n (\langle\CP,\CP_\0 \rangle) \subseteq \Ccl_n (S_n)$
shows that $\order \Ccl_n (S_n)$ is unbounded but divides $[p^{n-e}]^\spp$
since, in $K_e$, each ${\mathfrak p}_i \mid {\mathfrak p}$ is $p$-principal 
and totally ramifies in $K_n/K_e$. 
In the case $e = 0$,  $\order \Ccl_n (S_n) = p^n$.
An important result will be the estimation of $\order \Ccl_n(S_n^{\mathfrak p})$ 
(inequalities \eqref{orderCl(S)}, for arbitrary $e < \delta$).

\subsection{Properties of the symbols of the \texorpdfstring{$S_e$-units}{Lg}
in \texorpdfstring{$K_n/K_e$}{Lg}}\label{normx}
The primes ${\mathfrak p}_{1}, \ldots, {\mathfrak p}_{\spp}\!\! \mid\! {\mathfrak p}$, 
${\mathfrak p}_{0}\! \mid \ov{\mathfrak p}$, of $K_e$, are 
$p$-principal of the form\,\footnote{\, As we explained, one must read 
${\mathfrak p}_i^{\he} =: (x_i)$, for the prime-to-$p$ class number 
$\he$ of $K_e$}
$(x_1), \ldots, (x_{\spp}), (x_\0)$, $x_j \in K_e^\times$. Since $G_e$ is cyclic 
and the ideals ${\mathfrak p}_i$ conjugate for $i \in [1, \spp]$, one may suppose 
that, for a generator $\sigma_e := \sigma_{e/0}$ of $G_e$: 
$$x_{k+1} = x_1^{\sigma_e^k },\ \, k \in [0, \spp - 1], \ \ 
x_\0^{\sigma_e} = x_\0^{} \cdot \eta_\0^{},\ \eta_\0^{} \in E_e,$$
which defines in the first case a circular permutation of the $x_k$'s with 
$k \in \Z/p^{\spp} \Z$, and the unit $\eta_\0^{}$ in the second one since
${\mathfrak p}_{0}$ is invariant. We put:
\begin{equation*}
X_e^{\mathfrak p} := \{x_1, \ldots, x_{\spp} \}, \ \ 
X_e^{\ov{\mathfrak p}}  := \{x_\0 \}, \ \ 
X_e := X_e^{\mathfrak p}  \cup X_e^{\ov{\mathfrak p}} .
\end{equation*}

Since the residue degrees of the ideals ${\mathfrak p}_i$ are $1$ in $K_e/k$, 
we have $\BN_{e/0}(x_j) = x$, for $j \in [1, \spp]$ and $\BN_{e/0}(x_\0) = \ov x$.

\smallskip
We consider the symbols $\big(\ffrac{x_j,K_n/K_e}{{\mathfrak p}_i} \big) \in g_{n/e}$, 
for $i, j \in [1, \spp] \cup \{0\}$.

\smallskip
Since $G_n$ is abelian, we have:
\ft\begin{equation*}
\left \{\begin{aligned}
\Big(\ffrac{x_j,K_n/K_e}{{\mathfrak p}_i} \Big)^{\sigma_e}  
& = \Big(\ffrac{x_j,K_n/K_e}{{\mathfrak p}_i} \Big) = 
\Big(\ffrac{x_j^{\sigma_e},K_n/K_e}{{\mathfrak p}_i^{\sigma_e}} \Big)
= \Big(\ffrac{x_{j+1},K_n/K_e}{{\mathfrak p}_{i+1}}\Big),\  i, j \in [1, \spp], \\
\Big(\ffrac{x_j,K_n/K_e}{{\mathfrak p}_\0} \Big)^{\sigma_e}  
& = \Big(\ffrac{x_j,K_n/K_e}{{\mathfrak p}_\0} \Big) = 
\Big(\ffrac{x_j^{\sigma_e},K_n/K_e}{{\mathfrak p}_\0} \Big)
=  \Big(\ffrac{x_{j+1},K_n/K_e}{{\mathfrak p}_\0} \Big),\  j \in [1, \spp], \\
\Big(\ffrac{x_\0,K_n/K_e}{{\mathfrak p}_i} \Big)^{\sigma_e} 
& = \Big(\ffrac{x_\0,K_n/K_e}{{\mathfrak p}_i} \Big) = 
\Big(\ffrac{x_\0^{} \cdot \eta_\0^{},K_n/K_e}
{{\mathfrak p}_{i+1}} \Big),\  i \in [1, \spp], \\
\Big(\ffrac{x_\0,K_n/K_e}{{\mathfrak p}_\0} \Big)^{\sigma_e} 
& = \Big(\ffrac{x_\0,K_n/K_e}{{\mathfrak p}_\0} \Big) = 
\Big(\ffrac{x_\0^{} \cdot \eta_\0^{},K_n/K_e}
{{\mathfrak p}_\0} \Big).
\end{aligned} \right.
\end{equation*}\ns

In particular, all the symbols $\big(\ffrac{x_{j+k},K_n/K_e}{{\mathfrak p}_{i+k}} \big)$,
$k \in [1,\spp]$, are equal, $\big(\ffrac{x_j,K_n/K_e}{{\mathfrak p}_\0^{}} \big)$ is 
independent of $j \in [1, \spp]$, and $\big(\ffrac{\eta_\0^{}, K_n/K_e}{{\mathfrak p}_\0^{}} \big) =1$. 

From Lemma \ref{functorial}\,(iv), we obtain the following 
images in $G_n$, for each $x_j \in X_e^{\mathfrak p}$, then 
for $x_\0 \in X_e^{\ov{\mathfrak p}}$
(recall that notations without indices are related to $k$, as
${\mathfrak p}$, $\ov {\mathfrak p}$, $x$, $\ov x$):
\ft\begin{equation*}
\left \{\begin{aligned}
\prd_{k = 1}^{\spp}
\Big(\ffrac{x_j,K_n/K_e}{{\mathfrak p}_k} \Big) 
& \too \Big(\ffrac{\BN_{e/0}(x_j),K_n/k}{{\mathfrak p}} \Big)
= \Big(\ffrac{x,K_n/k}{{\mathfrak p}} \Big),\  j \in [1, \spp], \\
\Big(\ffrac{x_j,,K_n/K_e}{{\mathfrak p}_\0^{}} \Big)
& \too \Big(\ffrac{\BN_{e/0}(x_j),K_n/k}{\ov {\mathfrak p}} \Big)
= \Big(\ffrac{ x, K_n/k}{\ov {\mathfrak p}} \Big), \  j \in [1, \spp], \\
\prd_{k = 1}^{\spp}
\Big(\ffrac{x_\0,K_n/K_e}{{\mathfrak p}_k} \Big) 
& \too \Big(\ffrac{\BN_{e/0}(x_\0),K_n/k}{{\mathfrak p}} \Big)
= \Big(\ffrac{\ov x,K_n/k}{{\mathfrak p}} \Big), \\
\Big(\ffrac{x_\0,K_n/K_e}{{\mathfrak p}_\0^{}} \Big)
& \too \Big(\ffrac{\BN_{e/0}(x_\0),K_n/k}{\ov {\mathfrak p}} \Big)
 = \Big(\ffrac{\ov x, K_n/k}{\ov {\mathfrak p}} \Big). 
\end{aligned} \right.
\end{equation*}\ns

Product formula gives
$\big(\ffrac{x, K_n/k}{\ov {\mathfrak p}} \big) \!=\!
\big(\ffrac{x, K_n/k}{{\mathfrak p}} \big)^{-1}$, \!
$\big(\ffrac{\ov x, K_n/k}{\ov {\mathfrak p}} \big)\! =\!
\big(\ffrac{\ov x, K_n/k}{{\mathfrak p}} \big)^{-1}$. 

\smallskip
It follows, from Proposition \ref{mainorder}, that:
\begin{equation*}
\left \{\begin{aligned}
\prd_{k = 1}^{\spp} & \big( \ffrac{x_j,K_n/K_e}{{\mathfrak p}_k} \big)
\ \ \&\ \ \big(\ffrac{x_j,K_n/K_e}{{\mathfrak p}_\0^{}} \big)
\hbox{ are of order $p^{n  - \delta}$, for $j \in [1, \spp]$,} \\
\prd_{k = 1}^{\spp}
&\big( \ffrac{x_\0,K_n/K_e}{{\mathfrak p}_k} \big)  
\ \ \&\ \ \big(\ffrac{x_\0,K_n/K_e}{{\mathfrak p}_\0^{}} \big)
\hbox{ are of order $p^{n  - e - \delta}$}.
\end{aligned}\right.
\end{equation*}

Applying Lemma \ref{functorial}\,(iv) to units $\varepsilon \in E_e$ leads,
since $\BN_{e/0}(\varepsilon) = 1$, to:
$$\prd_{k = 1}^{\spp} \Big(\ffrac{\varepsilon,K_n/K_e}{{\mathfrak p}_k} \Big) = 1, \ \ 
\Big(\ffrac{\varepsilon,K_n/K_e}{{\mathfrak p}_\0^{}} \Big) = 1. $$

We can write, instead, for any $y \in K_e^\times$,
$\big(\ffrac{y, K_n/K_e}{{\mathfrak p}_k} \big) =
\big(\ffrac{y, K_n/K_e}{{\mathfrak p}_k} \big)^{\sigma_e^{1-k}} =
\big(\ffrac{y^{\sigma_e^{1-k}} , K_n/K_e}{{\mathfrak p}_1}\big)$;
hence the formulas
$\prod_{k = 1}^{\spp} \big(\ffrac{y,K_n/K_e}{{\mathfrak p}_k} \big) 
= \big(\ffrac{\BN_{e/0}(y),K_n/K_e}{{\mathfrak p}} \big)\ \ \& \ \ 
\big(\ffrac{y,K_n/K_e}{{\mathfrak p}_\0} \big) =  
\big(\ffrac{\BN_{e/0}(y),K_n/K_e}{\ov {\mathfrak p}} \big)$. 

\smallskip
Since $e < \delta$, the decomposition groups, in $K_n/K_e$, of the ${\mathfrak p}_k$,
$k \in [1, \spp]$, coincide with the inertia groups $g_{n/e}$; thus, by definition 
of $\omega_{n/e} :X_e \too \wt \Omega_{n/e}^{\,\mathfrak p} \simeq 
\bigoplus_{i =1}^{\spp} g_{n/e} \simeq (\Z/p^{n-e}\Z)^\spp$:
\begin{equation*}
\omega_{n/e} (x_j) = \Big(\big(\ffrac{x_j,K_n/K_e}{{\mathfrak p}_1} \big), \ldots,
\big(\ffrac{x_j,K_n/K_e}{{\mathfrak p}_{\spp}} \big) \Big),\  j \in [1, \spp] \cup \{0\}.
\end{equation*}

\subsection{Introduction of a governing matrix}\label{matrix}
To determine kernel and image of the map $\omega_{n/e} : E_e^{S_e} \to 
\wt \Omega_{n/e}^{\,\mathfrak p}$ when $e < \delta$, we try to 
determine the rank and the $p$-adic properties of the following matrix
$\CM_{\infty/e}$, where the $\spp+ 1 + \re$ lines correspond to 
the images by $\omega_{\infty/e}$ of the $x_j$'s, $j \in [1,\spp] \cup \{0\}$
(sub-matrices $\CM_X^{\mathfrak p}$, $\CM_X^{\ov{\mathfrak p}}$), 
then to the $\varepsilon_k$'s, $k \in [1,\re]$ (sub-matrice $\CM_E$). 

\smallskip
From the product formula, the column of the symbols at 
${\mathfrak p}_\0$ may be suppressed. 

\smallskip
The fact that 
$\wt \Omega_{\infty/e}^{\,\mathfrak p} \simeq (g_{\infty/e})^\spp$ 
supposes that we consider the symbols $\big(\ffrac{\mb, K/K_e}
{{\mathfrak p}_i} \big)$, $i \in [1,\spp]$. In the coefficients of the 
matrix, we omit the mention of the extension $K/K_e$.

\ft$$\CM_{\infty/e} := 
\begin{pmatrix}
\CM_X^{\mathfrak p} :=
\begin{pmatrix}
\big(\ffrac{x_1}{{\mathfrak p}_1}\big) & \!\!\!\!\ldots\!\!\!\! &
 \big(\ffrac{x_1}{{\mathfrak p}_j}\big) & \!\!\!\!\ldots\!\!\!\! &
\big(\ffrac{x_1}{{\mathfrak p}_{\spp-1}}\big) & 
\!\!\!\!  \big(\ffrac{x_1}{{\mathfrak p}_\spp}\big)  \\
\vdots  && \vdots && \vdots \\
\big(\ffrac{x_j}{{\mathfrak p}_1}\big)  & \!\!\!\!\ldots\!\!\!\! &
 \big(\ffrac{x_j}{{\mathfrak p}_j}\big) & \!\!\!\!\ldots\!\!\!\! &
\big(\ffrac{x_j}{{\mathfrak p}_{\spp-1}}\big) &
\!\!\!\!  \big(\ffrac{x_j}{{\mathfrak p}_\spp}\big)  \\
\vdots  && \vdots && \vdots \\
\big(\ffrac{x_{\spp}}{{\mathfrak p}_1}\big)  & \!\!\!\!\ldots\!\!\!\! &
 \big(\ffrac{x_{\spp}}{{\mathfrak p}_j}\big) & \!\!\!\!\ldots\!\!\!\! &
\big(\ffrac{x_{\spp}}{{\mathfrak p}_{\spp-1}}\big)&
\!\!\!\!  \big(\ffrac{x_\spp}{{\mathfrak p}_\spp}\big) 
\end{pmatrix} \sim
\begin{pmatrix}
\big(\ffrac{x_1}{{\mathfrak p}_1}\big) & \!\!\!\!\ldots\!\!\!\! &
 \big(\ffrac{x_1}{{\mathfrak p}_j}\big) & \!\!\!\!\ldots\!\!\!\! &
\big(\ffrac{x_1}{{\mathfrak p}_{\spp-1}}\big) & 
\!\!\!\!  \big(\ffrac{x}{{\mathfrak p}}\big)  \\
\vdots  && \vdots && \vdots \\
\big(\ffrac{x_j}{{\mathfrak p}_1}\big)  & \!\!\!\!\ldots\!\!\!\! &
 \big(\ffrac{x_j}{{\mathfrak p}_j}\big) & \!\!\!\!\ldots\!\!\!\! &
\big(\ffrac{x_j}{{\mathfrak p}_{\spp-1}}\big) &
\!\!\!\!  \big(\ffrac{x}{{\mathfrak p}}\big)  \\
\vdots  && \vdots && \vdots \\
\big(\ffrac{x_{\spp}}{{\mathfrak p}_1}\big)  & \!\!\!\!\ldots\!\!\!\! &
 \big(\ffrac{x_{\spp}}{{\mathfrak p}_j}\big) & \!\!\!\!\ldots\!\!\!\! &
\big(\ffrac{x_{\spp}}{{\mathfrak p}_{\spp-1}}\big)&
\!\!\!\!  \big(\ffrac{x}{{\mathfrak p}}\big) 
\end{pmatrix} \\ \\
\CM_X^{\ov{\mathfrak p}} := 
 \begin{pmatrix}
\, \big(\ffrac{x_\0}{{\mathfrak p}_1}\big)  & \!\!\!\!\ldots\!\!\!\! &
 \big(\ffrac{x_\0}{{\mathfrak p}_j}\big) & \!\!\!\!\ldots\!\!\!\! &
\big(\ffrac{x_\0}{{\mathfrak p}_{\spp-1}}\big) &
\!\!  \big(\ffrac{x_\0}{{\mathfrak p}_\spp}\big)\,   \\ 
\end{pmatrix}\   \sim \ 
 \begin{pmatrix}
\big(\ffrac{x_\0}{{\mathfrak p}_1}\big)  & \!\!\!\!\ldots\!\!\!\! &
 \big(\ffrac{x_\0}{{\mathfrak p}_j}\big) & \!\!\!\!\ldots\!\!\!\! &
\big(\ffrac{x_\0}{{\mathfrak p}_{\spp-1}}\big) &
\!\!  \big(\ffrac{\ov x}{{\mathfrak p}}\big)   \\ 
\end{pmatrix} \\  \\
\CM_E := \,
\begin{pmatrix}
\big(\ffrac{\varepsilon_1}{{\mathfrak p}_1}\big) & \!\!\!\!\ldots\!\!\!\! &
 \big(\ffrac{\varepsilon_1}{{\mathfrak p}_j}\big) & \!\!\!\!\ldots\!\!\!\! &
\big(\ffrac{\varepsilon_1}{{\mathfrak p}_{\spp-1}}\big) &
\!\!\!\!  \big(\ffrac{\varepsilon_1}{{\mathfrak p}_\spp}\big) \\
\vdots  && \vdots && \vdots \\
\big(\ffrac{\varepsilon_j}{{\mathfrak p}_1}\big)  & \!\!\!\!\ldots\!\!\!\! &
 \big(\ffrac{\varepsilon_j}{{\mathfrak p}_j}\big) & \!\!\!\!\ldots\!\!\!\! &
\big(\ffrac{\varepsilon_j}{{\mathfrak p}_{\spp-1}}\big) &
\!\!\!\!  \big(\ffrac{\varepsilon_j}{{\mathfrak p}_\spp}\big)  \\
\vdots  && \vdots && \vdots \\
\big(\ffrac{\varepsilon_{\re}}{{\mathfrak p}_1}\big)  & \!\!\!\!\ldots\!\!\!\! &
 \big(\ffrac{\varepsilon_{\re}}{{\mathfrak p}_j}\big) & \!\!\!\!\ldots\!\!\!\! &
\big(\ffrac{\varepsilon_{\re}}{{\mathfrak p}_{\spp-1}}\big) &
\!\!\!\!  \big(\ffrac{\varepsilon_\re}{{\mathfrak p}_\spp}\big) \\
\end{pmatrix} \sim \
\begin{pmatrix}
\big(\ffrac{\varepsilon_1}{{\mathfrak p}_1}\big) & \!\!\!\!\ldots\!\!\!\! &
 \big(\ffrac{\varepsilon_1}{{\mathfrak p}_j}\big) & \!\!\!\!\ldots\!\!\!\! &
\big(\ffrac{\varepsilon_1}{{\mathfrak p}_{\spp-1}}\big) &  1\ \  \\
\vdots  && \vdots && \vdots \\
\big(\ffrac{\varepsilon_j}{{\mathfrak p}_1}\big)  & \!\!\!\!\ldots\!\!\!\! &
 \big(\ffrac{\varepsilon_j}{{\mathfrak p}_j}\big) & \!\!\!\!\ldots\!\!\!\! &
\big(\ffrac{\varepsilon_j}{{\mathfrak p}_{\spp-1}}\big) &  1\ \   \\
\vdots  && \vdots && \vdots \\
\big(\ffrac{\varepsilon_{\re}}{{\mathfrak p}_1}\big)  & \!\!\!\!\ldots\!\!\!\! &
 \big(\ffrac{\varepsilon_{\re}}{{\mathfrak p}_j}\big) & \!\!\!\!\ldots\!\!\!\! &
\big(\ffrac{\varepsilon_{\re}}{{\mathfrak p}_{\spp-1}}\big) &  1\ \   \\
\end{pmatrix}
\end{pmatrix}$$\ns

\smallskip
We denote by $\CD_X^{\mathfrak p}$ the determinant of the square
sub-matrix $\CM_X^{\mathfrak p}$.

\begin{definitions}
Since $n$ can be arbitrary large, we work in the projective limits, defining 
$g_{\infty/e} := \Gal(K/K_e)$, generated by a topological generator
$\sigma_{\infty/e}$, the symbols $\big(\ffrac{\mb,K/K_e}{{\mathfrak p}_i} \big)$
being of the form $\sigma_{\infty/e}^{\theta_i}$, $\theta_i \in \Z_p$. 
This makes sense because, by restriction to $K_n$, the image of 
$\big(\ffrac{\mb,K/K_e}{{\mathfrak p}_i} \big) = \sigma_{\infty/e}^{\theta_i}$ 
is $\big(\ffrac{\mb,K_n/K_e}{{\mathfrak p}_i} \big) = \sigma_{n/e}^{\theta_i}$
(Lemma \ref{functorial}\,(ii)). 

The corresponding maps 
$\omega_{n/e}$ give rise to $\omega_{\infty/e} : E_e^{S_e} \to 
\wt \Omega_{\infty/e}^{\,\mathfrak p} \simeq \Z_p^\spp$ that we extend to 
the vectorial space map $\omega_{\infty/e} : E_e^{S_e} \otimes \Q_p \to 
\wt \Omega_{\infty/e}^{\,\mathfrak p} \otimes \Q_p \simeq \Q_p^\spp$.

\smallskip
We will say, classically, that $\CM_{\infty/e}$ is of $\Q_p$-rank $\rho_\infty^{}$ 
if there exists a square sub-matrix invertible over $\Q_p$ of maximal 
dimension $\rho_\infty^{}$ (i.e. with a non-zero determinant).
We will say that $\CM_{\infty/e}$ is of arithmetic rank $\rho_p$ if there
exists a square sub-matrix invertible over $\Z_p$ of maximal dimension 
$\rho_p$ (i.e. with a determinant in $\Z_p^\times$). 

\smallskip
In an arithmetic context, $\CM_{\infty/e}$ may be of $\Q_p$-rank $\spp$
but of arithmetic rank less than~$\spp$. This is characterized by 
computing its list of invariants of similitude $(q_1, \ldots , q_{\rho_\infty^{}}, 0, \ldots , 0)$, 
of the form $q_1 \mid q_2 \mid \cdots \mid q_{\rho_\infty^{}}$ with $q_1 = \cdots =
q_{\rho_p} = 1$, then $q_{\rho_p+1} \equiv \cdots \equiv q_{\rho_\infty^{}}\equiv 0 
\pmod p$, and $q_{\rho_\infty+1} = \cdots = q_{\spp} = 0$.

\smallskip
As explained above, we see these matrices and determinants in an additive way;
for the matrix $\CM_X^{\mathfrak p}$, we put $\big(\ffrac{x_j, \, K/K_e}{{\mathfrak p}_i}\big) 
=: \sigma_{\infty/e}^{\theta_i(x_j)}$, $\theta_i(x_j) \in \Z_p$, $i, j \in [1, \spp]$, and so on 
for $\CM_X^{\ov {\mathfrak p}}$ and $\CM_E$; then it suffices to consider the matrix of 
the exponents $\theta_i(x_j)$.
\end{definitions}

If $e = 0$, $\spp = 1$, $\re = 0$, $S_e = S_k$, $S_e^{\mathfrak p} = 
\{{\mathfrak p}\}$, $S_e^{\ov{\mathfrak p}} = \{{\ov{\mathfrak p}}\}$, 
$\wt \Omega_{\infty/e}^{\,\mathfrak p} \simeq 
G_{\infty/0}$, the group of $S_k$-units is reduced to $\langle x, \ov x \rangle$, with 
$\omega_{n/0} (E_k^{S_k})$ of order $p^{n - \delta}$,  $\CM_X^{\mathfrak p} 
= \big( (\frac{x, K/k}{{\mathfrak p}}) \big)$, with $(\frac{x, K/k}{{\mathfrak p}}) =
\sigma_{\infty/0}^{p^\delta}$, whence $\theta(x) = p^\delta$, and the map 
$\omega_{\infty/0}$ is not surjective as soon as $\delta \geq 1$. 

\smallskip
Using relations given in \S\,\ref{normx}, we can replace the symbols of the 
last column of $\CM_X^{\mathfrak p}$ by the symbols $\Big(\ffrac{x,K/K_e}
{{\mathfrak p}} \Big)$, then $\Big(\ffrac{\ov x,K/K_e}{{\mathfrak p}} \Big)$ for 
the last column of $\CM_X^{\ov{\mathfrak p}}$, and $1$ for $\CM_E$; this 
does not depend on the choice of $x_1$ modulo $E_e$, with $x_{k+1} = 
x_1^{\sigma_e^k },\ \, k \in [0, \spp - 1]$, since $\BN_{e/0}(E_e) = 1$.

\begin{remark}\label{independence}
In practice, the numbers $\theta_i(x_j)$ depend 
on Fermat's quotients of the form $p^{\delta_{{\mathfrak p}_i}(x_j)}$ that 
may be given by $p$-adic logarithms. More precisely, Lemma \ref{delta}
applies to $x_j$ and ${\mathfrak p}_i$ for $i \ne j$ giving an identity in $\Z_p$
of the form $\log_{{\mathfrak p}_i}(x_j) = p^{\delta_{{\mathfrak p}_i}(x_j)}
\cdot u_{i,j}$, $u_{i,j} \in \Z_p^\times$.

Let $u_i \in \CU_e^{{\mathfrak p}_i} = 1 + {\mathfrak p}_i \simeq \Z_p$ 
be such that $\big(\ffrac{u_i, K/K_e}{{\mathfrak p}_i}\big) = \sigma_{\infty/e}$;
then $\theta(x_j) = \ffrac{\log_{{\mathfrak p}_i}(x_j)}{\log_i(u_i)}$. 
\end{remark}

\begin{lemma}\label{theta}
(i) For $e$ fixed, the matrix $\CM_X^{\mathfrak p}$ is independent of the
$\Z_p$-extension $K/k$, under the sole condition $K_e := K \cap L = L_e$
(or $u \in 1+ p^e \,\Z_p^\times$ in the parametrization of the $K$'s).

\smallskip
(ii) The matrix $\CM_X^{\mathfrak p}$ is of $\Q_p$-rank $\rho_\infty^{} = \spp$ but
of arithmetic rank $\rho_p \leq \spp - 1$.
\end{lemma}

\begin{proof}
(i) It suffices to show that the computation of the symbols only depends
on the local arithmetic of $L_e$ and more precisely of the images
of the $S_e^{\mathfrak p}$-units in $\CU_e^{\mathfrak p}$. This may be 
seen as follows. Consider the symbol
$\big(\ffrac{x_j, K/K_e}{{\mathfrak p}_i}\big)$, for $j \ne i$, $i, j \in [1, \spp]$
(since $x_j$ is a unit at ${\mathfrak p}_i$, the symbol is in $g_{\infty/e}$); its
computation is justified in Remark \ref{independence}.
The particular computation of $\big(\ffrac{x_i, K/K_e}{{\mathfrak p}_i}\big)$, 
$i \in [1, \spp]$, comes from the specific product formulas giving, in $G_{\infty/0}$,
$\big(\ffrac{x_i, K/K_e}{{\mathfrak p}_i}\big) = \big(\ffrac{x, K/k}{{\mathfrak p}}\big)
\cdot \prod_{k \ne i}\big(\ffrac{x_i, K/K_e}{{\mathfrak p}_k}\big)^{-1}$, where
$\big(\ffrac{x, K/k}{{\mathfrak p}}\big)$ is known (indeed, $\big(\ffrac{x, K_n/k}
{{\mathfrak p}}\big)$ is of order $p^{n-\delta}$, whence an element of $g_{\infty/e}$
since $e < \delta$).

\smallskip
(ii) Consider the matrix $\CM_X^{\mathfrak p}$ in its original form (left 
part of $\CM_{\infty/e}$). This is a circular matrix since $\big(\ffrac{x_j, K/K_e}
{{\mathfrak p}_i}\big)^{\sigma_e} = \big(\ffrac{x_j, K/K_e}{{\mathfrak p}_i}\big) =
\big(\ffrac{x_{j+1}, K/K_e}{{\mathfrak p}_{i+1}}\big)$ for all $i, j \in [1, \spp]$; its 
determinant $\CD_X^{\mathfrak p}$ is a product of linear forms as follows. 

\smallskip
The additive writing is obtained from the symbols of a fixed conjugate $x_1$, 
by means of the definition $\big(\ffrac{x_1, K/K_e}{{\mathfrak p}_i}\big) =: 
\sigma_{\infty/e}^{\theta_i}$, $\theta_i\in \Z_p$, $i \in [1, \spp]$; then 
$\big(\ffrac{x_j, K/K_e}{{\mathfrak p}_k}\big) = \sigma_{\infty/e}^{\theta_{k-j+1}}$, 
$j, k \in [1, \spp]$ (see Remark \ref{independence} and Lemma \ref{independence}). 
Thus, these linear forms are:
$$L_k = \sm_{i=1}^\spp \theta_i \cdot \xi_\spp^{i k} \in \Q_p[\xi_\spp], 
\ \, k\in [1, \spp], $$

\noindent
with $\xi_\spp$ of order $\spp$ .
These are non-zero by independence of the 
logarithms over an algebraic closure of $\Q$, under the Leopoldt and Gross--Kuz'min 
conjectures in $K_e$. Since the $\Z_p[G_e]$-module $X_e$ is monogenic, one may
also refer to \cite[Section 5]{Jau1985}.
\end{proof}

The determinant of $\CM_X^{\mathfrak p}$ is:
\begin{equation}\label{DetX}
\CD_X^{\mathfrak p} \sim \prd_{k = 1}^\spp
\sm_{i=1}^\spp \theta_i \cdot \xi_\spp^{i k} =: p^{\Delta_e^{\mathfrak p}}.
\end{equation}
 
Since the elements of the last column of the sub-matrix $\CM_X^{\mathfrak p}$ 
in $\CM_{n/e}$ (right version) are of orders less than or equal to $p^{n-\delta}$ (giving 
exponents $\theta \equiv 0 \pmod {p^{\delta-e}}$), any square sub-matrix of 
$\CM_{\infty/e}$ containing  the corresponding part of this column is non-invertible 
in $\Z_p$. Thus, the map $\omega_{\infty/e}$ is not surjective, giving  
$\CH_n^{S_n^{\mathfrak p}} \ne 1$ from Chevalley--Herbrand formula in $K_n/K_e$.

\begin{lemma}\label{Ddelta}
Let $\Delta_e^{\mathfrak p} := v_p(\CD_X^{\mathfrak p})$; then $\Delta_e^{\mathfrak p}$ 
is an absolute constant for the set of $\Z_p$-extensions $K/k$ such that the integer 
$e < \delta$ is fixed; then $\Delta_e^{\mathfrak p} = \delta$ for $e = 0$ and 
$\Delta_e^{\mathfrak p} \geq \delta$ for $0 < e < \delta$.
\end{lemma}

\begin{proof}
The last linear form $L_\spp = \sum_{i=1}^\spp \theta_i$ is particular since 
$\sum_{i=1}^\spp \theta_i$ is given, modulo $p^{n-e},$ via the symbols
$\big(\ffrac{x, K_n/K_e}{{\mathfrak p}}\big)$, of orders $p^{n - \delta}$, thus 
of the form $\sigma_{n/e}^{p^{\delta-e} \cdot \,u_n}$, $u_n \in \Z_p^\times$.

Hence, $\Delta_e^{\mathfrak p} = \delta$ if $e = 0$. If $e>0$, the linear forms 
for $k \ne \spp$ are such that:
$$L_k = \sm_{i=1}^\spp \theta_i \cdot \xi_\spp^{i k} = 
\sm_{i=1}^\spp \theta_i + \sm_{i=1}^\spp \theta_i \cdot (\xi_\spp^{i k} - 1); $$

\noindent
but $\xi_\spp^{i k} - 1 = (\xi_\spp^{v(k)} - 1) \cdot \gamma_{i,k}(\xi_\spp)$, 
$\gamma_{i,k}(\xi_\spp) \in \Z[\xi_\spp]$, where $v(k) := v_p(k)$. So, we get:
$$L_k = \sm_{i=1}^\spp \theta_i + \sm_{i=1}^\spp \theta_i \cdot 
(\xi_\spp^{v(k)} - 1) \cdot \gamma_{i,k}(\xi_\spp) = p^{\delta - e} \cdot u_n +
(\xi_\spp^{v(k)} - 1) \cdot \sm_{i=1}^\spp \theta_i \cdot \gamma_{i,k}(\xi_\spp). $$

It seems difficult to compute the valuations of the sums 
$\sum_{i=1}^\spp \theta_i \cdot \gamma_{i,k}(\xi_\spp)$ which may be non 
trivial. But $\Delta_e^{\mathfrak p}$ only depends on $L_e = K \cap L$, 
independent of $K$ once $e$ is fixed, and an elementary computation, 
using the valuation of each $k$, gives $\Delta_e^{\mathfrak p} \geq \delta - e + e 
= \delta$ (the minimal value $\delta$ being obtained when all the sums 
$\sum_{i=1}^\spp \theta_i \cdot \gamma_{i,k}(\xi_\spp)$
are $p$-adic units).
\end{proof}

We note the obvious consequence of the computation of $\CD_X^{\mathfrak p} = 
p^{\Delta_e^{\mathfrak p}}$:

\begin{corollary}\label{rank}
There exists a sub-matrix of $\CM_{\infty/e}$ of dimension $\spp$ having a 
determinant of $p$-valuation $\Delta_e \leq \Delta_e^{\mathfrak p}$.
\end{corollary}

This may result from the using of some lines of the matrices 
$\CM_X^{\ov{\mathfrak p}}$ and $\CM_E$.
Let's denote by $\CD_{\infty/e} =: p^{\Delta_e}$ such a determinant. 

\smallskip
We deduce, from the above, the Chevalley--Herbrand formula for 
$\CH_n^{S_n^{\mathfrak p}}$ in $K_n/K_e$:

\begin{theorem}
Let $K$ be a $\Z_p$-extension of $k$ distinct from $L$, $\ov L$. Put 
$K \cap L =: K_e$ and assume that $0 \leq e < \delta$. Then
$\order (\CH_n^{S_n^{\mathfrak p}})^{g_{n/e}} = \ffrac{[p^{n-e}]^{{\spp}}}
{\order \omega_{n/e}(E_e^{S_e^{\mathfrak p}})} = p^{\Delta_e}$, for $n \geq e$. 
\end{theorem}

\begin{proof}
It suffices to verify that the index of the image $\omega_{n/e}(E_e^{S_e^{\mathfrak p}})$ 
in $\wt \Omega_{n/e}^{\,\mathfrak p}$ is $\CD_{\infty/e} = p^{\Delta_e}$. This 
obvious fact can be specified by taking a $\Z_p$-basis $\{b_1, \ldots, b_\spp\}$ of 
$\wt \Omega_{\infty/e}^{\,\mathfrak p}$, such that the set $\{q_1 \cdot b_1, \ldots, 
q_\spp \cdot b_\spp\}$ be a basis of the image (invariants of similitude $q_i \in p^{\N}$ 
which may have an arithmetic interpretation); thus the index is $q_1 \cdots q_\spp 
= p^{\Delta_e}$. 
\end{proof}

\subsection{Estimation of \texorpdfstring{$\order \Ccl_n (S_n^{\mathfrak p})$}{Lg}}
\label{estimation}
We still have $\order \CH_n^{g_{n/e}} = p^{n-e} \cdot \CR_e^{\mathfrak p}$ 
(Theorem \ref{Ptorsion} for $\delta' = e$) and we may obtain an order of magnitude of 
$\Ccl_n (S_n^{\mathfrak p}) \subseteq \Ccl_n (S_n)$ as follows:

\smallskip
We consider $\CH_n$ and $\CH_n^{S_n^{\mathfrak p}}$ as 
$\Z_p[g_{n/e}]$-modules and use the above Chevalley--Herbrand
formula. The exact sequence: 
$$1 \to \Ccl_n(S_n^{\mathfrak p}) \to \CH_n \to 
\CH_n^{S_n^{\mathfrak p}} \to 1, $$

\noindent
gives rise to the following one, 
recalling that the primes of $S_n^{\mathfrak p}$ are invariants under 
$g_{n/e}$ and that $K_e$ is $p$-principal:
$$1 \to \Ccl_n(S_n^{\mathfrak p}) \too \CH_n^{g_{n/e}} \too 
(\CH_n^{S_n^{\mathfrak p}})^{g_{n/e}} \mathop {\too}^{\psi} 
{\rm Im}(\psi) \subseteq {\rm H}^1(g_{n/e},\Ccl_n(S_n^{\mathfrak p}))
= \Ccl_n(S_n^{\mathfrak p}) \to 1.$$

Whence $\order \Ccl_n(S_n^{\mathfrak p}) \geq \order \CH_n^{g_{n/e}} \times
\big [\order (\CH_n^{S_n^{\mathfrak p}})^{g_{n/e}}\big]^{-1} \geq
p^{n-e} \cdot \order \CT_e^{\mathfrak p} \times p^{-\Delta_e}$. 
 
\smallskip 
In conclusion, we get:
\begin{equation}\label{orderCl(S)}
p^{n-e-\Delta_e^{\mathfrak p}} \cdot \order \CT_e^{\mathfrak p} \leq
p^{n-e-\Delta_e} \cdot \order \CT_e^{\mathfrak p} \leq \order \Ccl_n(S_n^{\mathfrak p}) 
\leq \order \CH_n^{g_{n/e}} = p^{n-e} \cdot \order \CT_e^{\mathfrak p}.
\end{equation}

The obvious inclusion $(\CH_n^{S_n^{\mathfrak p}})^{G_n} \subseteq 
(\CH_n^{S_n^{\mathfrak p}})^{g_{n/e}}$ implies 
$\order (\CH_n^{S_n^{\mathfrak p}})^{g_{n/e}} \geq p^{\delta-e}$ 
(Formula \eqref{CH3}), which yields $\Delta_e \geq \delta - e$.
If $e = 0$ the formulas coincide, as expected.

\smallskip
These results suggest that the $S^{\mathfrak p}$-Iwasawa 
invariants may be trivial (i.e. $\lambda^{\mathfrak p}(K/k) = 0$, 
$\mu^{\mathfrak p}(K/k) = 0$, with obvious notations).

\smallskip
If so,  $\order \CH_n^{S_n^{\mathfrak p}} = p^{\nu^{\mathfrak p}(K/k)}$ 
for $n \gg 0$ and: 
$$\order \CH_n = p^{\nu^{\mathfrak p}(K/k)} 
\cdot \order \Ccl_n(S_n^{\mathfrak p}); $$

\noindent
but Formula \eqref{orderCl(S)} implies $\lambda(K/k) = 1$, $\mu(K/k) = 0$ and 
an estimation of $\nu(K/k)$; whence the result characterizing non-exceptional 
$\Z_p$-extensions in the case $e < \delta$:

\begin{theorem}\label{mainbis}
Let $k$ be an imaginary quadratic field and let $p \geq 3$ split in $k$. Assume 
that $\CH_k = 1$.  Let $\delta = v_{\ov{\mathfrak p}}(x^{\,p-1}-1) - 1$, where $x$ 
is the $S_k$-unit given by ${\mathfrak p}$. Let $K/k$ be a $\Z_p$-extension, 
$K \ne L, \ov L$, and put $K \cap L =: K_e$. 

\smallskip
(i) If $e < \delta$, we have, for $n$ large enough: 
$$p^{n-e-\Delta_e} \cdot \order \CT_e^{\mathfrak p} \leq 
\order \Ccl_n(S_n^{\mathfrak p}) \leq p^{n-e} \cdot \order \CT_e^{\mathfrak p}. $$ 

\smallskip
(ii) If $e = 0$ (total ramification), $p^{n-\delta} \leq \order \Ccl_n(S_n^{\mathfrak p})
\leq p^n$, for $n \geq \delta$ (from Lemma \ref{Ddelta}).

\smallskip
(iii) Then, $\lambda(K/k) = 1$ $\&$ $\mu(K/k) = 0$,
if and only if the $S^{\mathfrak p}$-Iwasawa invariants are zero 
($\lambda^{\mathfrak p}(K/k) = \mu^{\mathfrak p}(K/k) = 0$). If so, $\nu(K/k) \in 
[\nu^{\mathfrak p}(K/k) - e + t_\delta - \Delta_e,\, \nu^{\mathfrak p}(K/k) - e + t_\delta]$.
\end{theorem}

\begin{remark}\label{filtration}
Exceptional $\Z_p$-extensions take place in the family of 
$\Z_p$-extensions $K/k$, $K \ne L, \ov L$, such that $0 \leq e < \delta$
and such that the $S^{\mathfrak p}$-Iwasawa invariants are not all zero
(whence $\lambda^{\mathfrak p}(K/k) \geq 1$ if one assume $\mu^{\mathfrak p}(K/k) = 0$).
Except $K = k^\cyc$, we do not see any mean to find them (numerically or theoretically). 
Some algorithmic observations may help.

\smallskip
From the Chevalley--Herbrand formula $\order (\CH_n^{S_n^{\mathfrak p}})^{G_n} 
= p^{\delta - e}$, we define a filtration as follows in $K_n/k$. Put
$\M_n := \CH_n^{S_n^{\mathfrak p}}$ and:
$$\M_n^i = \{h \in \M_n ,\ h^{(1-\sigma_{n/0})^i} = 1\},\  i \geq 0; $$ 

\noindent
this filtration is such that $\order (\M_n^{i+1}/\M_n^i)$ is a decreasing $i$-sequence;
indeed, the maps: 
$$\ds \M_n^{i+1}/\M_n^i \mathop 
{\toooo}^{\hbox{\tiny$1\!-\!\sigma_{n/0}$}}\M_n^i/\M_n^{i-1} $$ 

\noindent
are injective.

\smallskip
Then $\CH_n^{S_n^{\mathfrak p}} = \M_n^{b_n}$, for the least integer $i = b_n$ 
such that $\M_n^{i+1} = \M_n^i$.

\smallskip
As we know, in a practical aspect, the algorithm computing the $\M_n^i$'s, for 
$n$ fixed, is rapidly random; in our case, it only depends on norm factors of the form 
$\ffrac{p^{n-e}}{\order\omega_k(\Lambda_k^i)}$, where the $\Lambda_k^i$'s 
are increasing modules containing $\langle x, \ov x \rangle$; since the first 
step gives the ``small'' order $p^{\delta - e}$, independent of $n$, the next ones 
give decreasing orders (not necessarily strictly); thus, it is not easy to imagine 
unbounded values of the $\order \CH_n^{S_n^{\mathfrak p}}$ (see \cite[\S\,6.3.1]
{Gra2026b} for more details about the algorithm which must be adapted to the 
case of $S_n^{\mathfrak p}$-classes).
\end{remark}

\subsection{A formula for \texorpdfstring{$\order \omega_{n/e}(E_e)$}{Lg}}

There exists a remarquable relation between $\omega_{n/e}(E_e)$ and
the image of the specific unit $\eta_\0^{} \in E_e$ coming from the relation
${\mathfrak p}_\0^{\he} = (x_\0)$ (for the prime-to-$p$ class number $\he$ 
of $K_e$) invariant under $\sigma_e$ (generator of $G_0 =
\Gal(K_e/k)$), then giving $x_\0^{1 - \sigma_e} = \eta_\0^{}$. 

\begin{proposition}\label{eta0}
Let $\eta_\0^{} \in E_e$ be the unit defined by $x_\0^{1 - \sigma_e} =: 
\eta_\0^{}$. Put $E_{e,\0} := \big \langle \eta_\0^{}\big  \rangle^{}_{\Z[G_e]}$. 
Let $\CR_e^{\mathfrak p}$ be the $p$-regulator defined in Theorem 
\ref{Ptorsion}\,(i) and set $\re = \spp - 1 = p^e-1$. Then:
$$\order \omega_{n/e}(E_e) =  \order\omega_{n/e}\big (E_{e,\0}\big) 
= [p^{n-e}]^\re  \cdot \CR_e^{\mathfrak p}. $$
\end{proposition}

\begin{proof}
Let $\varepsilon \in E_e$; since 
$\BN_{e/0}(\varepsilon) = 1$, there exists $\alpha \in K_e^\times$ such that 
$\varepsilon = \alpha^{1-\sigma_e}$, where the ideal $(\alpha)$ is invariant 
under $\sigma_e$; since ${\mathfrak p}$ totally split in $K_e$ and 
$\ov {\mathfrak p}$ ramifies, invariant ideals in $K_e/k$ are generated by
$\CP = \prod_{i=1}^{\spp} {\mathfrak p}_i$, $\CP_\0 = {\mathfrak p}_\0$, and
$p$-principal ideals $(\alpha_\0)$ of $k$. So (we omit to write powers $\hp$
of $\varepsilon$, $\alpha$, giving $ {\mathfrak p}^{\hp} = (x)$, etc.): 
$$(\alpha) = \big(\prd_{i=1}^{\spp} 
{\mathfrak p}_i \big)^a \cdot {\mathfrak p}_\0^b \cdot (\alpha_\0) = 
\big(\prd_{i=1}^{\spp} x_i \big)^a \cdot (x_\0)^b \cdot (\alpha_\0) =
(x)^a \cdot (x_\0)^b \cdot (\alpha_\0). $$

Thus $\alpha = \eta \cdot x^a \cdot x_\0^b \cdot \alpha_\0$, $\eta \in E_e$; then,
since $x \in k^\times$ and $x_\0 \in K_e$:
$$\varepsilon = \alpha^{1-\sigma_e} =  \eta^{1-\sigma_e} \cdot x_\0^b{}^{(1-\sigma_e)} 
= \eta^{1-\sigma_e}  \cdot \eta_\0^b. $$

Whence $E_e = E_e^{1-\sigma_e} \cdot \langle \eta_\0^{} \rangle_\Z^{}$, 
which leads to $E_e = E_e^{(1-\sigma_e)^{p^N}}\!\! \cdot 
\langle \eta_\0^{} \rangle^{}_{\Z[G_e]}
= E_e^{(1-\sigma_e)^{p^N}}\!\! \cdot E_{e,\0}$,  for all $N$.
There exists $N \gg 0$ such that $(1-\sigma_e)^{p^N} \in p^{n-e}\Z[G_e]$; 
so, $\omega_{n/e}(E_e) = \omega_{n/e} \big (E_{e,\0} \big)$, of order 
$[p^{n-e}]^\re \cdot \CR_e^{\mathfrak p}$ (Theorem \ref{Ptorsion}).
\end{proof}

\begin{remark}
Let $K_n$ be any layer in $K/k$ with $n \leq e$; the previous notations 
are now related to $K_n$ instead of $K_e$. The relation $E_n = 
E_n^{1-\sigma_n} \cdot \langle \eta_\0^{} \rangle_\Z^{}$ is still valid 
with $\eta_\0^{} \in K_n$ (the $p$-principality of $K_n$ is crucial).
One deduces in the same way that $E_n \otimes \Z_p = 
\big \langle \eta_\0^{} \otimes 1 \big \rangle^{}_{\Z_p[G_n]}$ is 
monogenic of $\Z_p$-rank $p^n - 1$ (refer to \cite{BLM2023}
for a general setting about Minkowski units). It follows that {any layer}
in $L/k$ and in $\ov L/k$, the group of units has this property.
\end{remark}

\subsection{Properties of the \texorpdfstring{$p$}{Lg}-Hilbert class field of 
\texorpdfstring{$\wt k$ for $e < \delta$}{Lg}}\label{diagram2}

The diagram, analogous to that of \S\,\ref{diagram1}, is the following one for 
$e < \delta$, where $K/K_e$ is totally ramified at $p$. The field $M_K$ is still 
the compositum $K H_e^{{\mathfrak p}\,\pr}$:
\unitlength=1.52cm 
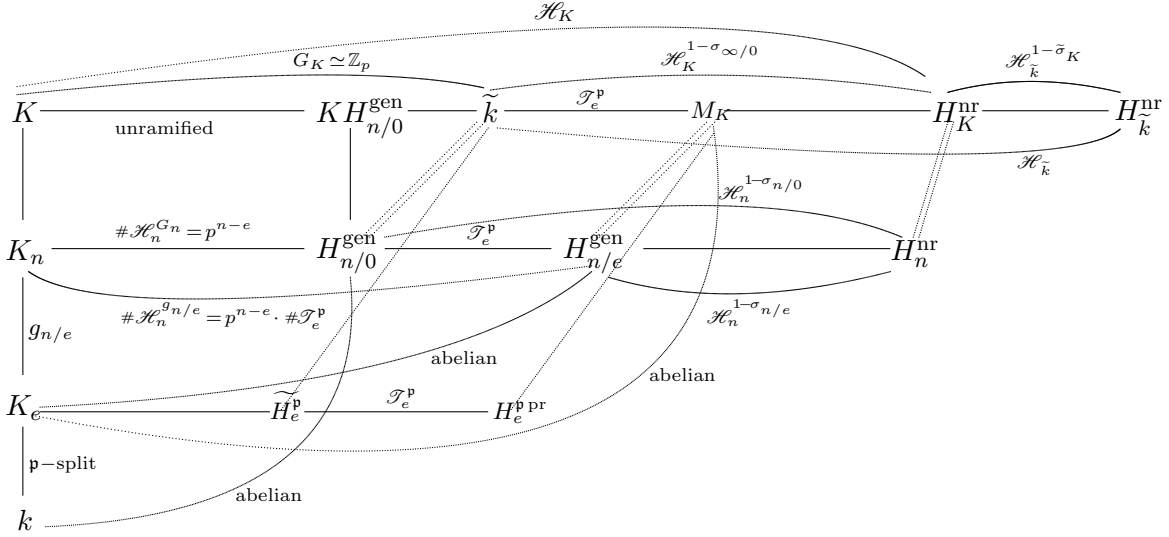
\begin{equation*}\label{genus}
\begin{aligned}
\vbox{\hbox{\hspace{-10.2cm} 
\begin{picture}(10.0,4.7)
\put(6.6,2.50){\line(1,0){1.45}}
\put(6.0,2.42){$H_{n/0}^\gen$}
\bezier{50}(6.4,2.6)(6.9,3.1)(7.4,3.6)
\bezier{50}(6.47,2.6)(6.97,3.1)(7.47,3.6)
\bezier{50}(8.4,2.6)(8.9,3.1)(9.4,3.6)
\bezier{50}(8.47,2.6)(8.97,3.1)(9.47,3.6)
\bezier{50}(11.2,2.6)(11.35,3.1)(11.5,3.6)
\bezier{50}(11.25,2.6)(11.40,3.1)(11.55,3.6)
\put(8.85,2.50){\line(1,0){2.15}}
\put(8.15,2.42){$H_{n/e}^\gen$}
\put(3.3,2.4){$K_n$}
\put(3.7,2.50){\line(1,0){2.2}}
\put(7.3,2.58){\tiny$\CT_e^{\mathfrak p}$} 
\put(8.25,3.76){\tiny$\CT_e^{\mathfrak p}$} 
\put(4.25,2.6){\tiny$\order \CH_n^{G_n}\!=\!p^{n-e}$}
\put(11.0,2.4){$H_n^\nr$}
\put(3.45,2.75){\line(0,1){0.8}}
\put(6.3,2.75){\line(0,1){0.8}}
\put(3.35,3.6){$K$}
\bezier{500}(3.4,3.91)(10.0,4.9)(11.3,4.0)
\ft\put(7.9,4.5){$\CH_K$}\ns
\put(3.6,3.7){\line(1,0){2.3}}
\put(4.25,3.5){\tiny${\rm unramified}$}
\put(6.0,3.6){$K H_{n/0}^\gen$}
\put(6.8,3.7){\line(1,0){0.6}}
\put(7.64,3.7){\line(1,0){1.6}}
\put(7.45,3.6){$\wt k$}
\ft\put(9.26,3.64){$M_K$}\ns 
\bezier{400}(3.53,1.05)(10.0,-0.2)(9.46,3.56)
\put(8.9,1.35){\tiny${\rm abelian}$}
\put(9.55,3.7){\line(1,0){1.8}}
\put(11.36,3.6){$H_K^\nr$}
\put(11.8,3.7){\line(1,0){1.1}}
\put(12.95,3.62){$H_{\wt k}^\nr$}
\bezier{300}(7.52,3.88)(9.45,4.15)(11.35,3.86)
\bezier{400}(7.58,3.55)(12.5,3.1)(13.0,3.54)
\put(12.1,3.2){\tiny$\CH_{\,\wt k}$}
\put(9.0,4.1){\tiny$\CH_K^{1 - \sigma_{\infty/0}}$}
\bezier{400}(3.4,3.84)(6.3,4.15)(7.46,3.9)
\put(5.8,4.09){\tiny$G_K\!\simeq\! \Z_p$}
\bezier{300}(11.5,3.86)(12.2,4.05)(13.0,3.86)
\put(12.0,4.08){\tiny$\CH_{\,\wt k}^{1-\wt\sigma_K^{}}$} 
\put(9.5,2.93){\tiny$\CH_n^{1\!-\!\sigma_{n/0}}$} 
\bezier{400}(6.6,2.6)(9.8,3.15)(11.1,2.6)
\ft\put(9.4,1.85){\tiny$\CH_n^{1\!-\!\sigma_{n/e}}$}\ns 
\bezier{300}(8.55,2.25)(9.6,1.94)(11.0,2.3)
\put(3.45,1.4){\line(0,1){0.85}}
\put(3.3,1.05){$K_e$}
\put(3.45,0.35){\line(0,1){0.6}}
\put(3.4,0.05){$k$}
\bezier{400}(3.6,1.12)(7.2,1.3)(8.4,2.3)
\bezier{400}(3.5,2.3)(4.2,1.8)(8.4,2.35)
\ft\put(4.3,1.85){\tiny$\order \CH_n^{g_{n/e}}\!=\! 
p^{n-e}\! \cdot \order\! \CT_e^{\mathfrak p}$} \ns 
\bezier{300}(3.65,0.08)(6.5,0.2)(6.3,2.25)
\put(7.0,1.5){\tiny${\rm abelian}$}
\put(5.3,0.3){\tiny${\rm abelian}$}
\ft\put(3.5,1.75){$g_{n/e}$}\ns
\put(3.5,0.58){\tiny${\mathfrak p}{\rm -split}$}
\put(3.6,1.08){\line(1,0){2.0}}
\put(5.9,1.08){\line(1,0){1.6}}
\put(6.6,1.16){\tiny$\CT_e^{\mathfrak p}$}
\ft\put(5.6,1.02){$\wt{H_e^{\mathfrak p}}$}\ns
\ft\put(7.54,1.02){${H_e^{{\mathfrak p}\,\pr}}$}\ns%
\bezier{150}(5.7,1.1)(6.6,2.32)(7.5,3.55)
\bezier{150}(7.7,1.1)(8.58,2.3)(9.46,3.5)
\end{picture}}} 
\end{aligned}
\end{equation*}
\unitlength=1.0cm

The formulas $\order \CH_n^{G_n} = p^{n-e}$, $\order \CH_n^{g_{n/e}} = 
p^{n-e} \cdot \CT_e^{\mathfrak p}$, and the corresponding genus fields are 
still valid with $\delta' = \min(e, \delta) = e$; $\wt k$ is the genus field
of $K/k$ (note that $M_K/k$ is not abelian since $H_k^\pr = \wt k$). The 
main difference is that, $\CH_n$ being distinct from $\CH_n^{g_{n/e}}$ 
(due to $\order (\CH_n^{S_n^{\mathfrak p}})^{G_n} = p^{\delta - e} \ne 1$), 
then $H_n^\nr$ is distinct from $H_{n/e}^\gen$ and since $\lambda(K/k)$ 
may be large, the degree $[H_n^\nr : H_{n/e}^\gen]$ increases, hence $[H_K^\nr : M_K]$ 
may be infinite, contrary to the case $e \geq \delta$; more precisely, assuming 
$\mu(K/k) = 0$, we get $\order \CH_n^{1-\sigma_{n/0}} = p^{(\lambda(K/k)-1) n + \nu_e}$ 
for some explicit constant $\nu_e$. 
Moreover, $Y_{\wt k} = \CH_{\,\wt k}^{1-\wt\sigma_K^{}}$, for the canonical generator
$\wt\sigma_K^{}$ of $\Gal(\wt k/K)$, since $H_{\wt k/K}^\gen$ is still $H_K^\nr$.
The following part of the diagram:
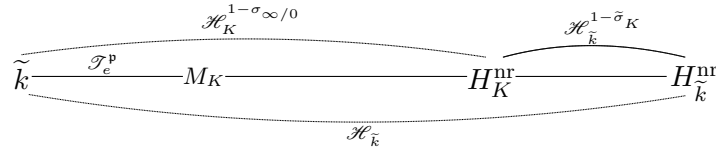
\unitlength=1.6cm 
\begin{equation*}
\vbox{\hbox{ 
\begin{picture}(9.0,1.15)
\put(1.45,0.6){$\wt k$}
\put(2.05,0.76){\tiny$\CT_e^{\mathfrak p}$}
\ft\put(2.85,0.64){$M_K$}\ns 
\put(1.6,0.7){\line(1,0){1.25}}
\put(3.2,0.7){\line(1,0){2.0}}
\put(5.6,0.7){\line(1,0){1.25}}
\put(5.2,0.6){$H_K^\nr$}
\put(6.88,0.62){$H_{\wt k}^\nr$}
\put(6.0,1.05){\tiny$\CH_{\,\wt k}^{1-\wt\sigma_K^{}}$}
\bezier{300}(5.5,0.86)(6.2,1.05)(7.0,0.86)
\bezier{300}(1.52,0.88)(3.45,1.15)(5.35,0.86)
\bezier{400}(1.58,0.55)(4.3,0.1)(7.0,0.54)
\put(4.2,0.18){\tiny$\CH_{\,\wt k}$}
\put(3.0,1.1){\tiny$\CH_K^{1 - \sigma_{\infty/0}}$}
\end{picture}}} 
\end{equation*}
\unitlength=1.0cm

\noindent
only depends on $K$ regarding the element $1 - \wt\sigma_K^{}$; then the 
fields $\wt k$, $M_K$, $H_{\wt k}^\nr$, and the class group
$\CH_{\,\wt k}$ do not depend on $K$; this diagram represents that of the 
case $e \geq \delta$ (see \S\,\ref{diagram1}), provided that $H_K^\nr = M_K$.

Thus, varying $(K, \wt\sigma_K^{})$ (via the canonical parametrization defined in 
Proposition \ref{3cases}), gives the quotient $\CH_{\,\wt k}/\CH_{\,\wt k}^{1-\wt\sigma_K^{}}$, 
and an exceptional $\Z_p$-extension $K = K_\0$, fixed by $\wt\sigma_{K_\0}$ occurs if and 
only if this quotient is infinite. In some sense, whatever $e$, we can write by abuse 
the ``Chevalley--Herbrand formula at infinity'', 
$\order \CH_{\,\wt k}^{G_K} = \ffrac{\order \CH_K \times 1}{p^\infty \times
\order \omega_{\wt k/K}(E_K)} = \ffrac{\order \CH_K}{p^\infty}$, 
since units are norms in the unramified extension~$\wt k/K$.

\subsection{Looking for exceptional non-Galois \texorpdfstring{$\Z_p$-extensions}{Lg}}

Recall that if $K \ne L, \ov L$ is a $\Z_p$-extension, $G_K = \Gal(\wt k/K) =: \langle 
\wt\sigma_K^{} \rangle$, where $\wt\sigma_K^{} =  \wt\sigma_-^{p^\alpha \cdot u} \cdot 
\wt\sigma_+^{p^\beta \cdot v}$, $\alpha \cdot \beta = 0$, $u, v \in \Z_p^\times$,
$u, v \ne \pm 1$. This writing defines a unique $K$ as soon as we fix for instance 
$v = 1$. If $e \ne 0$, $u \in 1 + p^e \Z_p^\times$ in which case necessarily 
$\alpha = \beta = 0$ (case (iii) of Proposition \ref{3cases}).

\smallskip
Taking a non-exceptional $\Z_p$-extension $K_\0$ (which exists 
infinitely many with $e \geq \delta$), Theorem \ref{pseudo-null} claims that 
$\CH_{\,\wt k}/\CH_{\,\wt k}^{1-\wt\sigma_\0^{}} < \infty$. If $e < \delta$ for the 
$\Z_p$-extension $K$, $\CH_{\,\wt k}$ has not necessarily the property 
$\order(\CH_{\,\wt k}/\CH_{\,\wt k}^{1-\wt\sigma_K}) < \infty$, but if one 
considers that the exceptional $\Z_p$-extensions are finite in number, 
for all the others, Theorem \ref{pseudo-null} applies; thus, $\CH_{\,\wt k}$ 
has the above property for almost all $K$. 

\smallskip
With the above parametrization, for the integers $\alpha$, $\beta$ fixed and $v = 1$, 
one obtains a sub-family indexed by a single parameter $u$ varying in the compact 
group $\Z_p^\times$; so it is a question to understand how the arithmetic invariants 
of $K$ vary. In a naive approach, for $K = K(u)$ fixed, $K' = K'(u') \to K$ (as $u' \to u$) 
if $K' \cap K = K_N$ for $N \to \infty$. 

\smallskip
The following lemma will point out how $\Z_p$-extensions intersect:

\begin{lemma}\label{M}
Let $K \ne L, \ov L$ be a fixed non-Galois $\Z_p$-extension defined by
$G_K = \langle \wt\sigma_K \rangle = \langle \wt\sigma_-^{p^\alpha \cdot u} 
\cdot \wt\sigma_+^{p^\beta v} \rangle$, $\alpha \cdot \beta = 0$, $u, v \in 
\Z_p^\times$, $G_K \ne \langle \wt\sigma_-^{\pm 1} \cdot \wt\sigma_+ \rangle$, 
and let $K'$ be a variable $\Z_p$-extension defined by
$G_{K'} = \langle \wt\sigma_{K'} \rangle = \langle \wt\sigma_-^{p^{\alpha'} \cdot u'} 
\cdot \wt\sigma_+^{p^{\beta'} v'}\rangle$, with similar conditions. 

\smallskip
(i) There are infinitely many $K'$ such that $K' \cap K = K_N$, $N$ arbitrary large;
they are such that $\alpha' = \alpha$, $\beta' = \beta$, $v' = v =1$, and $u' = u + p^N w$,
$w \in \Z_p^\times$.

\smallskip
(ii) If $K$ is Galois ($K = k^\cyc$ or $k^\acyc$), the fields $K'$ are defined by 
$G_{K'} = \langle \wt\sigma_-^{p^N w} \cdot \wt\sigma_+ \rangle$ if $K = k^\acyc$, 
and by $G_{K'} = \langle \wt\sigma_- \cdot \wt\sigma_+^{p^N w'} \rangle$ if 
$K = k^\cyc$.
\end{lemma}

\begin{proof}
We have to consider the four matrices:
\begin{equation}
\left \{\begin{aligned}
A_1& = \begin{pmatrix} p^\alpha u  & v \\ u' & p^{\beta'}v' \\
\end{pmatrix}, \ \alpha' = \beta = 0, \ \  \ \  \ \ 
A_2  = \begin{pmatrix} u  & p^\beta v \\ p^{\alpha'} u'  & v' 
\end{pmatrix}, \ \beta' = \alpha = 0, \\
A_3 & = \begin{pmatrix} p^\alpha u  &\ v \\ p^{\alpha'} u'  &\ v' \\
\end{pmatrix}, \ \beta' = \beta = 0, \ \  \ \  \ \ 
\ \ A_4  = \begin{pmatrix} \, u  &\  p^\beta \, v \\ \ u' &\  p^{\beta'}v' \\
\end{pmatrix}, \ \alpha' = \alpha = 0, 
\end{aligned}\right .
\end{equation}

\noindent
whose determinants are, respectively: 
$$\hbox{ $\bullet$ $p^{\alpha + \beta'} u v' - u'v$,\ \ \ \ $\bullet$ $p^{\alpha' + \beta} u' v - uv'$,
\ \ \ \  $\bullet$ $p^\alpha uv' - p^{\alpha'} u'v$,\ \ \ \   $\bullet$ $p^{\beta'} uv' - p^\beta u'v$. } $$

Then $K' \cap K = K_N$, $N$ unbounded, if and only if the determinants 
are of unbounded $p$-valuation, which yields, respectively:
\begin{itemize}
\item  $\alpha' = \alpha = 0$; $\beta' = \beta = 0$, and $v'^{-1}u' = v^{-1}u + p^N w$,
\item  $\alpha' = \alpha = 0$; $\beta' = \beta = 0$, and $v'^{-1}u' = v^{-1}u + p^N w$,
\item  $\alpha' = \alpha$, $\beta' = \beta = 0$, and $v'^{-1}u' = v^{-1}u + p^N w$,
\item  $\beta' = \beta$, $\alpha' = \alpha = 0$, and $v'^{-1}u' = v^{-1}u + p^N w$.
\end{itemize}
Whence the existence, for any given $\alpha$, $\beta$, $u$, $v$,
of $K' = K'(N)$, for arbitrary $N$; since by choice of the generators, one 
may suppose that $v = v' = 1$, then, whatever $\alpha$, $\beta$, $u$,
such that $\alpha \cdot \beta = 0$, $u \in \Z_p^\times$, the fields $K'(N)$, 
are defined via $\alpha' = \alpha$, $\beta' = \beta$ and $u' = u + p^N w$,
hence by the parameter $w \in \Z_p^\times$.

\smallskip
The case of the two Galois $\Z_p$-extensions is
immediate from the associated matrices.
\end{proof}

We propose the following conjecture and try to give some arguments for its 
credibility; note that, if we assume $\CH_k = 1$, the logarithmic class group 
$\wt \CH_k$ (hence $\delta = \delta_p(k)$) may be arbitrary. 

\begin{conjecture}\label{conj}
Let $k$ be an imaginary $p$-principal quadratic field for $p$ split in $k$. 
Then only the Galois $\Z_p$-extension $k^\cyc$ can be exceptional; whence, 
for any non Galois $\Z_p$-extension $K \ne L, \ov L$, and for $K = k^\acyc$, 
one has $\lambda_p(K/k) = 1$ and $\mu_p(K/k) = 0$.
\end{conjecture}

({\bf a}) First indication. 
Let $K = K(e, \alpha, \beta, u)$ for $\alpha$, $\beta$ fixed such that 
$\alpha \cdot \beta = 0$, $v = 1$ and $u \in \Z_p^\times$, $u \ne \pm 1$. 
The map, $u \mapsto \wt\sigma_K^{} =  \wt\sigma_-^{p^\alpha \cdot u} \cdot 
\wt\sigma_+^{p^\beta}$ is continuous for the natural topologies of $\Z_p^\times$ 
and $\Gamma = \langle \wt\sigma_- , \wt\sigma_+\rangle$. Consider $f_e(u) := 
\order \big (\CH_{\,\wt k}/\CH_{\,\wt k}^{1 - \wt\sigma_K^{}} \big)$; we may assume 
$0 \leq e < \delta$. This defines $\delta$ functions $f_e(u)$ of the parameter 
$e$, with values in $p^\N \cup \{p^\infty\}$. 

\smallskip
Let us fix $e$, vary $u $, and assume there exists a finite number 
of exceptional $K(e, u)$, say for $u \in \{u_1, \ldots , u_h \}$; thus 
$f_e(u) \in p^\N$ for $u$ varying in the open set $U := \Z_p^\times 
\setminus \{u_1, \ldots , u_h \}$. So, $f_e(u)$ is finite for $u \in U$ and infinite for 
$u = u_i$. This situation may give a contradiction if such functions $f_e(u)$ are governed 
by continuous analytic functions; since $f_e(\Z_p^\times) \subset p^\N \cup \{p^\infty\}$, 
the map $f_e$ is locally constant or very discontinuous. The same remark does exist
for the integer invariants $\lambda$, $\mu$, $\nu$ of $K$.

\smallskip
Suppose, now, that $K$ is an exceptional $\Z_p$-extension, and
let $K' = K'(N)$ be such that $K' \cap K = K_N$ for $N \geq e$.
Then, for all $n \geq N$, $K'_n/K_N$ is totally ramified at $p$;
since ${\mathfrak p}$ splits into $\spp = p^e$ primes in $K_e$, 
and $\ov {\mathfrak p}$ totally ramifies, there are $\spp + 1$
$p$-places of $K_N$ ramified in $K'_n/K_N$.
We then have the usual Chevalley--Herbrand formula in $K'_n/K_N$:
\begin{equation*}
\order \CH_{K'_n}^{\Gal(K'_n/K_N)} = \order \CH_{K_N} 
\frac{[p^{n-N}]^\spp}{\order \omega_{n/N}(E_{K_N})}.
\end{equation*}

Taking the projective limits on $n$, $\ds \limproj (\omega_{n/N}(E_{K_N})) 
\subseteq \Z_p^\spp$ is isomorphic to $\Z_p^{\spp'}$, $\spp' \in [0, \spp]$.

As in the \S\,\ref{matrix}, we obtain
$\prod_{j=1}^\spp \big(\frac{\varepsilon_N, K'_n/K_N}{{\mathfrak p}_{N,j}}\big) = 1$
for all $\varepsilon_N \in E_N$, and $\big(\frac{\varepsilon_N, K'_n/K_N}
{{\mathfrak p}_{N,0}}\big) = 1$. So, the corresponding matrix reduces
to a sub-matrix with $\spp - 1$ columns, the number of lines being
$p^N - 1$, arbitrary large. One may deduce that $\order\omega_{n/N}(E_{K_N})$
is of the form $\ffrac{[p^{n-N}]^{\spp-1}}{\CR_N^{\mathfrak p}}$, for a suitable
regulator of $K_N$. Thus:
$$\order \CH_{K'_n}^{\Gal(K'_n/K_N)} = \order \CH_{K_N} \times
p^{n-N} \cdot \CR_N^{\mathfrak p} = p^{n + (\lambda(K/k) - 1)N + 
\mu(K/k)p^N + \nu(K/k)} \times \CR_N^{\mathfrak p}.$$

But this is not a contradiction with $\lambda(K'/k) = 1$ and $\mu(K'/k) = 0$;
if so, for $N \gg 0$:
$$\order \CH_{K'_n} = p^{n + \nu(K'/k)} \geq \order \CH_{K'_n}^{\Gal(K'_n/K_N)}, 
\ n \gg N, $$ 

\noindent
which defines a huge constant, unbounded regarding $N$:
$$\nu(K'/k) \geq  (\lambda(K/k) - 1)N + \mu(K/k)p^N + \nu(K/k) + v_p(\CR_N^{\mathfrak p}). $$

\medskip
({\bf b}) Second indication.
It is worth noting that the previous phenomenon does exist when $K = k^\cyc$ is 
exceptional, with the fields $K'(N)$ defined in Lemma \ref{M}; in that case, 
$e = 0$, $\spp = 1$. A difficulty is to understand why $k^\cyc$ is specific 
regarding the others $\Z_p$-extensions.

\begin{remarks}
(i) The cyclotomic $\Z_p$-extension is singular, and governs many properties 
of $\wt k$. As we know, in our context, the condition $\wt \CH_k = 1$ (thus 
$\delta = 0$), is equivalent to $\CH_{\,\wt k} = 1$ thus equivalent to 
$\lambda(k^\cyc/k) = 1$ (Gold, Gold--Sand criterions \cite{Gold1974, San1991, 
San1993}, \cite[Th\'eor\'eme 10, Proposition 12]{Jau2024}, Fujii \cite{Fu2013, 
Fu2025}, and a synthetic proof in \cite[Theorems 7.3, 7.4]{Gra2026b}), in which 
case all the $\Z_p$-extensions of $k$ have minimal Iwasawa's invariants. 
It follows that $\lambda(k^\cyc/k) \geq 2$ under our hypothesis $\delta \geq 1$.

\smallskip
(ii) In other contexts, $k^\cyc$ allows some unaccustomed constructions 
as that of $\Z_p$-extensions having non-trivial $\mu$-invariants
(Iwasawa, Ozaki \cite{Oza2004}).

\smallskip
(iii) The units of $K = k^\cyc$ (the cyclotomic units over $\Q$) are norms of units 
(in $\Q^\cyc/\Q_N^\cyc$), then the units of $\Q^\cyc$ are norms in the relative 
unramified extension $K' k^\cyc/k^\cyc$, whence norms in $K'/k^\cyc_N$; one finds 
again that $\big(\frac{\varepsilon_N, K'_n/K_N}{{\mathfrak p}_{N}}\big) = 1$. Since $\spp = 1$, 
the norm factor is then equal to $p^{n-N}$. This only confirms that $\lambda(K'/k) 
\geq 1$ with the above formula for $K = k^\cyc$ and $K'$.

\smallskip
(iv) Another main characteristic of $k^\cyc/k$ is that the logarithmic class groups 
(as well as the ordinary ones) never capitulate in $k^\cyc$; indeed, capitulation 
of $\wt \CH_k$ in $K/k$ is more or less a criterion of minimality of Iwasawa's 
invariants of $K/k$ (it is a theorem in the totally real case \cite{Jau2019} and 
when $\wt \CH_k = 1$, whatever $\CH_k$ \cite{Jau2026}).

\smallskip
See \cite[Section A.5]{Gra2026a} for numerical examples of capitulations of
$p$-classes of $k$ in the first layer of the anti-cyclotomic $\Z_3$-extension of an 
imaginary quadratic field; of course, these examples do not suppose 
any hypothesis on the decomposition of $3$ in $k$, nor on the $3$-class groups,
but show that there is a priori no obstruction for capitulations in larger layers.

\smallskip
(v) A fundamental difference between $k^\cyc$ and an other $\Z_p$-extension, 
when $\CH_k = 1$ and $\delta \ne 0$, will be given in Appendix \ref{another}, 
\S\,\ref{appliconj}.
\end{remarks}

({\bf c}) Third indication.
Consider the non-Galois $\Z_p$-extensions $K$ and $\ov K$, and assume
they are exceptional; thus, $\CH_{\,\wt k}/\CH_{\,\wt k}^{1 - \wt\sigma_K^{}}$ and 
$\CH_{\,\wt k}/\CH_{\,\wt k}^{1 - \wt\sigma_{\ov K}}$ are infinite. For illustration, let's take 
$G_K = \langle \wt\sigma_-^u\cdot \wt\sigma_+ \rangle$, $u \in \Z_p^\times$, $u \ne \pm 1$;
then $G_{\ov K} = G_K^\tau = \langle \wt\sigma_-^{-u} \cdot \wt\sigma_+ \rangle$.
So, we have a priori the two independent conditions 
$\CH_{\,\wt k}/\CH_{\,\wt k}^{1 - \wt\sigma_-^u\cdot \wt\sigma_+}$ and
$\CH_{\,\wt k}/\CH_{\,\wt k}^{1 - \wt\sigma_-^{-u} \cdot \wt\sigma_+}$ are of 
infinite orders.

\smallskip
The only case where the two conditions are not independent is that of $k^\cyc$ 
and $k^\acyc$ for which $\wt\sigma_{k^\cyc} = \wt\sigma_-$, of conjugate $\wt\sigma_-^{-1}$ 
generating the same group, and $\wt\sigma_{k^\acyc} = \wt\sigma_+$ which commutes 
with~$\tau$. For example, for $k^\cyc$ this reduces to the sole condition
$\CH_{\,\wt k}/\CH_{\,\wt k}^{1 - \sigma_-}$ infinite.

Exceptional non-Galois cases raise the following problem. Let $H_K^\nr$ 
and $H_{\ov K}^\nr$ be the Hilbert class fields of $K$ and $\ov K$; the 
compositum $H_K^\nr H_{\ov K}^\nr$ is abelian over $\wt k$; then a 
question is the nature and order of the various Galois groups, like 
$\Gal(H_{\wt k}^\nr/H_K^\nr H_{\ov K}^\nr)$, in the diagrams below
(see Appendix \ref{another} for some partial results when $\wt k$ is the 
compositum of $K$ and $k^\cyc$):

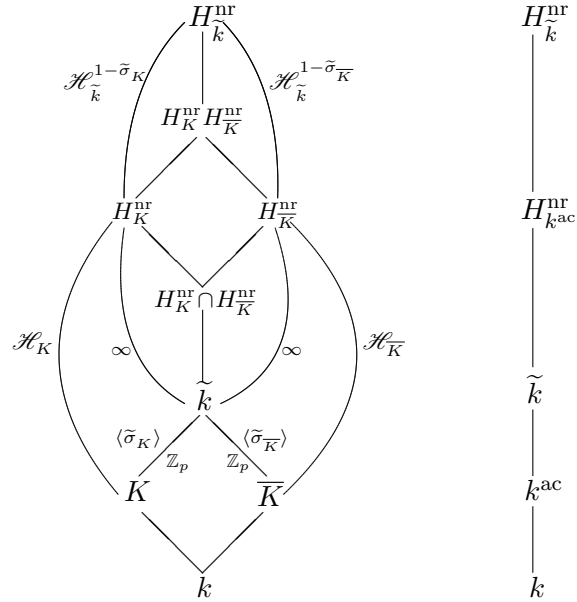
\unitlength=0.8cm 
\begin{equation*}
\begin{aligned}
\vbox{\hbox{\hspace{-3.0cm} 
\begin{picture}(10.0,9.6)
\put(3.86,9.45){$H_{\wt k}^\nr$}
\ft\put(3.4,7.8){$H_K^\nr H_{\ov K}^\nr$}\ns
\put(4.1,8.15){\line(0,1){1.2}}
\put(3.0,6.6){\line(1,1){1}}
\put(5.16,6.6){\line(-1,1){1}}
\ft\put(2.6,6.25){$H_K^\nr$}\ns
\ft\put(5.0,6.25){$H_{\ov K}^\nr$}\ns
\put(4.1,5.12){\line(-1,1){1}}
\put(4.1,5.12){\line(1,1){1}}
\ft\put(3.3,4.8){$H_K^\nr\! \cap\! H_{\ov K}^\nr$}\ns
\put(4.1,3.55){\line(0,1){1.2}}
\put(4.0,3.1){$\wt k$}
\put(3.05,2.0){\line(1,1){1}}
\put(5.16,2.0){\line(-1,1){1}}
\put(2.8,1.58){$K$}
\put(5.0,1.5){$\ov K$}
\put(4.1,0.4){\line(-1,1){1}}
\put(4.1,0.4){\line(1,1){1}}
\put(4.0,0.0){$k$}
\put(9.5,0.0){$k$}
\put(9.36,6.25){$H_{k^\acyc}^\nr$}
\put(9.36,9.45){$H_{\wt k}^\nr$}
\put(9.56,0.4){\line(0,1){1.1}}
\put(9.56,2.0){\line(0,1){1.1}}
\put(9.56,3.8){\line(0,1){2.3}}
\put(9.56,6.7){\line(0,1){2.6}}
\put(9.42,1.6){$k^\acyc$}
\put(9.46,3.2){$\wt k$}
\bezier{400}(2.8,6.6)(2.8,8.5)(3.8,9.5)
\bezier{400}(5.34,6.6)(5.4,8.5)(4.44,9.5)
\ft\put(1.8,8.4){$\CH_{\,\wt k}^{1-\wt\sigma_K^{}}$}\ns
\ft\put(5.2,8.4){$\CH_{\,\wt k}^{1-\wt\sigma_{\ov K}^{}}$}\ns
\put(2.65,2.55){\tiny$\langle \wt\sigma_K^{} \rangle$}
\put(4.75,2.55){\tiny$\langle \wt\sigma_{\ov K}^{} \rangle$}
\put(3.5,2.15){\tiny$\Z_p$}
\put(4.5,2.15){\tiny$\Z_p$}
\ft\put(0.95,4.1){$\CH_K^{}$}\ns
\ft\put(6.75,4.1){$\CH_{\ov K}$}\ns
\bezier{400}(2.7,1.7)(0.8,3.95)(2.6,6.2)
\bezier{400}(5.55,6.2)(7.8,3.95)(5.44,1.7)
\bezier{300}(3.8,3.2)(2.5,3.95)(2.8,6.1)
\bezier{300}(4.4,3.2)(6.0,3.95)(5.3,6.1)
\put(2.58,4.0){\tiny$\infty$}
\put(5.4,4.0){\tiny$\infty$}
\end{picture}}} 
\end{aligned}
\end{equation*}
\unitlength=1.0cm

\smallskip
({\bf d}) Fourth indication.
Simply that $k$ and any of its $p$-ramified $p$-extensions $F$ are $p$-rational, 
which is not in contradiction with non-trivial $p$-class groups, since in that case, 
the $p$-Hilbert class field of $F$ is contained in the compositum of its 
$\Z_p$-extensions, which limits the order of magnitude of $\order \CH_F$.
Moreover, unit groups of $p$-rational fields have specific norm properties
(see for instance \cite[IV.3,\,(b)]{Gra2005}).

\medskip
However, to show that a given non-Galois $\Z_p$-extension is not
exceptional, one may use the criterion of Theorem \ref{mainbis}.

\begin{conclusion} 
{\rm It is clear that proofs of many conjectures of Greenberg type \cite{Gree1976, Gree1998}
with techniques of Iwasawa's theory, relies almost exclusively on the deepest arithmetic tools, 
such as class field theory involving $p$-class groups, $S$-class groups, variations of 
Chevalley--Herbrand formulas (e.g., \cite{Gra2017a, Jau1986}), and especially 
logarithmic class groups (Jaulent \cite[Th\'eor\`eme 21]{Jau1994} and his corresponding fixed 
point formula for $\wt \CH$ in a cyclic $p$-extension, \cite{Jau2026}, Kleine \cite{Kle2021}), 
then ``partial'' Leopoldt conjectures in incomplete $p$-ramification leading into
independence of logarithms of algebraic numbers (see, e.g., Maksoud \cite{Mak2023}, 
Yanagisawa \cite[Theorem 1.3, Assumption (B)]{Yan2026}, after Ozaki \cite{Oza2001}, etc.).
All this may proceed, ultimately, on arithmetic $p$-adic conjectures that are currently 
beyond of reach as classical Leopoldt and Gross-Kuz'min conjectures.

\smallskip
If the algebraic structures of $\Lambda$-modules give valuable informations and 
limitations, and allows more straightforward writings, the problems addressed do 
indeed involve, random invariants/objects, as the Fermat quotients of $S$-units 
(i.e. the $\delta$'s), $p$-adic ranks in incomplete $p$-ramification, $p$-adic regulators 
needed for some crucial Chevalley--Herbrand formulas in relative extensions, etc.

\smallskip
In other words, elements defining $\Lambda$-module structures are consequence 
of $p$-adic random parts of arithmetic and not the inverse; for instance, the integer 
$r \geq 1$ in \cite[Formula (1)]{Oza2001} giving the finiteness of exceptional 
$\Z_p$-extensions is purely algebraic and unknown. 

\smallskip
But general {\it effective class field theory} has also its limits, which probably explains 
that many results are essentially subtle improvements of classical genus theory, whose 
crucial parts  are written in terms of variations of the Chevalley--Herbrand formulas.

\smallskip
Rare partial proofs require strong assumptions on the arithmetic; don't forget that 
in the articles assuming $\CH_k=1$ for $k$ imaginary quadratic, the fields under 
consideration are $p$-rationals, which represents a very low level of arithmetic 
complexity. But, if $p$-rationality gives a poor arithmetic, the level zero is reached 
when $\CH_k = \wt \CH_k = 1$ (usual and logarithmic $p$-principalities).}
\end{conclusion}

\begin{appendices}

\section{Another approach of \texorpdfstring{$\CH_{\,\wt k}$ 
when $\CH_k = 1$ and $\delta \geq 1$}{Lg}} \label{another}

In this Appendix, $K/k$ is an auxiliary $\Z_p$-extension,
distinct from $L, \ov L$, linearly disjoint from $k^\cyc/k$, which is 
considered as a parameter characterized by $G_K^{} := \Gal(\wt k/K)$.

\smallskip
We assume that $K/k$ is non-exceptional (whence, $\order \CH_n = 
p^{n + \nu(K/k)}$ for $n \gg 0$). Moreover, we suppose that $e = 0$;
this is possible from the result of Ozaki on the 
finiteness of exceptional $\Z_p$-extensions for each value of $e$ (the case 
$e \geq \delta$ would be suitable and effective from Theorem \ref{fundamental}, 
but the diagram of fields (Hilbert's fields, genus fields, Galois groups, etc.) 
is much more simple with $e = 0$). So, one may also take $K = k^\acyc$ 
as soon as it is not exceptional.

We consider the compositum $F_n = k_n^\cyc K_n$, of degree $p^2$ 
over $k$, for which $\bigcup_n F_n = \wt k$; we have the following 
diagram where the red lines denote unramified extensions (some  
properties will be justified later):
\unitlength=1.42cm 
\begin{equation}\label{diagramFn}
\begin{aligned}
\vbox{\hbox{\hspace{-5.5cm} 
\begin{picture}(10.8,8.2)
\put(6.65,3.55){\color{red}\line(2,3){1.2}}
\put(8.55,4.64){\color{red}\line(2,3){1.2}}
\put(5.0,1.9){\color{red}\line(1,1){1.35}}
\put(9.95,6.7){\color{red}\line(1,1){1.4}}
\put(6.7,3.43){\color{red}\line(3,2){1.5}}
\put(8.3,5.55){\color{red}\line(3,2){1.35}}
\put(4.85,0.55){\color{red}\line(0,1){1.1}}
\ft\put(11.3,8.2){\hbox{$H^\nr_{F_n}$}}\ns
\put(3.62,1.75){\color{red}\line(1,0){1.03}}
\put(4.7,1.7){\tiny\hbox{$F_n$}}
\put(3.56,0.50){\line(1,0){1.1}}
\put(4.7,0.44){\tiny\hbox{$K_n$}}
\bezier{500}(5.0,0.55)(8.2,2.0)(8.3,4.3)
\put(7.45,2.0){$\CH_n$}
\bezier{500}(3.6,1.9)(5.0,5.3)(7.8,5.45)
\put(4.9,4.6){$\CH_n^\cyc$}
\bezier{500}(4.85,1.85)(6.1,5.0)(7.8,5.4)
\put(5.1,4.04){\tiny\hbox{$(\CH_n^\cyc)^{1\!-\!\sigma_n^\cyc}$}}
\put(4.8,3.75){\tiny\hbox{$= \!\BN_n(\!\CH_{F_n}\!)$}}
\bezier{400}(4.9,1.78)(7.2,2.6)(8.2,4.35)
\put(6.4,2.4){\tiny\hbox{$\CH_n^{1\!-\!\sigma_n}$}}
\put(6.05,2.15){\tiny\hbox{$= \!\BN_n^\cyc(\!\CH_{F_n}\!)$}}
\bezier{500}(7.9,5.6)(8.5,7.6)(11.3,8.2)
\bezier{500}(8.55,4.45)(10.85,5.0)(11.4,8.15)
\put(8.0,7.3){\ft$\Ker(\BN_n)$}
\put(7.8,7.0){\ft$=\!\CH_{F_n}^{1\!-\!\sigma_n}$}
\put(10.2,5.1){\ft$\Ker(\BN_n^\cyc)$}
\put(9.7,4.8){\ft$=\!\CH_{F_n}^{1\!-\!\sigma_n^\cyc}$}
\put(3.45,0.65){\line(0,1){1.00}}
\put(3.45,1.9){\line(0,1){3.0}}
\put(3.4,5.0){$k^\cyc$}
\put(3.35,1.7){\ft$k_n^\cyc$}
\put(4.1,0.12){\tiny\hbox{$G_n$}}
\bezier{300}(3.5,0.35)(4.1,0.2)(4.7,0.35)
\put(3.76,1.85){\tiny\hbox{$g_n\! \!=\!\langle \sigma_n\rangle$}}
\put(4.9,1.3){\tiny\hbox{$g_n^\cyc\!\!=$}}
\put(4.9,1.05){\tiny\hbox{$\langle \sigma_n^\cyc \rangle$}}
\put(3.76,1.1){\tiny\hbox{$G_n^\cyc$}}
\bezier{400}(3.55,0.6)(3.86,1.15)(3.55,1.66)
\put(4.0,1.56){\tiny\hbox{$\BN_n$}}
\put(4.55,1.1){\tiny\hbox{$\BN_n^\cyc$}}
\put(6.4,3.3){$Z_n$} 
\put(5.4,2.65){\ft$\CZ_n$} 
\put(9.5,6.5){\tiny\hbox{$H^{\cyc\, \nr}_n \! H_n^\nr$}}
\put(10.25,7.3){\ft$\CA_n$}
\put(7.85,5.4){\ft$H^{\cyc\, \nr}_n$} 
\put(8.2,4.4){\ft$H_n^\nr$} 
\put(4.9,0.5){\line(1,0){3.0}}
\put(8.0,0.44){$K$}
\put(3.4,0.44){$k$} 
\bezier{110}(8.1,0.7)(8.1,2.8)(8.1,4.9)
\put(8.0,4.95){$\wt k$}
\bezier{110}(3.62,5.05)(5.86,5.05)(8.1,5.05)
\end{picture}}}
\end{aligned}
\end{equation}
\unitlength=1.0cm

We denote by $Z_n$ the intersection of the Hilbert class fields 
$H_n^\nr$ and $H_n^{\cyc\, \nr}$ of $K_n$ and $k_n^\cyc$, 
respectively, and put $\CZ_n := \Gal(Z_n/F_n)$. Let
$\CH_n$ and $\CH_n^\cyc$ be the $p$-class groups of 
$K_n$ and $k_n^\cyc$, respectively; the $p$-Hilbert class
fields being $H_n^\nr$ and $H_n^{\cyc,\nr}$. By abuse $\sigma_n$
and $\sigma_n^\cyc$ denote lifts in $\Gal(F_n/k_n^\cyc)$ and
$\Gal(F_n/K_n)$.

\smallskip
Then put
$\CA_n := \Gal(H^\nr_{F_n}/H_n^{\cyc\, \nr} \,H_n^\nr)$,
where the union of the Hilbert class fields $H^\nr_{F_n}$ 
of $F_n$ is $\CH_{\,\wt k}$. 
By definition, $\CA_n = \Gal(H^\nr_{F_n}/H_n^{\cyc\, \nr}) \cap
\Gal(H^\nr_{F_n}/H_n^\nr)$. The relations with the kernels of 
the norms will be the object of Proposition \ref{An}.

\subsection{Inertia groups and genus fields}
We recall certain results established in the preceding sections.

\subsubsection{Inertia groups}
The pro-cyclic inertia groups of ${\mathfrak p}$ and $\ov {\mathfrak p}$
are $I_{\mathfrak p} = \Gal(\wt k/L)$ and $I_{\ov {\mathfrak p}} = 
\Gal(\wt k/ \ov L)$, respectively, and are such that $\Gal(\wt k/k) = 
I_{\mathfrak p} \oplus I_{\ov {\mathfrak p}}$.
This implies that $\wt k/k^\cyc$ and $\wt k/K$ are unramified,

\begin{lemma}\label{D01}
We have the following properties of ramification of $F_n/k$:

\smallskip
(i) For all $n$, $F_n/k_n^\cyc$ and $F_n/K_n$ are unramified.

\smallskip
(ii) For $n \geq \delta$, the decomposition fields of ${\mathfrak p}$ and 
$\ov {\mathfrak p}$, in $F_n/k_n^\cyc$ and $F_n/K_n$, are
of degree $p^{\delta}$ over $k$. 
\end{lemma}

\begin{proof}
Since $k^\cyc/k$ and $K_n/k$ are totally ramified at $p$, and since 
$\wt k/k^\cyc$, $\wt k/K$ are unramified, the inertia groups, 
$I_{{\mathfrak p}, n}$ and $I_{\ov {\mathfrak p},n}$ in $F_n/k$, 
are cyclic of degree $p^n$ and such that $\Gal(F_n/k) = 
I_{{\mathfrak p}, n} \oplus I_{\ov {\mathfrak p},n}$ since $\CH_k=1$
(the inertia groups in $\Gal(\wt k/k)$ must project both on the totally 
ramified extensions $k_n^\cyc/k$ and $K_n/k$ of degrees $p^n$; 
so they are ``diagonally''  placed in the diagram and does not have 
a non-trivial intersection with $F_n/k_n^\cyc$ nor with $F_n/K_n$. 

\smallskip
The description of the decomposition of $p$ in $F_n/k$ comes 
from its decomposition in $\wt k/k$ (see Section \ref{ramdec}).
\end{proof}

\subsubsection{Genus fields}
One may also refer to Diagram \ref{diagram2}. Consider the $p$-genus 
field $H_{n/0}^\gen \subset \wt k$ of $K_n/k$; it is fixed by the image of 
$\CH_n^{1-\sigma_n}$, where  $\sigma_n$ is a generator of $G_n =
\Gal(K_n/k)$. Similarly, the $p$-genus field $H_{n}^{\cyc,\gen} \subset \wt k$ 
of $k_n^\cyc/k$ is fixed by the image of $\CH_n^{\cyc, 1-\sigma_n^\cyc}$, 
where  $\sigma_n^\cyc$ is a generator of $G_n^\cyc = \Gal(k_n^\cyc/k)$.

\smallskip
Chevalley--Herbrand formulas in $K_n/k$ and $k_n^\cyc/k$ give 
$\order \CH_n^{G_n} = \order (\CH_n^\cyc)^{G_n} = p^n$, 
and the exact sequences:
\begin{equation*}
1\to (\CH_n)^{G_n} \too \CH_n \too \CH_n^{1 - \sigma_n} \to 1,\ \ \ 
1\to (\CH_n^\cyc)^{G_n} \too \CH_n^\cyc \too 
(\CH_n^\cyc)^{1 - \sigma^\cyc_{n}} \to 1,
\end{equation*}

\noindent
lead to $[H_{n/0}^\gen : K_n] = [H_{n/0}^{\cyc, \gen} : k_n^\cyc] = p^n$.
The non-ramification of $F_n/k_n^\cyc$ and $F_n/K_n$ shows that, due 
to the degrees, $H_{n/0}^{\cyc, \gen} = F_n$ and $H_{n/0}^\gen = F_n$. 
Whence, in the diagram, $F_n$ fixed both by the images of 
$(\CH_n^\cyc)^{1 - \sigma_n^\cyc}$ and $\CH_n^{1 - \sigma_n}$. 

\smallskip
The orders of $\CH_n^\cyc$ and $\CH_n$ are, 
$p^{\lambda(k^\cyc/k) \cdot n + \nu(k^\cyc/k)}$ and 
$p^{\lambda(K/k) \cdot n + \mu(K/k) \cdot p^n + \nu(K/k)}$, 
respectively, for $n$ large enough. 

\begin{corollary}\label{nu}
Put $g_n := \Gal(F_n/k_n^\cyc)$ and $g_n^\cyc := \Gal(F_n/K_n)$. We have:
$$\hbox{$[H^{\cyc\, \nr}_n : F_n] = \order \CH_{F_n}^{g_n} = p^{(\lambda(k^\cyc/k) -1) 
\cdot n + \nu(k^\cyc/k)}$ and $[H_n^\nr : F_n] = \order \CH_{F_n}^{g_n^\cyc} = 
p^{\nu(K/k)}$}. $$
\end{corollary}

\subsection{Study of a filtration}
Let $K/k$ be any $\Z_p$-extension distinct from $L, \ov L$; as before, since
$\CH_k = 1$, we have total ramification of ${\ov{\mathfrak p}}$ and total 
ramification of ${\mathfrak p}$ from $K_e$, $e < \delta$.
Recall that $\CH_n^i = \{c \in \CH_n,\ c^{(1-\sigma_n)^i} = 1\}$.
For $n \geq 0$ fixed, we consider the kernels $X_n^i$ of the maps:
$$\CH_n^{(1-\sigma_n)^i} \mathop {\toooo}^{1 - \sigma_n}
\CH_n^{(1-\sigma_n)^{i+1}}, \ \, i \geq 0. $$

\begin{lemma}\label{filtrebis}
We have the following properties:
\begin{itemize}
\item $X_n^0 = \CH_n^1$, of order $p^{n-e}$, for all $n \geq e$, 
\item The $i$-sequence $X_n^i$ is decreasing from $X_n^0$ to $1$,
\item for all $n \geq \delta + e$, $\order X_n^1 = \order \wt \CH_k = p^\delta$,
\item $\order \CH_n = \prod_{i \geq 0} 
\order X_n^i = p^{n - e + \delta} \cdot \prod_{i \geq 2} \order X_n^i$.
\end{itemize}
\end{lemma}

\begin{proof}
We have $X_n^i = \CH_n^{(1-\sigma_n)^i} \cap \CH_n^1$,
proving the decreasing, from $X_n^0 = \CH_n^1$ to $1$.

\smallskip
From the exact sequences:

\smallskip
$1 \to X_n^1 \to \CH_n^{1-\sigma_n} \to \CH_n^{(1-\sigma_n)^2} \to 1$,

\smallskip
$1 \to \CH_n^1 \to \CH_n \to \CH_n^{1-\sigma_n} \to 1,\ \ 
1 \to \CH_n^2 \to \CH_n \to \CH_n^{(1-\sigma_n)^2} \to 1$,

\smallskip\noindent
and $\order (\CH_n^2/\CH_n^1) = p^{\delta}$, for $n \geq e + \delta$
(\cite[Theorem 7.1, Jaulent, Appendix A]{Gra2026b}), we deduce that 
$\order X_n^1 = p^{\delta}$.
\end{proof}

\subsection{Application to Conjecture \ref{conj}}\label{appliconj}
It should be noted that the $\Z_p$-extension $K/k$ is exceptional if and 
only if $\prod_{i_n \geq 2} \order X_n^{i_n}$ is unbounded regarding $n$, 
knowing that for $i_n \geq 2$, $\order X_n^{i_n}$ is a decreasing 
$i_n$-sequence with $\order X_n^{i_n} \leq p^\delta$; thus, for
$n \gg 0$, one must have $X_n^{i_n} \ne 1$, for large values of $i_n$.
Assume, to simplify, that $e = 0$.

\medskip
({\bf a}) From Chevalley--Herbrand formula in $K_n/k$, 
$\order \langle \Ccl_n({\mathfrak P}_n), 
\Ccl_n({\mathfrak P}'_n) \rangle = p^n$; from Theorem \ref{mainbis}
applied to $S_n^{\mathfrak p}$ and $S_n^{\ov{\mathfrak p}}$:
$$p^{n-e-\Delta_e} \cdot \order \CT_e^{\mathfrak p} \leq 
\order \Ccl_n({\mathfrak P}_n) ,\   \order\Ccl_n({\mathfrak P}'_n)
\leq p^{n-e} \cdot \order \CT_e^{\mathfrak p}, $$

\noindent
where $e = 0$, $\Delta_e = \delta$, $\CT_e^{\mathfrak p} = 1$; thus:
\begin{equation}\label{2deltas}
\order \langle \Ccl_n({\mathfrak P}_n) \rangle = p^{n - \delta_n},\ \ 
\order \langle \Ccl_n({\mathfrak P}'_n) \rangle = p^{n - \delta'_n},\ \,
\delta_n,\, \delta'_n \in [0, \delta]. 
\end{equation}

Since $\order \big(\langle \Ccl_n({\mathfrak P}_n)\rangle 
\cap \langle\Ccl_n({\mathfrak P}'_n)\rangle \big) = 
p^{n - \delta_n - \delta'_n}$, $\Ccl_n({\mathfrak P}_n)^{p^{\delta'_n}}$ 
and $\Ccl_n({\mathfrak P}'_n)^{p^{\delta_n}}$ generate the same
subgroup of order $p^{n - \delta_n - \delta'_n}$; this yields a non-trivial 
relation for instance of the form:
\begin{equation}\label{form00}
({\mathfrak P}_n^{p^{\delta'_n}} \cdot {\mathfrak P}_n'^{\, p^{\delta_n}\psi_n}) 
= (\gamma_n),\ \psi_n \in \Z_p^\times, \gamma_n \in K_n^\times.
\end{equation}

If $K/k$ is exceptional, we obtain that when $n \to \infty$, there exist large 
$i_n$ such that:
\begin{equation*}
\CH_n^{(1-\sigma_n)^{i_n}} \!\cap \langle \Ccl_n({\mathfrak P}_n),
\Ccl_n({\mathfrak P}'_n) \rangle \ne 1,
\end{equation*}

\noindent
whence, from \eqref{2deltas}, that there exist integers $a_n$ 
modulo $p^{n-\delta_n}$ and $a'_n$ modulo $p^{n-\delta'_n}$ 
such that, for an ideal ${\mathfrak A}_n$ of $K_n$:
\begin{equation}\label{form1}
{\mathfrak P}_n^{a_n} \cdot {\mathfrak P}'^{a'_n}_n =
(\alpha_n) \cdot {\mathfrak A}_n^{(1-\sigma_n)^{i_n}} \, \ \&\ \ 
\hbox{${\mathfrak P}_n^{a_n} \cdot {\mathfrak P}'^{a'_n}_n$
non-$p$-principal.}
\end{equation}

A necessary condition may be deduced by taking the absolute norm in \eqref{form1}:
\begin{equation*}
x^{a_n} \cdot \ov x^{\,a'_n} =  \BN_n(\alpha_n). 
\end{equation*}

Let $\sigma_\infty$ be a topological generator of $\Gal(K/k)$; 
put $\big(\ffrac{x , K/k}{{\mathfrak p}}\big) = \sigma_\infty^{p^\delta w}$,  
$\big(\ffrac{\ov x , K/k}{{\mathfrak p}}\big) = \sigma_\infty^{p^\delta w'}$,
$w, w' \in \Z_p^\times$. We then obtain the necessary condition
$a_n w + a'_n w' \equiv 0 \pmod{p^{n-\delta}}$ that we write
$a'_n = \theta \cdot a_n + c_n \cdot p^{n-\delta}$, $c_n \in \Z_p$, 
for all $n \geq \delta$, with $\theta := - w \cdot w'^{-1}$.

\begin{lemma}
We have $\theta = 1$ if and only if $p$ is local norm in $K/k$
at the $p$-places. 
\end{lemma}

\begin{proof}
The above definitions lead to $\big(\ffrac{x , K/k}
{{\mathfrak p}}\big) \cdot \big(\ffrac{\ov x , K/k}
{{\mathfrak p}}\big) = \big(\ffrac{x \cdot \ov x , K/k}
{{\mathfrak p}}\big) = \sigma_\infty^{p^\delta \cdot w \cdot 
(1 - \theta)}$, that is $\big(\ffrac{p , K/k}{{\mathfrak p}}\big)^{\hp}
= \sigma_\infty^{p^\delta \cdot w \cdot (1 - \theta)}$, where the order
$\hp$ of $\Ccl({\mathfrak p})$ is prime to $p$, whence 
$\theta = 1$ if and only if $p$ is local norm at the $p$-places. 
\end{proof}

Note that $p$ may be norm in $K_n/k$ for $n \leq \delta$, but 
Relation \eqref{form1} must hold for all $n$; it is now equivalent to:
\begin{equation}\label{thetabis}
\left \{\begin{aligned}
& ({\mathfrak P}_n \cdot {\mathfrak P}'^\theta_n)^{a_n} \!\cdot
 {\mathfrak P}'^{\,c_n \cdot p^{n-\delta}}_n =
(\alpha_n) \cdot {\mathfrak A}_n^{(1-\sigma_n)^{i_n}},  
\ \, \hbox{$a_n \in \Z_p$ modulo $p^{n-\delta_n}$},  \\
& \hbox{$({\mathfrak P}_n \cdot {\mathfrak P}'^\theta_n )^{a_n} \cdot
{\mathfrak P}'^{\,c_n \cdot p^{n-\delta}}_n$
non-$p$-principal.}
\end{aligned} \right.
\end{equation}

So, Relations \eqref{form00}, \eqref{thetabis} may be incompatible and/or 
\eqref{thetabis} may limit values of $i_n$.

\medskip
({\bf b}) The case $K = k^\cyc$ is very different;
indeed, we know that ${\mathfrak P}_n \cdot {\mathfrak P}'_n$ is principal 
for all $n \geq 0$ and that $\Ccl_n({\mathfrak P}_n)$ and $\Ccl_n({\mathfrak P}'_n)$
are of order $p^n$ (whence $\delta_n = \delta'_n = 0$, $\psi_n = 1$, $\theta = 1$). 
Then Relation \eqref{form00} is empty. To get $X_n^{i_n} = \CH_n^{(1-\sigma_n)^{i_n}}
\cap \CH_n^1 = \CH_n^{(1-\sigma_n)^{i_n}} \cap 
\langle \Ccl_n({\mathfrak P}_n)\rangle \ne 1$, the condition
replacing Relation \eqref{thetabis} is:
\begin{equation}\label{formcyc}
{\mathfrak P}_n ^{\,c_n \cdot p^{n-\delta}} =
(\alpha_n) \cdot {\mathfrak A}_n^{(1-\sigma_n)^{i_n}}, \ \, 
v_p(c_n) \in [0, \delta - 1].
\end{equation}

A relation of the form ${\mathfrak B}_n^{(1-\sigma_n)^{j_n}} \in \CH_n^1$, 
for an ideal ${\mathfrak B}_n$ of $K_n$ does exist by definition for 
some $j_n$ (refer to the properties of the usual filtration given in 
Remark \ref{filtration}); but the more precise condition given by
Relation \eqref{formcyc} suggests  that in $(1-\sigma_n)^{i_n}$, the integer 
$i_n$ may be large enough regarding $n$, without any obstruction.

\medskip
({\bf c}) Let's give numerical examples for $p = 3$, $K = k^\acyc$, $n = 1$
($k^\acyc$ may present less constraints than the non-Galois case, but 
we have no other choice in degree $3$ totaly ramified). Then the classes
of ${\mathfrak P}_1$ and ${\mathfrak P}'_1$ have same order $3$, 
which gives two possible relations contrary to the case of $k^\cyc$; 
the two cases occur (see \S\,\ref{acyclo} for the notations):

\smallskip
\ft\begin{verbatim}
m=1370 h_k=44 Clogk=[[9],[9],[]] delta=2 Racyc=x^6+6*x^4-368*x^3+12339*x^2-9324*x+35226
Clogkacyc=[[9],[9],[3]] Classkacyc=[66,6]  Hacyc=[3,3]
Components of Cl(P1/P2), Cl(P1*P2):[22,2],[0,0]  (orders: 3, 1)
m=1379 h_k=16 Clogk=[[3],[3],[]] delta=1 Racyc=x^6+66*x^4+1089*x^2+5516
Clogkacyc=[[3],[],[3,3]] Classkacyc=[144,3]  Hacyc=[9,3]
Components of Cl(P1/P2), Cl(P1*P2):[0,0],[96,0]  (orders: 1, 3)
\end{verbatim}\ns

\smallskip
In the first example, the relation is $\Ccl_1({\mathfrak P}_1 {\mathfrak P}'_1) = 1$,
and $\Ccl_1({\mathfrak P}_1 {\mathfrak P}'^{\,-1}_1)$ of order $3$:

\smallskip
\ft\begin{verbatim}
kacyc.cyc=[66,6]
bnfisprincipal(kacyc,Splus)=[0,0],[-1583043164775,-11914375779,
                                   -11914375779,-109089956952,-97175581173,0]
bnfisprincipal(kacyc,Sminus)=[22,2]
\end{verbatim}\ns

In the second, 
the relation is $\Ccl_1({\mathfrak P}_1 {\mathfrak P}'^{\,-1}_1) = 1$, and
$\Ccl_1({\mathfrak P}_1 {\mathfrak P}'_1)$ of order $3$:

\ft\begin{verbatim}
kacyc.cyc=[144,3]
bnfisprincipal(kacyc,Splus)=[96,0]
bnfisprincipal(kacyc,Sminus)=[0,0],[-73/81,1681/729,1/27,130/729,34/729,-31/81]
\end{verbatim}\ns

\smallskip
Of course, it will 
be interesting to compute in degrees $n \geq 2$ to get more complex relations.

\smallskip
So, a main obstruction, for a $\Z_p$-extension $K \ne k^\cyc$
to be exceptional, lies in the arithmetic of the classes of
${\mathfrak P}_n$ and ${\mathfrak P}'_n$ as $n \to \infty$.
The needed relations for an exceptionality, may be an obstruction
from some layer, obstruction which deappears for $K = k^\cyc$.
This reinforces Conjecture \ref{conj} and could be a way to
improve the understanding of Greenberg's conjectures.

\subsection{Application to  \texorpdfstring{$\Gal(H_{F_n}^\nr/k)$}{Lg}}
Recall that $K/k$ is a non-exceptional $\Z_p$-extension, totally ramified 
at $p$ ($e = \ov e = 0$), and that, for $n \geq 0$, $\sigma_n$ is a 
generator of $G_n = \Gal(K_n/k)$. We come back to our context using 
the family $\{F_n = k_n^\cyc K_n \}_{n \geq 1}$ (refer to Diagram \ref{diagramFn}).
Recall that $\CH_n$ and $\CH_n^\cyc$ are the $p$-class groups of 
$K_n$ and $k_n^\cyc$, respectively; the $p$-Hilbert class
fields being $H_n^\nr$ and $H_n^{\cyc,\nr}$. The field $F_n$ is
the genus field of $K_n/k$ and $k_n^\cyc/k$. Then $\sigma_n^\cyc$
and $\sigma_n$ are generators of $\Gal(F_n/K_n) \simeq \Gal(k_n^\cyc/k)$ 
and $\Gal(F_n/k_n^\cyc) \simeq \Gal(K_n/k)$, respectively.
 
\begin{lemma}\label{zfinite}
We have $\order \CZ_n \leq p^{\delta}$, for all $n \geq \delta$.
\end{lemma}

\begin{proof}
Since the Artin map gives the isomorphism 
$\Gal(H^{\cyc\, \nr}_n/F_n) \simeq \CH_n^{\cyc, 1-\sigma_n^\cyc}$,
any element $\ov s \in \CZ_n$ is the restriction $\Big(\frac{Z_n/F_n}
{{\mathfrak A}^{1- \sigma_n^\cyc}}\Big)$ of an Artin symbol
$s = \Big(\frac{H^{\cyc\, \nr}_n/F_n}{{\mathfrak A}^{1- \sigma_n^\cyc}}\Big)$,
for an ideal ${\mathfrak A}$ of $k_n^\cyc$. It is clear that, since all the
extensions $H_n^{\cyc\, \nr}$, $H_n^\nr$, $Z_n$ are Galois over $k$,
$\sigma_n^\cyc$ operates on $\Gal(H^{\cyc\, \nr}_n/F_n)$ by inner 
automorphisms; thus, the action of $1 - \sigma_n^\cyc$ on the symbol 
yields:
$$s_{}^{1-\sigma'^\cyc_n} = \Big(\frac{H^{\cyc\, \nr}_n/F_n}
{{\mathfrak A}^{(1- \sigma_n^\cyc)\cdot (1- \sigma_n^\cyc)}}\Big) = 
\Big(\frac{H^{\cyc\, \nr}_n/F_n}{{\mathfrak A}^{(1- \sigma_n^\cyc)^2}}\Big); $$

\noindent
since $Z_n/K_n$ is abelian, $\ov s_{}^{1-\sigma'^\cyc_n} = 1$, thus
$\Big(\frac{H^{\cyc\, \nr}_n/F_n}{{\mathfrak A}^{(1- \sigma_n^\cyc)^2}}\Big) 
\in \Gal(H^{\cyc\, \nr}_n/Z_n)$ and then the image of
$(\CH_n^\cyc)^{(1-\sigma_n^\cyc)^2}$ is contained in $\Gal(H^{\cyc\, \nr}_n/Z_n)$,
whence:
$$\order \CZ_n \leq \order \big((\CH_n^\cyc)^{1-\sigma_n^\cyc}/ 
(\CH_n^\cyc)^{(1-\sigma_n^\cyc)^2} \big). $$

But, from Lemma \ref{filtrebis}, the order of the right member is equal to 
$\order X_n^1 = p^{\delta}$.
\end{proof}

Note that the reasoning may be handled using $\Gal(H_n^\nr/K_n)$ since 
$Z_n$ is both abelian over these two fields. We ignore if $\order \CZ_n$ may 
take arbitrary values less or equal to $p^\delta$. 

\begin{proposition}\label{An}
Put $\BN_n = \BN_{F_n/k_n^\cyc}$ and $\BN_n^\cyc = \BN_{F_n/K_n}$.
We have $\Ker(\BN_n) = \CH_{F_n}^{1-\sigma_n}$,
$\Ker(\BN_n^\cyc) = \CH_{F_n}^{1-\sigma_n^\cyc}$ and
$\CA_n = \CH_{F_n}^{1-\sigma_n} \cap \CH_{F_n}^{1-\sigma_n^\cyc}$. 
\end{proposition}

\begin{proof}
By definition, $\Gal(H_{F_n}^\nr/H^{\cyc\, \nr}_n) \simeq \Ker(\BN_n)$ 
and $\Gal(H_{F_n}^\nr/H_n^\nr) \simeq \Ker(\BN_n^\cyc)$; equivalently,  
$\BN_n(\CH_{F_n}) \simeq \Gal(H^{\cyc\, \nr}_n/F_n)$ and
$\BN_n^\cyc(\CH_{F_n}) \simeq \Gal(H^\nr_n/F_n)$ 
(cf. Diagram \ref{diagramFn}).

\smallskip
Let ${\mathfrak A}$ be an ideal of $F_n$ such that 
$\Ccl({\mathfrak A}) \in \Ker(\BN_n)$. Thus, we have
$\BN_n{\mathfrak A} = (\alpha_\cyc)$, $\alpha_\cyc \in 
k_n^\cyc$, and $\BN_n^\cyc{\mathfrak A} = (\alpha)$, $\alpha \in K_n$.
Since $\alpha_\cyc$ is norm of an ideal in the non-ramified extension
$F_n/k_n^\cyc$, it is norm of an element $\theta \in F^\times_n$, 
hence ${\mathfrak A} = (\theta) {\mathfrak B}^{1-\sigma_n}$, 
${\mathfrak B}$ ideal of $F_n$; then $\Ccl({\mathfrak A}) \in 
\CH_{F_n}^{1-\sigma_n}$ and $\Ker(\BN_n) = 
\CH_{F_n}^{1-\sigma_n}$. 
Similarly in $F_n/K_n$, if 
$\Ccl({\mathfrak A}) \in \Ker(\BN_n^\cyc)$, then
${\mathfrak A} = (\theta') {\mathfrak B}'^{1-\sigma_n^\cyc}$, 
$\theta' \in F^\times_n$, ${\mathfrak B}'$ ideal of $F_n$, and
$\Ker(\BN_n^\cyc) = \CH_{F_n}^{1-\sigma_n^\cyc}$. 
Whence $\CA_n \simeq \CH_{F_n}^{1-\sigma_n} \cap 
\CH_{F_n}^{1-\sigma_n^\cyc}$.
\end{proof}

\begin{lemma}\label{Anbis}
Let $\wt \CA_n \subseteq \CA_n$ be the image of 
$\CH_{F_n}^{(1-\sigma_n^\cyc)\cdot (1-\sigma_n^\cyc)}$.
The exponent of the quotient $\CA_n / \wt \CA_n$ divides 
$p^{\nu(K/k)}$. 
\end{lemma}

\begin{proof}
Let ${\mathfrak A}$ be an ideal of $F_n$ such that  
$\Ccl({\mathfrak A}) \in \CA_n$; then
${\mathfrak A} = (\theta) {\mathfrak B}^{1-\sigma_n} =
(\theta') {\mathfrak B}'^{1-\sigma_n^\cyc}$,
where ${\mathfrak B}$, ${\mathfrak B}'$ are 
ideals of $F_n$, then $\theta$, $\theta'$ are 
elements of $F_n^\times$.

\smallskip
Since $\BN_n^\cyc({\mathfrak A})
= \BN_n^\cyc(\theta') = \BN_n^\cyc((\theta){\mathfrak B}^{1-\sigma_n})
=\BN_n^\cyc(\theta) \cdot \BN_n^\cyc({\mathfrak B})^{1-\sigma_n}$, 
we get  $\BN_n^\cyc({\mathfrak B})^{1-\sigma_n} =
(\BN_n^\cyc(\beta))$ and similarly $\BN_n({\mathfrak B}')^{1-\sigma_n^\cyc} 
= (\BN_n(\beta'))$, $\beta, \beta' \in F_n$.
Having a symmetry, we consider, for instance, the relation
$\BN_n^\cyc({\mathfrak B})^{1-\sigma_n} = (\BN_n^\cyc(\beta))$ in $K_n$; 
taking the norm $\BN_n$ of this equality yields
$(1) = (\BN_n(\BN_n^\cyc(\beta)))$, thus $\BN_n(\BN_n^\cyc(\beta)) = 1$
since $E_k = 1$; so $\BN_n^\cyc(\beta) = \gamma^{1-\sigma_n}$, 
$\gamma \in K_n^\times$ (Hilbert's Theorem 90).

\smallskip
Thus, $\BN_n^\cyc({\mathfrak B})^{1-\sigma_n} = (\gamma)^{1-\sigma_n}$,
so that $\BN_n^\cyc({\mathfrak B}) (\gamma)^{-1}$ is an ideal of $K_n$
invariant by $G_n$, hence an ideal product of a principal ideal of $k$ 
by an element of $\CH_n^\ram \subseteq \CH_n$. 
This yields $\Ccl(\BN_n^\cyc({\mathfrak B})) \in \CH_n^{G_n} \cap
\CH_n^{1 - \sigma_n}$ using $\CH_n^{1 - \sigma_n}
= \BN_n^\cyc (\CH_{F_n})$; since $\Gal(H_n^\nr/F_n)$ is of order 
$p^{\nu(K/k)}$, $\Ccl({\mathfrak B})^{p^{\nu(K/k)}} \in 
\Gal(H_{F_n}^\nr/H_n^\nr) = \Ker(\BN_n^\cyc)
= \CH_{F_n}^{1-\sigma_n^\cyc}$ (Proposition \ref{An}).
Write ${\mathfrak B}^{p^{\nu(K/k)}} =  (\psi)
{\mathfrak D}^{1-\sigma_n^\cyc}$ in $F_n$.
Since ${\mathfrak A} = (\theta){\mathfrak B}^{1-\sigma_n}$,
we deduce ${\mathfrak A}^{p^{\nu(K/k)}} = (\theta^{p^{\nu(K/k)}})
((\psi) {\mathfrak D}^{1-\sigma_n^\cyc})^{1-\sigma_n}$, 
giving $\Ccl({\mathfrak A})^{p^{\nu(K/k)}} 
\in \CH_{F_n}^{(1-\sigma_n^\cyc)\cdot (1-\sigma_n)}$.
\end{proof}

\begin{lemma}
For all $m \geq n \geq 0$, the extensions $F_m/F_n$ do not contain 
any unramified sub-extension. Thus $\BN_{F_m/F_n} (\CA_m) = \CA_n$, 
for all $m \geq n \geq 0$ (cf. Corollary \ref{nu}).
\end{lemma}

\begin{proof}
In the same way as for the proof of Lemma \ref{D01}, one
sees that the two inertia groups in $F_m/F_n$ have a trivial
intersection and generate $\Gal(F_m/F_n)$. So, 
\end{proof}

\begin{corollary}
Taking the projective limits, we have the isomorphisms:
$$\wt \CA_K = \CH_{\,\wt k}^{(1-\wt\sigma_{k^\cyc}^{}) \cdot (1-\wt\sigma_K^{}) }
\simeq \CH_{\,\wt k}^{1-\wt\sigma_{k^\cyc}^{}} \big/
\CH_{\,\wt k}^{1-\wt\sigma_{k^\cyc}^{}} \cap \CH_{\,\wt k}^{G_K} \simeq
\CH_{\,\wt k}^{1-\wt\sigma_K^{}} \big /
\CH_{\,\wt k}^{1-\wt\sigma_K^{}} \cap \CH_{\,\wt k}^{G_{k^\cyc}} $$

\noindent
where $\wt\sigma_{k^\cyc}^{}$ and $\wt\sigma_K^{}$ are generators of
$G_{k^\cyc} = \Gal(\wt k/k^\cyc)$ and $G_K = \Gal(\wt k/K)$, respectively.

\noindent
The group $\CH_{\,\wt k}^{1-\wt\sigma_{k^\cyc}^{}} \cap \CH_{\,\wt k}^{G_K}$
is finite.
\end{corollary}

\begin{proof}
We have the exact sequences (the first one defining $Y_n^1$):
\begin{equation*}
\left \{ \begin{aligned}
& 1 \too Y_n^1 \too \CH_{F_n} \mathop 
{\tooooo}^{\hbox{\!\! \tiny\hbox{$(1\!-\!\sigma_n)(1\!-\!\sigma_n^\cyc)$}}}
\wt \CA_n \too 1, \\
& 1 \too Y_n^1 \too \CH_{F_n} \mathop 
{\toooo}^{\hbox{\tiny\hbox{$(1\!-\!\sigma_n)$}}}
\CH_{F_n}^{1-\sigma_n} / \CH_{F_n}^{1-\sigma_n} 
\cap (\CH_{F_n})^{g_n^\cyc} \too 1, \\
& 1 \too Y_n^1 \too \CH_{F_n} \mathop 
{\toooo}^{\hbox{\tiny\hbox{$(1\!-\!\sigma_n^\cyc)$}}}
\CH_{F_n}^{1-\sigma_n^\cyc} / \CH_{F_n}^{1-\sigma_n^\cyc} 
\cap (\CH_{F_n})^{g_n} \too 1,
\end{aligned}\right .
\end{equation*}

\noindent
Since $K/k$ is not exceptional, $\order (\CH_{\,\wt k})_{G_K}^{} =
\order (\CH_{\,\wt k}/\CH_{\,\wt k}^{1-\wt\sigma_K}) = 
\order \CH_{\,\wt k}{}^{G_K}_{}$ is finite,
which prove the claim. 
\end{proof}

By analogy with the computations of Lemma \ref{filtrebis}, one may 
put $\CH_{F_n}^{1-\sigma_n} \cap (\CH_{F_n})^{g_n^\cyc} =: 
\widehat \CH_n$ and $\CH_{F_n}^{1-\sigma_n^\cyc} 
\cap (\CH_{F_n})^{g_n} =: \widehat \CH_n^\cyc$;
therefore, a question arises: are they linked to the logarithmic 
class groups of the fields $K_n$ and $k_n^\cyc$ ?

\smallskip
Under the conditions $\CH_k = 1$ and $\delta \ne 0$, $\lambda(k^\cyc/k) \geq 2$
(Gold--Sands criterion), and the previous calculations leads, for $n$ large enough, to:
$$\order \CH_{F_n} = p_{}^{\, \lambda(k^\cyc/k) \cdot n  + \nu(k^\cyc/k) + \nu(K/k)}
\times  \order \CA_n \times (\order \CZ)^{-1} . $$

Only the subgroup $\CA_K$ is unknown and may be of order finite or not.
It is finite if and only if $\CH_{\,\wt k}^{1-\sigma_n}$ is finite.

\section{Numerical Computations}

We can not give computations in large degrees; so we are limited to take $p=3$,
with $e \in \{0, 1\}$ and $\delta \geq 1$ arbitrary. Moreover, to get splitting of 
${\mathfrak p}$ in $K_1$, for $K/k$ non Galois, necessarily $e = 1$, $\delta' = e = 1$, 
$K_1 = L_1$ is a non-Galois subfield of degree $3$ of the compositum 
$k^\cyc  k^\acyc$. This explains that many torsion groups are triviaux
(it is the case of $\CT_e^{\ov {\mathfrak p}}$).

\smallskip
Since $\CH_k = 1$, the mirror field $k^* = \Q(\sqrt{3 m})$ is also
$3$-principal and the radical needed for cyclic cubic 
extensions of $k$ is reduced to the fundamental unit $\varepsilon^*$ of $k^*$; 
then $\varepsilon^*$ gives rise to $k_1^\acyc$. From the polynomials defining
$k_1^\cyc$ and $k_1^\acyc$, one obtains easily the first layer $K_1$. Note that 
there exist infinitely many $\Z_p$-extensions $K/k$ such that $K_1 = L_1$.
One verifies that among the four fields $k_1^\cyc$, $k_1^\acyc$, $L_1$, 
$\ov L_1$, the required field $L_1$ has a minimal $3$-discriminant $3^4$,
that of $k_1^\cyc$ and $k_1^\acyc$ being $3^8$.

\subsection{Computation of \texorpdfstring{$\CT_e^{\mathfrak p}$}{Lg}} \label{torsion}

We use the {\sc pari/gp}-program, given in \cite[\S\,3.1.1]{Gra2019b}, 
computing, for any set $S$ of $p$-places of a given number field $K_1$
(via ${\sf bnfinit(PK1)}$), the $\Z_p$-rank of the Galois group $\CA_S$
(in ${\sf LA}$), of the maximal abelian $S$-ramified pro-$p$ extension 
of $K_1$, and its torsion group $\CT_S$ (in ${\sf LT}$). 
In the completely general case of incomplete $p$-ramification, the 
determination of the $\Z_p$-rank is only conjectural (\cite{Jau1998}, 
\cite{Gra2005}, \cite{Mai2005}).

\smallskip
In the numerical results, the data ${\sf S=List([d1, d2, d3, d4])}$,
$d_i \in \{0, 1\}$, gives the modulus ${\mathfrak m}_\0 = 
\prod_i {\mathfrak p}_i^{d_i}$, and the program computes the ray 
class group of modulus ${\mathfrak m} = {\mathfrak m}_\0^N$, for $N$ 
large enough such that the torsion part does appear in the structure of 
$\CA_S$. The number $\delta$ is given by the $p$-valuation of the 
order of the logarithmic class group $\wt \CH_k$ (via ${\sf bnflog(k,3)[1]}$).

\smallskip
Take care that when $m$ varies, {\sc pari/gp} does not always give 
${\mathfrak p}_\0$ (the unque prime of ramification index $3$) at the same 
place in the list of primes ${\mathfrak p}_i \mid p$ in $K_1$ (for instance,
for $m = 14$ one finds ${\sf List([P1,P2,P0,P3])}$, while for $m = 41$
one finds ${\sf List([P1,P0,P2,P3])}$). 

\smallskip
\ft\begin{verbatim}
{N=16;for(m=5,10^3,if(core(m)!=m,next);if(Mod(-m,3)!=1,next);P=x^2+m;
k=bnfinit(P);hk=k.no;if(Mod(hk,3)==0,next);Clogk=bnflog(k,3)[1];if(Clogk==[],next);
print();print("m=",m," h_k=",hk," Clogk=",Clogk);
\\ FUNDAMENTAL UNIT OF k* -- DEFINING POLYNOMIAL OF THE COMPOSITUM k^c.k^ac:
Pstar=x^2-3*m;kstar=bnfinit(Pstar);w=kstar.fu[1];Q=x^3-3*x-trace(w);
R=polcompositum(P,Q)[1];R=polredbest(R);print("Rac=",R);
kacyc=bnfinit(R);Rc=polcompositum(R,polsubcyclo(9,3))[1];
\\ DETERMINATION OF K_1 in  k^c.k^ac:
Kc=bnfinit(Rc);XX=nfsubfields(Kc,6);for(t=1,4,RQ=XX[t][1];
if(valuation(nfdisc(RQ),3)==4,print("PK1=",RQ);K1=bnfinit(RQ);
\\ COMPUTATIONS OF ALL THE S-MODULUS:
F=idealfactor(K1,3);print(F);d=matsize(F)[1];F1=component(F,1);
Sp=List(component(F,2));i=0;Pr=List([P1,P2,P3]);Ram=List;
for(t=1,4,c=Sp[t];if(c==1,i=i+1;listput(Ram,Pr[i]));if(c==3,listput(Ram,P0)));
print();print(Ram,":");for(z=2^d,2^(d+1)-1,bin=binary(z);mod=List;
for(j=1,d,listput(mod,bin[j+1],j));
\\ COMPUTATION OF THE S-RAY-CLASS GROUPS:
M=1;for(j=1,d,ch=mod[j];if(ch==1,F1j=F1[j];ej=F1j[3];
F1j=idealpow(K1,F1j,ej);M=idealmul(K1,M,F1j)));Idn=idealpow(K1,M,N);
Kpn=bnrinit(K1,Idn);Hpn=Kpn.cyc;LA=List;LT=List;RgA=0;RgT=0;e=matsize(Hpn)[2];
for(k=1,e,c=Hpn[e-k+1];w=valuation(c,3);if(w>=N-2,RgA=RgA+1;listinsert(LA,0,1));
if(w<N-2 & w>0,RgT=RgT+1;listinsert(LT,3^w,1)));
\\ WRITING OF THE RANKS AND TORSION GROUPS:
print(mod," rk_Z3(A_S)=",RgA," rk_3(A_T)=",RgT," LA=",LA," LT=",LT));break)))}

m=14 h_k=4 Clogk=[3] Qac=x^3-3*x+26 Rac=x^6+12*x^4+36*x^2+224
PK1=x^6-2*x^5+11*x^4-18*x^3+40*x^2-40*x+23
P1=[[3,[-1,-1,-1,0,-1,-1]~,1,1,[2,2,-4,5,6,0;0,3,-3,0,-3,0;2,-1,2,-1,0,3;
                             1,-2,4,4,0,-6;0,3,-3,-3,3,3;1,1,4,4,3,0]],1;
P2=[3,[-1,0,0,-1,1,-1]~,1,1,[2,-1,4,3,11,3;1,4,-4,3,4,0;0,-6,3,0,-3,0;
                             0,-3,0,-3,-3,-3;2,2,1,0,5,6;1,-5,5,0,-2,-3]],1;
P0=[3,[-1,1,1,0,-1,0]~,3,1,[0,0,-3,3,9,3;1,1,-4,1,-1,-1;2,-4,1,-1,-2,1;
                             0,-3,3,0,0,-6;1,4,-1,-2,2,5;1,-2,5,4,2,-4]],3;
P3=[3,[0,0,0,1,1,-1]~,1,1,[1,-4,2,11,11,1;2,4,-5,1,1,5;1,-4,-1,-4,-7,4;
                             2,-2,7,4,-2,-13;1,-1,-1,-7,2,7;0,0,9,6,0,-6]],1]
[P1,P2,P0,P3]
[0,0,0,0] rk_Z3(A_S)=0 rk_3(A_T)=0 LA=[]           LT=[]
[0,0,0,1] rk_Z3(A_S)=0 rk_3(A_T)=0 LA=[]           LT=[]
[0,0,1,0] rk_Z3(A_S)=1 rk_3(A_T)=0 LA=[00]         LT=[]
[0,0,1,1] rk_Z3(A_S)=2 rk_3(A_T)=0 LA=[00,00]      LT=[]
[0,1,0,0] rk_Z3(A_S)=0 rk_3(A_T)=0 LA=[]           LT=[]
[0,1,0,1] rk_Z3(A_S)=0 rk_3(A_T)=1 LA=[]           LT=[3]
[0,1,1,0] rk_Z3(A_S)=2 rk_3(A_T)=0 LA=[00,00]      LT=[]
[0,1,1,1] rk_Z3(A_S)=3 rk_3(A_T)=0 LA=[00,00,00]   LT=[]
[1,0,0,0] rk_Z3(A_S)=0 rk_3(A_T)=0 LA=[]           LT=[]
[1,0,0,1] rk_Z3(A_S)=0 rk_3(A_T)=1 LA=[]           LT=[3]
[1,0,1,0] rk_Z3(A_S)=2 rk_3(A_T)=0 LA=[00,00]      LT=[]
[1,0,1,1] rk_Z3(A_S)=3 rk_3(A_T)=0 LA=[00,00,00]   LT=[]
[1,1,0,0] rk_Z3(A_S)=0 rk_3(A_T)=1 LA=[]           LT=[3]
[1,1,0,1] rk_Z3(A_S)=1 rk_3(A_T)=1 LA=[00]         LT=[3]
[1,1,1,0] rk_Z3(A_S)=3 rk_3(A_T)=0 LA=[00,00,00]   LT=[]
[1,1,1,1] rk_Z3(A_S)=4 rk_3(A_T)=0 LA=[00,00,00,00]LT=[]
\end{verbatim}\ns

\smallskip
For our example, the three interesting lines are:

\ft\begin{verbatim}
[P1,P2,P0,P3]
[0,0,1,0] rk_Z3(A_S)=1 rk_3(A_T)=0 LA=[00]         LT=[]
[1,1,0,1] rk_Z3(A_S)=1 rk_3(A_T)=1 LA=[00]         LT=[3]
[1,1,1,1] rk_Z3(A_S)=4 rk_3(A_T)=0 LA=[00,00,00,00]LT=[]
\end{verbatim}\ns

\smallskip
For $k = \Q(\sqrt{-14})$, the units of $K_1$ are:

\smallskip\noindent
\ft\begin{verbatim}
Eps_1=Mod(7/31492800*x^5+109/2624400*x^4+47/17496*x^3+3007/36450*x^2 
                                                         +5923/4860*x+33361/4050,R),
Eps_2=Mod(11/10497600*x^5+289/1749600*x^4+29/2916*x^3+13469/48600*x^2 
                                                         +2797/810*x+11464/675,R).
\end{verbatim}\ns

\smallskip
One finds that $(\varepsilon_1^2-1) = (\varepsilon_2^2-1) =
{\mathfrak p}_1\,{\mathfrak p}_2 \, {\mathfrak p}_3\,
{\mathfrak p}_\0 = ({\mathfrak p}) \,{\mathfrak p}_\0$,
up to the product by prime to $p$ ideals of $K_1$, which gives 
$\delta_i(\varepsilon_j) = 0$ for $i \in [1, 3]$, $j \in [1, 2]$.

\smallskip
We compute that $\varepsilon_1 = \varepsilon_2^{\sigma_1}$ for a suitable generator 
$\sigma_1$ of $\Gal(K_1/k)$ and that $\eta_\0^{} = \varepsilon_2^{-2}$ which confirms 
the equality $E_e = \big\langle \eta_\0^{} \big\rangle^{}_{\Z_p[G_e]}$ of Proposition
\ref{eta0}.
We find that the $3$-parts of the ideals $(\eta_\0^{} - 1)$ and $(\eta_\0^{1-\sigma_1} - 1)$ 
are, respectively, ${\mathfrak p}_1 \cdot {\mathfrak p}_2 \cdot 
{\mathfrak p}_3 \cdot {\mathfrak p}_\0$ and
$(3) \cdot {\mathfrak p}_1 \cdot {\mathfrak p}_2 \cdot {\mathfrak p}_3^2 \cdot
{\mathfrak p}_\0$, so that $\iota_{\mathfrak p} (\eta_\0^{1-\sigma_1}) \in
(\CU_1^{\mathfrak p}{}^*)^3$.
In $\CU_1^{\mathfrak p}{}^*$, we see that the ${\mathfrak p}$-component of 
$(\eta_\0^2 - 1)$ is ${\mathfrak p} \cdot {\mathfrak p}_1 \cdot {\mathfrak p}_2 
\cdot {\mathfrak p}_3$, while that of $(\eta_\0^{1-\sigma_1} - 1)$ is 
${\mathfrak p} \cdot {\mathfrak p}_1 \cdot {\mathfrak p}_2 \cdot 
{\mathfrak p}_3^2$, giving the index $3$, so that
the index of $\iota_{\mathfrak p} (\langle \eta_\0^{} \rangle)$ in
$\CU_1^{\mathfrak p}{}^*$ is $3$ as expected.

\smallskip
Let's give other results giving some large torsion groups $\CT^{\mathfrak p}_1$:

\smallskip
\ft\begin{verbatim}
m=41 h_k=8 Clogk=[27] Qac=x^3+6*x-44
Rac=x^6-42*x^4+441*x^2+656
PK1=x^6+8*x^4-16*x^3+57*x^2-146*x+105
[P1,P0,P2,P3]:
[0,1,0,0] rk_Z3(A_S)=1 rk_3(A_T)=0 LA=[00]          LT=[]
[1,0,1,1] rk_Z3(A_S)=1 rk_3(A_T)=2 LA=[00]          LT=[9,9]
[1,1,1,1] rk_Z3(A_S)=4 rk_3(A_T)=0 LA=[00,00,00,00] LT=[]

m=86 h_k=10 Clogk=[3] Qac=x^3-3*x-514
Rac=x^6-114*x^4+3249*x^2+1376
PK1=x^6-10*x^4-10*x^3+111*x^2+394*x+369
[P1,P2,P0,P3]:
[0,0,1,0] rk_Z3(A_S)=1 rk_3(A_T)=0 LA=[00]          LT=[]
[1,1,0,1] rk_Z3(A_S)=1 rk_3(A_T)=2 LA=[00]          LT=[9,3]
[1,1,1,1] rk_Z3(A_S)=4 rk_3(A_T)=0 LA=[00,00,00,00] LT=[]

m=827 h_k=7 Clogk=[3] Qac= x^3-36*x-171
Rac=x^6-24*x^4+144*x^2+827
PK1=x^6-24*x^4-57*x^3+144*x^2+684*x+1019
[P1,P2,P0,P3]:
[0,0,1,0] rk_Z3(A_S)=1 rk_3(A_T)=0 LA=[00]          LT=[]
[1,1,0,1] rk_Z3(A_S)=1 rk_3(A_T)=2 LA=[00]          LT=[27,27]
[1,1,1,1] rk_Z3(A_S)=4 rk_3(A_T)=0 LA=[00,00,00,00] LT=[]

m=1577 h_k=28 Clogk=[3] Qac=x^3+45*x-396
Rac=x^6+30*x^4+225*x^2+6308
PK1=x^6+30*x^4-132*x^3+225*x^2-1980*x+5933
[P1,P0,P2,P3]:
[0,1,0,0] rk_Z3(A_S)=1 rk_3(A_T)=0 LA=[00]          LT=[]
[1,0,1,1] rk_Z3(A_S)=1 rk_3(A_T)=2 LA=[00]          LT=[81,27]
[1,1,1,1] rk_Z3(A_S)=4 rk_3(A_T)=0 LA=[00,00,00,00] LT=[]
\end{verbatim}\ns

\begin{remark}\label{Rnonnul}
As shown by the numerical results, for $e = 1$ and $p = 3$, the group 
$\CT^{\mathfrak p}_{\delta'}$ is never trivial when ${\delta'} \geq 1$; this comes 
from the following general fact, where $\{\varepsilon_1, \ldots, \varepsilon_\re\}$ 
is a system of fundamental $p$-principal units of $K_{\delta'}$ ($\re \geq p-1$ 
since $\delta' \ne 0$).

\smallskip
We have in $K_{\delta'}$, by extension of ${\mathfrak p}^{\hp} = (x)$, 
$\prod_{i=1}^{\spp} {\mathfrak p}_i^{\hp} = (x)$; then write $\varepsilon_j = 
1 + \pi \alpha_j$, $\alpha_j \in K_{\delta'}$, $j \in [1, \re]$, where $\pi$
is an uniformizing parameter for ${\mathfrak p}$ (one may suppose 
the $\alpha_i$'s prime to ${\mathfrak p}$ (but not to $\ov {\mathfrak p}$) 
otherwise the result is obvious); then, $\prod_{i=1}^{\re}\varepsilon_i^{a_i} 
\equiv 1+\pi \,\sum_{i=1}^{\re} a_i\,\alpha_i \pmod{{\mathfrak p}^2}$. Since the 
residue degrees of the $p$-places of $K_{\delta'}$ are equal to $1$, 
$\sum_{i=1}^{\re} a_i\,\alpha_i$ can not be prime to ${\mathfrak p}$ 
for all $a_i$, not all divisible by $p$, since $\re \geq 2$ for $p \geq 3$; 
thus, this implies the existence of some $p$-torsion in 
$\CU_{\delta'}^{\mathfrak p}{}^*/\ov E_{\delta'}^{\,\mathfrak p}$.

\smallskip
For instance, for $m = 41$ for which $\CT^{\mathfrak p}_1 \simeq
\Z/9\Z \times \Z/9\Z$, we find: 
\begin{equation}
\left\{\begin{aligned}
(\varepsilon_1^2 - 1) & = {\mathfrak a} \cdot {\mathfrak p}_1^3\,{\mathfrak p}_2^3 
\,{\mathfrak p}_3^4\, {\mathfrak p}_\0 = {\mathfrak a} \cdot ({\mathfrak p})
\cdot {\mathfrak p}_1^2\,{\mathfrak p}_2^2 \,{\mathfrak p}_3^3\, {\mathfrak p}_0, \\
(\varepsilon_2^2 - 1) & = {\mathfrak b} \cdot {\mathfrak p}_1^4\,
{\mathfrak p}_2^3\, {\mathfrak p}_3^3\, {\mathfrak p}_\0 =
{\mathfrak b} \cdot {\mathfrak p} \cdot {\mathfrak p}_1^3\,
{\mathfrak p}_2^2 \,{\mathfrak p}_3^2\, {\mathfrak p}_\0, 
\end{aligned}\right.
\end{equation}

\noindent
(with ${\mathfrak a}$, ${\mathfrak b}$ prime to $p$), justifying the
structure of $\CT^{\mathfrak p}_1$ since $\iota_{\mathfrak p}(\varepsilon_1)$ and
$\iota_{\mathfrak p}(\varepsilon_2)$ are in $(\CU_1^{\mathfrak p}{}^*)^9$
(see Lemma \ref{delta}).

\smallskip
Despite these congruences and the fact that $\varepsilon_1$ and $\varepsilon_2$
are relatively close to $1$ in $\CU_1^{\mathfrak p}{}^*$, the $3$-torsion group
$\CT_1 = \tor_{\Z_3} (\CU_1/\ov E_1)$ of $K_1$ is trivial since $K_1$ is by nature 
$3$-rational; roughly speaking, this is due to ${\mathfrak p}_\0$, totally ramified with
$v_{{\mathfrak p}_\0}(\varepsilon_i^2 - 1) = 1$ which can not give $u \in \CU_1$ such that 
$u^3 \in \ov E_1 \setminus \ov E_1^3$ (indeed, $u^3 \equiv 1 \pmod {{\mathfrak p}_\0^3}$).
\end{remark}

\subsection{Computations in the two first layers of 
\texorpdfstring{$k^\cyc$, $k = \Q(\sqrt{-23834})$}{Lg}}
\label{cyclo}

Since $k^\cyc$ can be exceptional, we can check some orders of magnitude, in the layers 
$k_1^\cyc$ and $k_2^\cyc$. We will compute, for $n = 1,2$, $\CH_n$, $\Ccl_n(S_n) := 
\langle \Ccl_n({\mathfrak P}_n), \Ccl_n(\ov{\mathfrak P}_n) \rangle$;
${\mathfrak P}_n \cdot \ov{\mathfrak P}_n$ is trivially principal in $\Q_n^\cyc$; 
thus, $\Ccl_n(S_n) = \langle \Ccl_n({\mathfrak P}_n) \rangle$.

\smallskip
Let's take, for $p=3$, the case of $k = \Q(\sqrt{-23834})$ for which $\CH_k = 1$ and
$\lambda(k^\cyc/k) = 10$ (given in Ozaki's paper as example). The following
obvious {\sc pari/gp} program can give results for the layers $k_1^\cyc$ and $k_2^\cyc$:
 
\ft\begin{verbatim}
{q=x^2+23834;k=bnfinit(q);for(n=1,2,Pcyc=polsubcyclo(3^(n+1),3^n);
Q=polcompositum(q,Pcyc)[1];Q=polredbest(Q);K=bnfinit(Q);print("Q=",Q);
print("Hn=",K.cyc," Clog=",bnflog(K,3));F=idealfactor(K,3);
F1=component(F,1);P1=F1[1];print("Cl(Pn)=",bnfisprincipal(K,P1)[1]))}

Q1=x^6+71496*x^4-2*x^3+1704178677*x^2+143010*x+13542540029547
Hn=[3132,36,4,2,2,2] 
Clog=[[81,9,3],[27],[9,9]] Cl(Pn)=[909,0,1,0,1,1]
Q2=x^18+214488*x^16+20447570079*x^14+1137143903665320*x^12+40655737576811246499*x^10
   -2*x^9+969070164048497993255928*x^8+1716066*x^7+15399817043842416811764049626*x^6
   -143160017418*x^5+157328931401635720374352785044160*x^4+2274931975284972*x^3
   +937641104284027054710612682252477593*x^2-5810155948876549410*x
   +2483707129129723239747767716046754925947
Hn=[2098451610324,108,36,6,6,6,6,6,2,2,2,2]
Clog=[[243,27,3,3,3,3,3,3,3],[27],[27,27,3,3,3,3,3,3]]
\end{verbatim}\ns

These data leads to the following results (the instruction ${\sf Cl(Pn)=
bnfisprincipal(K,P1)[1]}$ gives the relation $\Ccl_n({\mathfrak P}_n) =
\prod_{i} h_i^{c_i}$, where the $h_i$'s are the generating classes given 
by {\sc pari/gp}; so, $\CH_n^{S_1}$ may be deduced):
\ft\begin{equation*}
\left \{\begin{aligned}
\CH_k & = 1,\ \, \wt \CH_k \simeq \Z/27\Z, \ \,\delta_3(k) = 3, \\
\CH_1 & \simeq \Z/27\Z \times \Z/9\Z, \ \,
\wt \CH_1 \simeq \Z/81\Z \times \Z/9\Z \times \Z/3\Z, \\
\langle \Ccl_1(S_1) \rangle & = \langle \Ccl_1({\mathfrak P}_1) \rangle \simeq \Z/3\Z,\ \,
\CH_1^{S_1} \simeq \Z/9\Z \times \Z/9\Z, 
\end{aligned}\right.
\end{equation*}
\begin{equation*}
\ \ \ \,  \left \{\begin{aligned}
\CH_2 & \simeq \Z/81\Z \times \Z/27\Z \times \Z/9\Z \times (\Z/3\Z)^5, \\
\wt \CH_2 & \simeq \Z/243\Z \times \Z/27\Z  \times (\Z/3\Z)^7, \\
\langle \Ccl_2(S_2) \rangle & = \langle \Ccl_2({\mathfrak P}_2) \rangle \simeq \Z/9\Z,\ \,
\CH_2^{S_2} \simeq (\Z/27\Z)^2 \times (\Z/3\Z)^6. 
\end{aligned}\right.
\end{equation*}\ns

\medskip
Then, the fact that the $S$-Iwasawa invariant $\lambda^{S}(k^\cyc/k)$ 
is certainly non-trivial (due to a small $\Ccl_n(S_n)$), is coherent with a large 
$\lambda(k^\cyc/k)$. The data obtained go in the good direction.

\subsection{Computations of \texorpdfstring{$\langle \Ccl_1({\mathfrak P}), 
\Ccl_1(\ov{\mathfrak P}) \rangle$}{Lg} in the first layer of 
\texorpdfstring{$k^\acyc$}{Lg}}\label{acyclo}

From our computations of the first layer $k_1^\acyc$ of the anti-cyclotomic 
$\Z_3$-extension of an imaginary quadratic field \cite{Gra2026a}, 
we give examples with the following program giving ordinary
and logarithmic class groups, then the two components 
$\Ccl_1(S_1)^-$, $\Ccl_1(S_1)^+$, and the quotient $\CH_1^{S_1}$,
where $S_1 = \{{\mathfrak P}_1, \ov {\mathfrak P}_1\}$, in $k_1^\acyc$. 
A data of the form:

\ft\begin{verbatim}
Hacyc=[3,3]
Components of  Cl(P1/P2),Cl(P1*P2):[22,2],[0,0]
\end{verbatim}\ns

\noindent
means that, for $\CH_k = \langle h_1, h_2 \rangle \simeq \Z/3\Z \times \Z/3\Z$,
$\Ccl_1(S_1)^- = h_1^{22}h_2^2$ and $\Ccl_1(S_1)^+= h_1^0 h_2^0 = 1$,
whence $\Ccl_1(S_1)^-$ of order $3$ and $\Ccl_1(S_1)^+$ of order $1$.

\smallskip
\ft\begin{verbatim}
{p=3;for(m=5,10^4,if(core(m)!=m,next);if(Mod(-m,3)!=1,next);
P=x^2+m;k=bnfinit(P);hk=k.no;if(Mod(hk,3)==0,next);
\\ COMPUTATION OF delta:
Clog=bnflog(k,3);if(Clog[1]==[],next);dlog=matsize(Clog)[1];
delta=0;for(j=1,dlog,c=Clog[1][j];delta=delta+valuation(c,3));
print();print("m=",m," h_k=",hk," Clogk=",Clog," delta=",delta);
\\ COMPUTATION OF THE FUNDAMENTAL UNIT OF k* AND k_1^ac:
Pstar=x^2-3*m;kstar=bnfinit(Pstar,1);w=kstar.fu[1];
Q=x^3-3*x-trace(w);R=polcompositum(P,Q)[1];
R=polredbest(R);print("Racyc=",R);kacyc=bnfinit(R,1);
\\ COMPUTATION OF THE INVARIANTS OF k_1^ac:
Clogkacyc=bnflog(kacyc,3);
print("Clogkacyc=",Clogkacyc," Classkacyc=",kacyc.cyc);
hacyc=kacyc.no;hacyctame=hacyc/3^valuation(hacyc,3);
Hacyc=kacyc.cyc;d=matsize(Hacyc)[2];r=d;
H=List;for(j=1,d,c=Hacyc[j];v=valuation(c,3);if(v!=0,listput(H,3^v)));
\\ COMPUTATION OF THE CLASSES OF P1 AND P2:
Sideal1=component(component(idealfactor(kacyc,3),1),1);
Sideal1=idealpow(kacyc,Sideal1,hacyctame);
Sideal2=component(component(idealfactor(kacyc,3),1),2);
Sideal2=idealpow(kacyc,Sideal2,hacyctame);
Sminus=idealdiv(kacyc,Sideal1,Sideal2);Splus=idealmul(kacyc,Sideal1,Sideal2);
Pminus=bnfisprincipal(kacyc,Sminus);Pplus=bnfisprincipal(kacyc,Splus);
\\ COMPONENTS OF Cl(P)- AND Cl(P)+:
print("Hacyc=",H);print("Components of Cl(P1/P2), Cl(P1*P2):
",Pminus[1],", ",Pplus[1]))}

m=1370 h_k=44 Clogk=[[9],[9],[]] delta=2 Racyc=x^6+6*x^4-368*x^3+12339*x^2-9324*x+35226
Clogkacyc=[[9],[9],[3]] Classkacyc=[66,6]  Hacyc=[3,3]
Components of Cl(P1/P2), Cl(P1*P2):[22,2],[0,0]  (orders: 3, 1)

m=1379 h_k=16 Clogk=[[3],[3],[]] delta=1 Racyc=x^6+66*x^4+1089*x^2+5516
Clogkacyc=[[3],[],[3,3]] Classkacyc=[144,3]  Hacyc=[9,3]
Components of Cl(P1/P2), Cl(P1*P2):[0,0],[96,0]  (orders: 1, 3)

m=1451 h_k=13 Clogk=[[3],[3],[]] delta=1 Racyc=x^6-6*x^4+9*x^2+5804
Clogkacyc=[[3],[3],[3]] Classkacyc=[39,3]  Hacyc=[3,3]
Components of Cl(P1/P2),Cl(P1*P2):[26,2],[0,0]  (orders: 3, 1)

m=2762 h_k=26 Clogk=[[243],[243],[]] delta=5 Racyc=x^6-150*x^4+5625*x^2+1867112
Clogkacyc=[[243],[243],[3]] Classkacyc=[78,3]  Hacyc=[3,3]
Components of Cl(P1/P2), Cl(P1*P2):[0,1],[0,0]  (orders: 3, 1)

m=4865 h_k=56 Clogk=[[3],[3],[]] delta=1 Racyc=x^6+468*x^4+54756*x^2+6305040
Clogkacyc=[[3],[],[3,3]] Classkacyc=[126,6,2]  Hacyc=[9,3]
Components of Cl(P1/P2), Cl(P1*P2):[0,0,0],[84,0,0]  (orders: 1, 3)

m=30605 h_k=136 Clogk=[[3],[3],[]] delta=1 Racyc=x^6+12*x^4+36*x^2+489680
Clogkacyc=[[81,27],[3],[81,27]] Classkacyc=[5508,162]  Hacyc=[81,81]
Components of Cl(P1/P2), Cl(P1*P2):[0,108],[0,0]  (orders: 3, 1)

m=57830 h_k=284 Clogk=[[243],[243],[]] delta=5 Racyc=x^6-132*x^4+4356*x^2+925280
Clogkacyc=[[243,27,27],[243],[81,27]] Classkacyc=[11502,162]  Hacyc=[81,81]
Components of Cl(P1/P2), Cl(P1*P2):[0,54],[0,0]  (orders: 3, 1)

m=64790 h_k=304 Clogk=[[3],[3],[]] delta=1 Racyc=x^6+210*x^4+11025*x^2+6479000
Clogkacyc=[[81,27],[3],[81,27]] Classkacyc=[3078,162,2,2]  Hacyc=[81,81]
Components of Cl(P1/P2), Cl(P1*P2):[1026,0,0,0],[0,0,0,0]  (orders: 3, 1)
\end{verbatim}\ns

This alternation of cases does not occur for $k^\cyc$ since $\CH_{k_n^\cyc}^+$ 
is always trivial.

\smallskip
Let's give the case of $k_1^\acyc$ when $k$ is not assumed $3$-principal
and when $k^\acyc/k$ is totally ramified (${\sf hk3}$ is the order of $\CH_k$):

\ft\begin{verbatim}
m=107  hk3=3  Hacyc=[9,3]      Cl(P1/P2),Cl(P1*P2) : [2,0], [0,0]
m=302  hk3=3  Hacyc=[12,3]     Cl(P1/P2),Cl(P1*P2) : [0,1], [8,2]
m=503  hk3=3  Hacyc=[63,3]     Cl(P1/P2),Cl(P1*P2) : [14,0], [0,0]
m=509  hk3=3  Hacyc=[90,3]     Cl(P1/P2),Cl(P1*P2) : [20,0], [0,0]
m=533  hk3=3  Hacyc=[18,18,3]  Cl(P1/P2),Cl(P1*P2) : [2,8,0], [0,0,0]
m=602  hk3=3  Hacyc=[84,42]    Cl(P1/P2),Cl(P1*P2) : [28,0], [0,28]
m=617  hk3=3  Hacyc=[12,6,2]   Cl(P1/P2),Cl(P1*P2) : [4,4,0], [4,2,0]
m=713  hk3=3  Hacyc=[12,6]     Cl(P1/P2),Cl(P1*P2) : [8,4], [4,0]
m=863  hk3=3  Hacyc=[21,3]     Cl(P1/P2),Cl(P1*P2) : [7,1], [0,2]
m=1007 hk3=3  Hacyc=[30,15,5]  Cl(P1/P2),Cl(P1*P2) : [10,10,0], [0,10,0]
m=1046 hk3=3  Hacyc=[42,3]     Cl(P1/P2),Cl(P1*P2) : [0,1], [28,2]
m=1202 hk3=3  Hacyc=[72,3,3]   Cl(P1/P2),Cl(P1*P2) : [0,2,0], [48,0,0]
m=1319 hk3=9  Hacyc=[45,3]     Cl(P1/P2),Cl(P1*P2) : [10,2], [0,1]
m=1337 hk3=3  Hacyc=[36,6]     Cl(P1/P2),Cl(P1*P2) : [16,0], [0,0]
m=1355 hk3=3  Hacyc=[36,6,2]   Cl(P1/P2),Cl(P1*P2) : [32,0,0], [0,0,0]
m=1427 hk3=3  Hacyc=[15,3]     Cl(P1/P2),Cl(P1*P2) : [10,1], [10,0]
m=1454 hk3=3  Hacyc=[420,21]   Cl(P1/P2),Cl(P1*P2) : [280,0], [0,14]
m=1478 hk3=3  Hacyc=[30,3]     Cl(P1/P2),Cl(P1*P2) : [0,1], [20,1]
m=1547 hk3=3  Hacyc=[6,6]      Cl(P1/P2),Cl(P1*P2) : [2,2], [0,4]
m=1583 hk3=3  Hacyc=[33,3]     Cl(P1/P2),Cl(P1*P2) : [22,0], [0,1]
m=1619 hk3=3  Hacyc=[60,12,2,2]Cl(P1/P2),Cl(P1*P2) : [40,0,0,0], [0,8,0,0]
m=1865 hk3=3  Hacyc=[24,6]     Cl(P1/P2),Cl(P1*P2) : [8,2], [0,4]
m=1871 hk3=9  Hacyc=[90,6]     Cl(P1/P2),Cl(P1*P2) : [20,4], [60,2]
m=1895 hk3=3  Hacyc=[48,3]     Cl(P1/P2),Cl(P1*P2) : [16,1], [0,2]
m=1997 hk3=3  Hacyc=[42,6,2]   Cl(P1/P2),Cl(P1*P2) : [28,2,0], [0,4,0]
m=2081 hk3=3  Hacyc=[180,3,3]  Cl(P1/P2),Cl(P1*P2) : [120,0,0], [120,2,0]
m=2138 hk3=3  Hacyc=[126,3,3]  Cl(P1/P2),Cl(P1*P2) : [84,0,0], [42,2,0]
m=2141 hk3=3  Hacyc=[78,3]     Cl(P1/P2),Cl(P1*P2) : [0,2], [52,1]
m=2183 hk3=3  Hacyc=[42,6,6,3] Cl(P1/P2),Cl(P1*P2) : [0,4,4,0], [0,2,0,0]
m=2186 hk3=3  Hacyc=[42,3,3,3] Cl(P1/P2),Cl(P1*P2) : [28,2,2,0], [14,2,1,0]
m=2213 hk3=3  Hacyc=[42,3]     Cl(P1/P2),Cl(P1*P2) : [0,1], [28,2]
m=2273 hk3=3  Hacyc=[48,3]     Cl(P1/P2),Cl(P1*P2) : [16,2], [32,2]
m=2309 hk3=3  Hacyc=[198,3]    Cl(P1/P2),Cl(P1*P2) : [154,0], [0,0]
m=2351 hk3=9  Hacyc=[189,3]    Cl(P1/P2),Cl(P1*P2) : [28,1], [0,0]
m=2414 hk3=9  Hacyc=[36,6,2,2] Cl(P1/P2),Cl(P1*P2) : [28,4,0,0], [0,2,0,0]
m=2555 hk3=3  Hacyc=[6,6]      Cl(P1/P2),Cl(P1*P2) : [2,2], [0,4]
m=2627 hk3=3  Hacyc=[12,3]     Cl(P1/P2),Cl(P1*P2) : [0,1], [8,0]
m=2801 hk3=3  Hacyc=[60,3]     Cl(P1/P2),Cl(P1*P2) : [20,1], [20,2]
m=2822 hk3=3  Hacyc=[24,6]     Cl(P1/P2),Cl(P1*P2) : [0,4], [16,0]
m=2879 hk3=3  Hacyc=[57,3]     Cl(P1/P2),Cl(P1*P2) : [38,1], [19,1]
m=2915 hk3=3  Hacyc=[108,18,3] Cl(P1/P2),Cl(P1*P2) : [0,0,2], [72,0,0]
m=3206 hk3=3  Hacyc=[90,30,15] Cl(P1/P2),Cl(P1*P2) : [30,0,0], [30,10,0]
m=3359 hk3=3  Hacyc=[69,3]     Cl(P1/P2),Cl(P1*P2) : [23,1], [23,2]
m=3614 hk3=3  Hacyc=[30,6,6,6] Cl(P1/P2),Cl(P1*P2) : [0,4,2,0], [10,4,0,0]
m=3617 hk3=3  Hacyc=[60,3]     Cl(P1/P2),Cl(P1*P2) : [0,1], [40,2]
m=3647 hk3=27 Hacyc=[54,3]     Cl(P1/P2),Cl(P1*P2) : [26,2], [18,2]
m=3653 hk3=3  Hacyc=[30,6]     Cl(P1/P2),Cl(P1*P2) : [20,4], [20,0]
m=3665 hk3=3  Hacyc=[24,6,2,2] Cl(P1/P2),Cl(P1*P2) : [16,4,0,0], [0,4,0,0]
m=3695 hk3=9  Hacyc=[72,3]     Cl(P1/P2),Cl(P1*P2) : [64,1], [0,2]
m=3899 hk3=3  Hacyc=[24,3]     Cl(P1/P2),Cl(P1*P2) : [16,1], [0,2]
m=3917 hk3=3  Hacyc=[42,6,2,2,2]Cl(P1/P2),Cl(P1*P2): [28,4,0,0,0], [0,2,0,0,0]
m=4049 hk3=9  Hacyc=[72,6,2]   Cl(P1/P2),Cl(P1*P2) : [40,0,0], [0,4,0]
m=4103 hk3=3  Hacyc=[42,3]     Cl(P1/P2),Cl(P1*P2) : [14,2], [0,2]
m=4154 hk3=3  Hacyc=[48,6,2,2] Cl(P1/P2),Cl(P1*P2) : [16,4,0,0], [0,2,0,0]
m=4181 hk3=3  Hacyc=[30,6]     Cl(P1/P2),Cl(P1*P2) : [20,2], [20,4]
m=4229 hk3=3  Hacyc=[234,3]    Cl(P1/P2),Cl(P1*P2) : [52,0], [0,0]
m=4283 hk3=3  Hacyc=[63,3]     Cl(P1/P2),Cl(P1*P2) : [49,0], [0,0]
m=4415 hk3=3  Hacyc=[66,3]     Cl(P1/P2),Cl(P1*P2) : [22,2], [22,1]
m=4457 hk3=3  Hacyc=[60,3]     Cl(P1/P2),Cl(P1*P2) : [20,1], [40,1]
m=4478 hk3=3  Hacyc=[144,12,4] Cl(P1/P2),Cl(P1*P2) : [64,0,0], [0,0,0]
m=4574 hk3=3  Hacyc=[288,3,3]  Cl(P1/P2),Cl(P1*P2) : [96,0,0], [192,2,1]
m=4589 hk3=3  Hacyc=[42,6]     Cl(P1/P2),Cl(P1*P2) : [28,2], [0,2]
m=4595 hk3=3  Hacyc=[24,3]     Cl(P1/P2),Cl(P1*P2) : [16,2], [0,1]
m=4637 hk3=3  Hacyc=[234,3,3]  Cl(P1/P2),Cl(P1*P2) : [0,0,0], [0,1,0]
m=4790 hk3=3  Hacyc=[72,6,6,2] Cl(P1/P2),Cl(P1*P2) : [0,4,4,0], [48,0,0,0]
m=4799 hk3=9  Hacyc=[315,45,3] Cl(P1/P2),Cl(P1*P2) : [280,20,2], [105,0,0]
m=4826 hk3=3  Hacyc=[60,6]     Cl(P1/P2),Cl(P1*P2) : [20,4], [40,0]
m=4835 hk3=3  Hacyc=[270,9,3]  Cl(P1/P2),Cl(P1*P2) : [180,0,0], [90,0,1]
\end{verbatim}\ns

\subsection{The \texorpdfstring{$3$}{Lg}-class groups of the layers 
\texorpdfstring{$k_1^\cyc$ and $k_2^\cyc$}{Lg}}
The following program gives, for the $3$-principal imaginary quadratic fields
$k = \Q(\sqrt{-m})$, $-m \equiv 1 \pmod 3$, the number $\delta = \delta_3(k)$ and
the structure of the $3$-class groups of the layers $k_1^\cyc$ and $k_2^\cyc$
of the cyclotomic $\Z_3$-extension. This can illustrate the equivalence: $\delta = 0$
if and only if the Iwasawa invariants of $k^\cyc$ are minimal; indeed, if some regularities
are observed many structures are possible as soon as $\delta \ne 0$.
Since $\delta = 0$ implies $\CH_n \simeq \Z/3^n\Z$, we only print the 
non-trivial cases $\delta \geq 1$ and do not repeat identical structures:

\ft\begin{verbatim}
{p=3;for(m=2,10^4,if(core(m)!=m,next);if(Mod(-m,3)!=1,next);P=x^2+m;
k=bnfinit(P);hk=k.no;if(Mod(hk,3)==0,next);Clog=bnflog(k,3);if(Clog[1]==[],next);
dlog=matsize(Clog)[1];delta=0;for(j=1,dlog,c=Clog[2][j];delta=delta+valuation(c,3));
Pcyc=polcompositum(P,polsubcyclo(9,3))[1];kcyc=bnfinit(Pcyc,1);
H=kcyc.cyc;d=matsize(H)[2];Hcyc1=List;for(j=1,d,c=H[j];v=valuation(c,3);
if(v!=0,listput(Hcyc1,3^v)));Pcyc=polcompositum(P,polsubcyclo(27,9))[1];
kcyc=bnfinit(Pcyc,1);H=kcyc.cyc;d=matsize(H)[2];Hcyc2=List;
for(j=1,d,c=H[j];v=valuation(c,3);if(v!=0,listput(Hcyc2,3^v)));
print();print("m=",m," Clogk=",Clog);
print("delta=",delta," Hcyc1=",Hcyc1," Hcyc2=",Hcyc2))}

m=14 Clogk=[[3],[3],[]]     delta=1  Hcyc1=[3,3]  Hcyc2=[9,9]
m=41 Clogk=[[27],[27],[]]   delta=3  Hcyc1=[9,3]  Hcyc2=[27,9,3]
m=47 Clogk=[[9],[9],[]]     delta=2  Hcyc1=[3,3]  Hcyc2=[9,9]
m=86 Clogk=[[3],[3],[]]     delta=1  Hcyc1=[9,9]  Hcyc2=[27,27,3]
m=311 Clogk=[[3],[3],[]]    delta=1  Hcyc1=[9,3]  Hcyc2=[27,9,3,3]
m=323 Clogk=[[27],[27],[]]  delta=3  Hcyc1=[3,3]  Hcyc2=[9,9]
m=365 Clogk=[[81],[81],[]]  delta=4  Hcyc1=[9,9]  Hcyc2=[27,27,3,3]
m=458 Clogk=[[3],[3],[]]    delta=1  Hcyc1=[9,3]  Hcyc2=[27,9,3,3,3,3,3]
m=482 Clogk=[[3],[3],[]]    delta=1  Hcyc1=[9,3]  Hcyc2=[27,9,3]
m=641 Clogk=[[9],[9],[]]    delta=2  Hcyc1=[9,3]  Hcyc2=[27,9,3]
m=659 Clogk=[[3],[3],[]]    delta=1  Hcyc1=[27,9] Hcyc2=[81,27,3]
m=683 Clogk=[[3],[3],[]]    delta=1  Hcyc1=[9,9]  Hcyc2=[27,27,3]
m=806 Clogk=[[3],[3],[]]    delta=1  Hcyc1=[27,27]Hcyc2=[81,81,3]
m=1022 Clogk=[[3],[3],[]]   delta=1  Hcyc1=[9,3]  Hcyc2=[27,9,3,3,3,3]
\end{verbatim}\ns

\end{appendices}


\section*{Notations}\label{?}

We organize the notations regarding various  
categories of objects; $p \geq 3$ is a prime number split in $k$ into 
${\mathfrak p}\cdot {\ov {\mathfrak p}}$, with ${\ov {\mathfrak p}} 
:= {\mathfrak p}^\tau$ for the complex conjugation $\tau$.
Then $x$, $\ov x$ are the fundamental ${\mathfrak p}$
and $\ov {\mathfrak p}$-units, respectively.
In relative extensions $K_b/K_a$, $0 \leq a \leq b$, any writing 
of the form $O_{b/a}$ is relatve to $K_b/K_a$ (maps, Galois groups, etc.):

\smallskip
{\it Number fields}
\begin{itemize}
\item $k$: imaginary quadratic field,
\item $k^\cyc$: cyclotomic $\Z_p$-extension of $k$,
\item $k^\acyc$: anti-cyclotomic $\Z_p$-extension of $k$,
\item $\wt k$: compositum of the $\Z_p$-extensions of $k$,
\item $L, \ov L$: inertia fields of ${\mathfrak p}$, ${\ov {\mathfrak p}}$, 
in $\wt k/k$,
\item $K$: non-Galois $\Z_p$-extension of $k$, $K \ne L, \ov L$,
\item $K_n$: $n^{\rm th}$ layer of the $\Z_p$-extension $K$,
\item $K_e = K \cap L$: inertia field of ${\mathfrak p}$ in $K/k$; 
$K \cap \ov L = k$ ($\ov e = 0$),
\item $H_k^\nr$: $p$-Hilbert's class field of $k$,
\item $H_k^\pr$: maximal abelian $p$-ramified pro-$p$ extension of $k$,
\end{itemize}

{\it Units, $S$-units}
\begin{itemize}
\item $S_F, S \subseteq S_F$: sets of $p$-places of a number field $F$,
\item $E_F$: group of global units of $F$, $\ E_F^{S}$: group of global $S$-units, 
\item $X_F^S$: set of fundamental $S$-units, 
\item $\CU_F^v$: group of principal local units of the completion $F_v$,
\item $\CU_F^{\mathfrak p} = \hbox{$\bigoplus_{v \mid {\mathfrak p}}$} \CU_F^v$,
$\ \CU_F^S = \hbox{$\bigoplus_{v \in S}$} \CU_F^v$,
\item $\CU_F^{\mathfrak p}{}^*$: sub-module of elements 
$u \in \CU_F^{\mathfrak p}$ of norm $1$,
\item $\iota_S$: diagonal embeddings in $\CU_F^S$,
\item $\ov E_F^S$: topological closure of $\iota_S(E_F)$ in $\CU_F^S$,
\end{itemize}

\smallskip
{\it Class groups, $S$-class groups, Galois groups}
\begin{itemize}
\item $\CH_F$: $p$-class group of $F$, $\ \CH_F^{S}$ : $S$-class group of $F$, \, 
\item $\CT_k$: torsion $p$-group of $\Gal(H_k^\pr/ k)$,
\item $\CT^{\mathfrak p}_{\delta'}$: torsion $p$-group of 
$\Gal(H^{{\mathfrak p}\,\pr}_{\delta'}/K_{\delta'})$, in $F = K_{\delta'}$,
\item $\CR^{\mathfrak p}_{\delta'} = 
\CU_{\delta'}^{\mathfrak p}{}^*/\ov E_{\delta'}^{\,\mathfrak p}$, in $F = K_{\delta'}$,
\item $g_{b/a} = \Gal(K_b/K_a)$, $G_n = g_{n/0}$,
\item $G_K = \Gal(\wt k/K) = \langle \wt\sigma_K \rangle$
\end{itemize}

\smallskip
{\it Canonical maps}
\begin{itemize}
\item $\BN_{M/F}$: arithmetic norm in $M/F$, 
\item $\log_{\mathfrak p}$, $\log_{\ov {\mathfrak p}}$: $p$-adic
logarithms defined on the completions $k_{\mathfrak p}^\times$, 
$k_{\ov {\mathfrak p}}^\times$ of $k$,
\item $v_p$: $p$-adic valuation, $v_{\mathfrak p}$: ${\mathfrak p}$-adic valuation, 
\end{itemize}

\smallskip
{\it Class numbers and numerical invariants}
\begin{itemize}
\item $\hk$: class number of $k$,
\item $\hp \mid \hk$: order of the class of ${\mathfrak p} \mid p$ in $k$,
\item $\he$: class number of $K_e$ ($p$-principal),
\item $\delta =: v_{\mathfrak p}(\ov x^{\,p-1}-1) - 1$, $\ \delta' =: \min(e, \delta)$,
\item $\spp = p^{\delta'}$, $\re = \spp - 1$,
\end{itemize}

\smallskip
{\it Symbols}
\begin{itemize} 
\item $\big(\ffrac{\mb\,,\,F/k}{\mathfrak q}\big)$: norm residue symbol 
at ${\mathfrak q}$,
\item $\Omega_{n/e} := \big\{(s_i)_i \in \bigoplus_{i=0}^\spp \, g_{n/e},\ 
\prod_{i=1}^\spp s_i= 1 \big\}$, 
\item $\wt \Omega_{n/e}^{\,\mathfrak p} := \big\{(s_i)_i 
\in \bigoplus_{i=1}^\spp \, g_{n/e}\big\} = (g_{n/e})^\spp$,
\item $\omega_{n/e} := \omega_{K_n/K_e} : k^\times \to \wt \Omega_{n/e}^{\,\mathfrak p}$.
\end{itemize}

\end{document}